\documentclass[12pt,a4paper,oneside]{article}
\usepackage[top=2.5cm, bottom=2.5cm, left=1.3cm, right=1.3cm]{geometry}

\usepackage[english]{babel}
\usepackage[utf8]{inputenc}
\usepackage[T1]{fontenc}

\usepackage{amsthm,amsmath,amssymb,mathrsfs,dsfont,mathtools}
\usepackage{bbm}
\usepackage{bm}
\usepackage{graphicx}
\usepackage[colorlinks=true, allcolors=blue]{hyperref}
\usepackage{comment}
\usepackage[all]{nowidow}
\usepackage{booktabs}
\usepackage{cleveref}
\usepackage{makecell}
\usepackage{xcolor}
\usepackage{xparse}

\usepackage[authoryear, round]{natbib}
\usepackage{setspace}
\newcommand{\dequal}{\overset{\rm d}{=}}
\newcommand{\iid}{\stackrel{\mbox{\scriptsize iid}}{\sim}}
\newcommand{\ind}{\stackrel{\mbox{\scriptsize ind}}{\sim}}
\newcommand{\indicator}{\ensuremath{\mathbbm{1}}}

\newcommand{\calN}{\mathcal{N}}

\newcommand{\dd}{\mathrm d}

\newcommand{\plaw}{\mathbf{P}}

\newcommand{\Pp}{\mathsf{P}}
\newcommand{\X}{\mathbb{X}}

\newcommand{\M}{\mathbb{M}}
\newcommand{\N}{\mathbb{N}}

\newcommand{\E}{\mathsf{E}}
\newcommand{\R}{\mathbb{R}}

\newcommand{\prob}{\mathsf{P}}

\newcommand{\Acr}{\mathscr{A}}

\newcommand{\Xcr}{\mathscr{X}}

\definecolor{OliveGreen}{RGB}{85,107,47}

\NewDocumentCommand{\evalat}{sO{\big}mm}{%
  \IfBooleanTF{#1}
   {\mleft. #3 \mright|_{#4}}
   {#3#2|_{#4}}%
}

\renewcommand{\mid}{\ensuremath{\,|\,}}

\usepackage{authblk}
\date{}
\title{Extended feature allocation models}
\author[1]{Mario Beraha}
\author[1]{Federico Camerlenghi}
\author[2]{Lorenzo Ghilotti}

\affil[1]{University of Milano--Bicocca, Department of Economics, Management, and Statistics, Piazza dell'Ateneo Nuovo 1, 20126 Milano, Italy}
\affil[2]{Duke University, Department of Statistical Science, 218 Old Chemistry, Durham, North Carolina 27708, U.S.A.}

\affil[ ]{\begin{footnotesize} mario.beraha@unimib.it, federico.camerlenghi@unimib.it (corresponding author), lorenzo.ghilotti@duke.edu \end{footnotesize}}

\providecommand{\keywords}[1]{
  \small 
  \textbf{\textit{Keywords:}} #1
  \normalsize
}

\usepackage[sf]{titlesec}
 
\usepackage{sectsty}
\allsectionsfont{\centering \normalsize \normalfont\scshape} 

\usepackage{caption}
\newtheorem{theorem}{Theorem}
\newtheorem{corollary}{Corollary}
\newtheorem{lemma}{Lemma}
\newtheorem{proposition}{Proposition}
\newtheorem{remark}{Remark}

\theoremstyle{definition}

\theoremstyle{remark}

\newtheorem{example}{Example}

\begin{document}

\maketitle

\begin{abstract}
Feature allocation models are Bayesian nonparametric tools tailored to  data in which each observation can simultaneously exhibit multiple characteristics, or features. 
A fundamental limitation of standard formulations is that feature labels are assumed to be independent and identically distributed, and therefore play no role in posterior inference.
The present paper introduces a unified Bayesian framework for \emph{extended feature allocation models}, in which feature labels and proportions are modeled jointly, thereby enabling the simultaneous discovery of features and learning of dependencies among their labels. Building on point process theory, we develop a full Bayesian analysis of these models.
Within this general setting, we also characterize previously proposed priors as those leading to poor predictive distributions, which cannot capture label dependencies and are insensitive to the observed frequency spectrum. Our methodology is designed to move beyond such standard formulations by leveraging the information carried by feature labels. We demonstrate the usefulness of our approach by introducing: (i) a Cox process prior that clusters genomic variant embeddings while predicting new variants and new variant clusters; (ii) a determinantal point process prior for repeated forest surveys, where prediction concerns both the number and the locations of unobserved trees.
\end{abstract}


\section{Introduction}

\subsection{Background and motivation}

Since the introduction of the celebrated Indian buffet process (\textsc{ibp}) by \citet{Gha05}, feature allocation models \citep{Bro(13)} have become a cornerstone of Bayesian nonparametrics, with applications spanning image segmentation \citep{Tit(07), Gri(11), Hu(12), Bro(14)}, text mining \citep{Thi(07), Teh09}, matrix factorization \citep{Zho(12), zhou2016priors}, network analysis \citep{Mil(09), ayed2021nonnegative}, ecology \citep{dunsonstolf24}, and microbiome studies \citep{james2025poisson}.
In these models, each observation is represented by a finite set of ``characteristics'' or ``features'', which may occur with positive probability across different observations. We refer to the total occurrences of a feature as its frequency.
The connection between the \textsc{ibp} and beta-process priors in \citet{Thi(07)} paved the way to many generalizations, including priors that induce power-laws \citep{Teh09, Bro(13)} and general likelihoods \citep{Zho(12), zhou2016priors, Hea(16)}. See \cite{broderick_exponential} and \cite{Jam(17)} for comprehensive treatments under completely random measure (\textsc{crm}) priors.

A persistent limitation is that most feature allocation models treat the feature labels as essentially irrelevant, being  independent and identically distributed (i.i.d.) from a diffuse distribution. As a consequence, prior information can influence feature frequencies, but not dependence structures among the labels themselves. In many applications this is too restrictive. For instance, genomic variants may come with embeddings that encode functional or evolutionary similarity \citep[as suggested in][]{niu2025incorporating}; spatial locations may exhibit inhibition or repulsion; and latent features in image segmentation or factor models often benefit from regularizing dependence, as in repulsive mixture models \citep[see, e.g.,][and the references therein]{beraha2024}. These settings require models that can learn both from the  allocation of features across observations and from dependence among feature labels. To this end, we use  point processes to jointly model  feature labels and feature probabilities, allowing dependence among labels to directly inform posterior and predictive inference.

Prediction problems provide a natural setting in which the limitations of standard feature allocation models
become particularly apparent.
A canonical example is the \emph{unseen feature problem}: given an observed sample of size $n$ displaying $K_n$ distinct features,  the aim is to predict the number of new features appearing in an additional sample of size \(m\).
This task is a natural generalization of the classical \textit{unseen species problem} \citep[see, e.g.][for contributions both from the frequentist and Bayesian standpoint]{good1956number, orlitsky2016optimal, lijoi2007bayesian, favaro2012}. Applications of feature allocations model to the unseen feature problem can be found in genomic studies \citep{Mas(22), Cam(23), She(22)}, user activity prediction in online A/B testing \citep{beraha2025online}, and ecology \citep{ghilotti2024bayesian}. 
\citet{Jam(17)} showed that under \textsc{crm} priors the probability of discovering new features depends solely on the sample size $n$, making the predictive law insensitive to other relevant sampling information. This can be restrictive in prediction problems involving feature discovery  when the goal is not only to predict how many new features will appear, but also where they lie, which clusters they belong to, or how they relate to previously observed features. More generally, these examples highlight that the predictive rule induced by a prior is a key consideration in model specification.
This motivates the study of characterizations of feature allocation priors in terms on  how the observed sampling statistics enter the prediction rule, providing a principled basis for prior elicitation. 
In other words, one aims to identify sufficientness postulates in the spirit of those developed for species sampling models \citep{Pit96, Zab82}, which have led to characterizations of the Dirichlet process \citep{Reg(78), Lo(91)}, the Pitman--Yor process \citep{Pit(97), Zab05}, and Gibbs-type priors \citep{DeBlasi, Bac(17)}. 
By contrast, analogous predictive characterizations for feature allocation models remain limited, with the only contributions of \citet{Cam(22), Cam(23)}, who developed such characterizations within the class of scaled process priors \citep{Jam(15)}.

\subsection{Our contributions}

The present paper offers a unified Bayesian framework for \emph{extended feature allocation models}, in which feature labels and feature probabilities are modeled jointly. Relaxing the i.i.d. assumption on labels allows the model to capture interactions among features---such as repulsion and attraction---as well as arbitrary dependence structures among feature probabilities. More importantly, it expands the scope of inference beyond feature discovery, the primary objective of standard feature allocation models, to include learning dependence structures among labels and exploiting such information for prediction. To make this framework operational, we derive explicit expressions for the marginal, posterior, and predictive distributions of extended feature allocation models. These results provide the key ingredients for Bayesian posterior analysis when feature labels are themselves objects of inference.
Our theory holds without imposing specific assumptions on the prior distribution, requiring only a modest increase in analytical complexity compared with standard feature methods, such as those in \citet{Jam(17), broderick_exponential, ghilotti2024bayesian}, which are recovered as special cases.
Our analysis and proofs rely heavily on point process theory and Palm calculus, pioneered in Bayesian nonparametrics by \cite{Jam(02), Jam(05), Jam(06), Jam(17)}.

Within this general framework, we establish two sufficientness postulates that characterize the priors whose predictive law for unseen features depends only on the sample size, or only on the sample size and the number of distinct features. These results are the feature allocation analogue of the sufficientness postulates available for species sampling models. All standard feature allocation models, in which labels are irrelevant, fall within the classes characterized by these postulates. Our characterizations serve a dual purpose. On the one hand, they offer practical guidance for prior elicitation in  standard feature allocation settings, i.e., when the labels are irrelevant for inferential tasks and the unseen feature problem is the natural predictive target \citep{Cam(22), Cam(23)}. On the other hand, they also formally identify the limitations of standard models: any prior aiming to leverage informative labels in prediction must lie outside the classes characterized here.

Moving beyond these characterizations, we focus on the more compelling setting in which feature labels themselves carry relevant information for posterior inference. To this end, we develop two novel classes of models that capture qualitatively different forms of dependence among labels. We first introduce a prior based on cluster Cox processes to encode clustering among labels, and apply it to genomic variant data from patients affected by Glioblastoma, where each variant is associated with an embedding as in \cite{niu2025incorporating}. In this application, each patient is represented by the genomic variants it displays, which are the features, and each variant is associated with a high-dimensional Euclidean vector (the embedding) that encodes external information about the variant. We demonstrate how the proposed framework enables simultaneous clustering of variants and prediction of unseen ones, both globally and within clusters. We then introduce a prior based on determinantal point processes \citep[\textsc{dpp}s;][]{Hou(06),Lav(15)} to encode repulsion among labels, and apply it to a spatial setting where features represent tree locations in a forest. This model can recover both the total number of trees and the locations of unobserved ones, thereby addressing a classical problem posed by \citet{ord78}.
Although the paper focuses on settings in which both feature-sharing patterns and feature labels are observed, these ideas have clear potential in problems where features are latent, such as image segmentation and latent feature modeling. In such contexts, incorporating repulsion in the prior over features may yield more realistic and interpretable inference, similarly to what has been observed for repulsive mixture models \citep[see, e.g.,][and the references therein]{PeRaDu12, xie2019bayesian, cremaschi2023repulsion, beraha2024}.

\subsection{Outline of the paper}

The paper is organized as follows. In Section \ref{sec:prevwork}, we define the framework of \emph{extended feature models} and review key notions of  point process theory.
In Section \ref{sec:bnp}, we develop a general Bayesian theory for extended feature allocation models, without imposing assumptions on the prior distribution. Section \ref{sec:predictive_characterization} presents the two sufficientness postulates. Section \ref{sec:sncp_model_and_application} introduces a novel class of models based on cluster Cox processes to address both clustering of labels and feature discovery, and explores its use in genomics. Section \ref{sec:dpp_model_and_application} introduces a second class of models based on determinantal point processes that encode repulsion among feature labels, and illustrates its use in spatial statistics.  The paper ends with a discussion. Proofs, additional results, and further details on the applications are collected in the supplementary material. The code is available online at the link: \url{https://github.com/LGhilotti/ExtendedFeatureModels}.

\section{Extended feature allocation models}\label{sec:prevwork}

\subsection{Background on point processes}\label{sec:back_point_processes}

We review here the key basic definitions on point processes needed in  the rest of the paper, adopting notation from \cite{BaBlaKa}. See Appendix \ref{app:point_process_definition} for further mathematical details.

Let $\X$ be a Polish space equipped with the corresponding Borel $\sigma$-algebra $\Xcr$, a point process $\Phi$ on $\X$ can be represented as $\Phi = \sum_{j\geq 1} \delta_{X_j}$, where $(X_j)_{j\geq 1}$ is a sequence of random variables taking values in $\X$, and  $\delta_{X_j}$ denotes the Dirac delta mass at $X_j$.
Examples of point processes previously considered in the context of feature models encompass the Poisson, mixed Poisson, and the mixed binomial processes. In particular,
given a locally finite measure $\nu$ on $(\X, \Xcr)$, we say that $\Phi$ is a Poisson process with rate measure $\nu$ if for any collection of disjoint sets $B_1, \ldots, B_k \in \Xcr$, the random variables $\Phi(B_1), \ldots, \Phi(B_k)$ are independent and Poisson distributed with parameters $\nu(B_j)$, $j=1, \ldots, k$. The number of atoms of $\Phi$ can be either finite or infinite, but in the context of feature models it is typically assumed infinite. For brevity, we will write $\Phi \sim \textsc{pp} (\nu)$ to denote the distribution of a Poisson process.
A mixed Poisson process is simply a Poisson process with a random rate measure: $\nu = \gamma \tilde \nu$ where $\tilde \nu$ is a deterministic measure and $\gamma$ is a positive random variable. 
Instead, we say that $\Phi$ is a mixed binomial process if $\Phi = \sum_{j=1}^M \delta_{X_j}$, where $M$ is an integer-valued random variable (a.s. finite) and the $X_j$'s are i.i.d. $\X$-valued variables.

The mean measure $M_\Phi$ of $\Phi$ is defined by $M_\Phi(B) = \E[\Phi(B)]$, for $B \in \Xcr$, and the \emph{$k$-th order factorial} moment measure $M_{\Phi}^{(k)}$ of $\Phi$ by
\[
    M_{\Phi}^{(k)}(B) = \E \sum_{(j_1, \ldots, j_k)}^{\not =} \delta_{(X_{j_1}, \ldots, X_{j_k})} (B), \qquad B \in \Xcr^{\otimes k}
\]
where $\Xcr^{\otimes k}$ is the product $\sigma$-field, and the symbol $\not =$ indicates that the sum is extended over all pairwise distinct indexes.
We also denote by $\{\plaw_{\Phi}^x\}_{x \in \X}$ the \emph{Palm kernel} of $\Phi$. The Palm kernel is a central object in the study of point processes, and it arises as the disintegration probability kernel of the so-called Campbell measure (see Appendix \ref{app:point_process_definition} for further details).
Here, we limit ourselves to note that the Palm kernel is an extension of the concept of regular conditional distribution to the case of point processes. A point process
$\Phi_{x}$ with distribution $ \plaw_{\Phi}^x$ is called the \textit{Palm version of $\Phi$ at $x\in \X$}, and it holds that $\Pp(\Phi_x(\{x\}) \geq  1) = 1$, i.e., $x$ is a trivial atom of $\Phi_x$. Informally, $\plaw_{\Phi}^x$ can be understood as the probability distribution of $\Phi$ given that $\Phi$ has an atom at $x$. 
As a consequence, one can define the point process $\Phi^!_{x}:  = \Phi_{x} - \delta_{x}$, which is called 
the \textit{reduced Palm version of $\Phi$ at $x \in \X$}, whose associated reduced Palm kernel is indicated by $\plaw_{\Phi}^{!x}$.
In a similar fashion, it is possible to extend the definition of Palm distributions to multiple conditioning points $\bm{x} = (x_1, \ldots, x_k) \in \X^k$. In this case, the Palm distribution of $\Phi$ at $\bm{x}$ is interpreted as the probability distribution of $\Phi$ conditionally to $\Phi$ having $k$ atoms at locations $x_1, \ldots, x_k$. See Appendix \ref{app:point_process_definition}.
For the remainder of the paper, we denote by $\X$ the space of all possible feature labels. Moreover, it is worth noting that all the point processes central to the subsequent statistical modeling are defined on the product space $\X \times (0,1]$.

\subsection{Model definition} \label{sec:model_definetti}

Let $(X_j)_{j \geq 1}$ be a sequence of $\X$-valued random variables, hereafter interpreted as ``feature labels''.
An extended feature allocation model is a stochastic model for observations $(Z_i)_{i \geq 1}$ such that each $Z_i$ is characterized by its expressed features, or equivalently by the sequence of pairs $((A_{i,j}, X_j))_{j \geq 1}$, where $A_{i,j}=1$ if the feature $X_j$ belongs to the $i$-th individual, $A_{i,j}=0$ otherwise. Equivalently,  we can represent each $Z_i$ as a point process on $\X$, namely
\begin{equation} \label{eq:Zi_random_measure}
    Z_i = \sum_{j \geq 1} A_{i,j} \delta_{X_j}.
\end{equation}
We assume that, for any $j \geq 1$, the $A_{i,j}$'s are conditionally i.i.d. Bernoulli variables with a strictly positive random probability $S_j$, which entails both exchangeability of the $Z_i$'s and regularity of the feature allocation \citep{Bro(13)}. Namely, no feature appears in only a single observation within $(Z_i)_{i\geq 1}$, or, equivalently, if a label is observed in some individual, then it is observed in infinitely many individuals with probability one.

The feature labels $X_j$'s and parameters $S_j$'s, i.e., the probability of displaying feature $X_j$, are collected in the simple point process on $\X \times (0,1]$ defined by
$\Psi: = \sum_{j \geq 1} \delta_{(X_j, S_j)}$
or in its functional 
\begin{equation}
    \label{eq:mu_definition}
    \mu(B) = \int_{\X \times (0,1]} s \indicator_B (x ) \Psi(\dd x \, \dd s), \quad B \in \Xcr,
\end{equation}
where $\indicator_B $ denotes the indicator function of the Borel set $B$.
We say that $Z_i$ in \eqref{eq:Zi_random_measure} is a Bernoulli process with parameter $\mu$, indicated as $Z_i \sim  \mathrm{BeP}(\mu)$. According to our point process formulation, we observe that  $Z_i$ in \eqref{eq:Zi_random_measure} is obtained by first thinning $\Psi$ with retention function $p(x,s) = s$ and then removing the second component.
Summing up, we are dealing with the following statistical model  
\begin{equation} \label{eq:representation_theorem}
    \begin{aligned}
        Z_i \mid \mu & \iid \mathrm{BeP}(\mu),\\
        \mu & \sim \mathscr Q,
    \end{aligned}
    \end{equation}
where $\mathscr Q$ is the law of $\mu$. In other terms, $(Z_i)_{i \geq 1}$ is a sequence of exchangeable observations with de Finetti measure $\mathscr Q$.  We call model \eqref{eq:representation_theorem} an \textit{extended feature allocation model}, where \emph{extended} refers to the joint modeling of feature labels and feature probabilities.

\section{Bayesian analysis of extended feature models}\label{sec:bnp}

We provide here the full Bayesian analysis of the class of extended feature allocation models.
We focus on the statistical model \eqref{eq:representation_theorem} to describe the marginal distribution of a sample, the predictive structure of the next observation, and the posterior representation of the underlying measure $\mu$.

Throughout the rest of the paper, we assume that the $k$-th factorial moment measures of $\Psi$, denoted as $M_{\Psi}^{(k)}$, are $\sigma$-finite for all $k \geq 1$. This is a very mild assumption (for instance, it is verified by all feature allocation models previously considered in the literature) and is essential for the existence of the $k$-th order Palm distributions. Furthermore, under this assumption, the measure disintegration theorem \cite[Theorem 3.4]{Kallenberg21} entails
\begin{equation}\label{eq:mpsi_disinteg}
    M_{\Psi}^{(k)}(\dd \bm x \, \dd \bm s) =  \rho^{(k)}(\dd \bm s \mid \bm x) \tilde m^{(k)}_{\xi}(\dd \bm x),
\end{equation}
where $\tilde m^{(k)}_{\xi}(\, \cdot \, )$ is a $\sigma$-finite measure \emph{equivalent} to $M_{\Psi}^{(k)}(\, \cdot \,  \times (0,1]^k)$, i.e., both measures are absolutely continuous with respect to each other, and $\rho^{(k)}$ is a kernel from $\X^k$ to $(0,1]^k$.

Remarkably, another common assumption for feature allocation models is that any subject displays a.s. a finite number of features, meaning $Z_i(\X) < \infty$ a.s.. Clearly, this property is always guaranteed when $\Psi$ has a.s. a finite number of points. If $\Psi$ has an infinite number of points, a sufficient condition is $\int_{\X\times (0,1]} s M_{\Psi}(\dd x \, \dd s) < \infty$. See Appendix \ref{app:results_extended_feature} for further considerations on this point.

\subsection{General formulas: marginal, posterior and predictive distributions}\label{sec:general_bayesian_analysis}
We start by describing the marginal distribution of a sample $\bm{Z}:= (Z_1,\ldots,Z_n)$ from an extended feature model as in \eqref{eq:representation_theorem}. The proofs are based on an application of Palm calculus \citep{BaBlaKa} and deferred to \Cref{app:proofs_main_theorems}.

\begin{theorem}\label{thm:marg}
    Let $\bm{Z}$ be a sample from the statistical model \eqref{eq:representation_theorem}, where $\mu$ is the functional of a point process $\Psi$ defined via \eqref{eq:mu_definition}.
    The  probability that $\bm{Z}$ displays $k$ features having labels $\bm x^* = (x^*_1, \ldots, x^*_k)$ with corresponding vector of  frequency counts $\bm{m}:= (m_1, \ldots, m_k)$ is
    \[
         \int_{(0, 1]^k} \E\left\{ e^{ \int_{\X \times (0, 1]} n \log(1 - t) \Psi^!_{\bm x^*, \bm s}(\dd z \, \dd t)} \right\}   \prod_{\ell=1}^k  s_\ell^{m_\ell} (1 - s_\ell)^{n - m_\ell} \rho^{(k)}(\dd \bm s \mid \bm x^*) \cdot  \tilde m_{\xi}^{(k)}(\dd \bm x^*),
    \]
    where $\rho^{(k)}(\dd \bm s \mid \bm x^*)$ and $\tilde m_{\xi}^{(k)}(\dd \bm x^*)$ are defined as in \eqref{eq:mpsi_disinteg}, and, for $\bm x^* \in \X^k$ and $\bm s \in (0, 1]^k$, $\Psi^!_{\bm x^*, \bm s}$ is the reduced Palm version of $\Psi$ at $((x_1^*, s_1), \ldots, (x_k^*, s_k))$.
\end{theorem}

Theorem \ref{thm:marg} decomposes the marginal distribution into the contributions of the observed features and the hitherto unseen ones. The term \(\prod_{\ell=1}^k s_\ell^{m_\ell}(1-s_\ell)^{n-m_\ell}\) is  associated with the observed features, whereas the expectation involving the reduced Palm version summarizes the contribution of all features that have not yet appeared in the sample. This is precisely where interactions among feature labels enter the marginal distribution. For priors with i.i.d. labels, $\Psi^!_{\bm x^*, \bm s}$ does not depend on $\bm x^*$.
The law of the induced feature allocation is obtained by marginalizing out the labels $\bm x^*$ and dividing by $k!$. Such a law 
can always be expressed as a symmetric function of the frequency counts $m_1,\ldots,m_k$, which means that it admits the exchangeable feature probability function (\textsc{efpf}), defined by \cite{Bro(13)}. Moreover, all these extended feature models are regular by construction.

We proceed by describing the general posterior distribution of the point process $\Psi$, or equivalently the random measure $\mu$, given a sample of size $n$ from \eqref{eq:representation_theorem}.

\begin{theorem}\label{thm:post}
    Let $\bm{Z}$ be a sample from the statistical model \eqref{eq:representation_theorem}, where $\mu$ is the functional of a point process $\Psi$ defined via \eqref{eq:mu_definition}. Further, suppose that $\bm{Z}$ displays $k$ features having labels $\bm x^* = (x^*_1, \ldots, x^*_k)$ with corresponding vector of  frequency counts $\bm{m}:= (m_1, \ldots, m_k)$.
  Then, the posterior distribution  of $\mu$, conditionally on $\bm Z$, satisfies the distributional equality
   \begin{equation}\label{eq:post_mu_sum}
     \mu \mid \bm Z \dequal  \sum_{\ell=1}^{k} S^*_\ell \delta_{x^*_\ell} +  \mu^{\prime},
   \end{equation}
   where
   \begin{enumerate}
       \item[(i)] $\bm S^* := (S^*_1, \ldots, S^*_k)$ is a vector of positive random variables with joint distribution
       \[
        f_{\bm S^*}(\dd \bm s) \propto  \E\left\{e^{\int_{\X \times (0, 1]} n \log(1-t) \Psi^!_{\bm x^*, \bm s}(\dd z \, \dd t)}\right\} \prod_{\ell=1}^k s_\ell^{m_\ell} (1 - s_\ell)^{n - m_\ell} \rho^{(k)}(\dd \bm s \mid \bm x^*);
       \]
       \item[(ii)] conditionally on $\bm S^* = \bm s^*$, the random measure $\mu^\prime$ admits the representation  $\mu^{\prime} = \sum_{j \geq 1} S^\prime_j \delta_{X^\prime_j}$, where the distribution of $\Psi^\prime : = \sum_{j \geq 1} \delta_{(X^\prime_j, S^\prime_j)}$ is absolutely continuous with respect to the distribution of $\Psi^!_{\bm x^*, \bm s^*}$, with density
\begin{equation}\label{eq:density_psiprime}
        f_{\Psi}^\prime(\nu) \propto e^{\int_{\X \times (0, 1]} n \log(1 - t) \nu(\dd z \, \dd t)}.
\end{equation}
    \end{enumerate}
\end{theorem}

From Theorem \ref{thm:post}, the posterior distribution of $\mu$ is decomposed in a part that involves previously observed labels out of the sample and a component that involves hitherto unseen features, i.e., $\mu^\prime$. Note that $\mu^{\prime}$ in point (ii) is an a.s. discrete random measure whose Laplace functional is available and equals
\begin{equation}\label{eq:post_laplace_main}
    \E\left\{ e^{- \int_\X f(z) \mu^\prime(\dd z)}  \mid \bm S^* = \bm s^* \right\} = \frac{\E\left\{ e^{- \int_{\X \times (0, 1]} t f(z) - n \log(1 - t) \Psi^!_{\bm x^*, \bm s^*}(\dd z \, \dd t)} \right\}}{\E\left\{ e^{\int_{\X \times (0, 1]} n \log(1 - t) \Psi^!_{\bm x^*, \bm s^*}(\dd z \, \dd t)} \right\}}.
\end{equation}
The posterior representation resembles available results in the literature; see, for example, \cite[Theorem 3.1]{Jam(17)}, with the fundamental difference that here the distribution of $\mu^\prime$ depends on the previously observed features $\bm x^*$ via the reduced Palm version of $\Psi$, which describes how the distribution of unseen features is reshaped after conditioning on the observed labels.

To conclude the general Bayesian analysis of the extended feature models in \eqref{eq:representation_theorem}, we provide the predictive distribution of the next observation, conditionally on the available sample, which easily follows from \Cref{thm:post} by a standard disintegration argument.

\begin{theorem}\label{cor:pred}
  Under the same assumptions of Theorem \ref{thm:post}, the conditional distribution of $Z_{n+1}$, given the sample  $\bm Z$,  satisfies the following distributional equality
    \begin{equation}\label{eq:pred_sum}
       Z_{n+1} \mid \bm Z \dequal \sum_{\ell=1}^{k} A^*_{n+1, \ell} \delta_{x^*_\ell} + Z^\prime_{n+1},
    \end{equation}
    where
    \begin{enumerate}
        \item[(i)] $(A^*_{n+1, 1}, \ldots, A^*_{n+1, k})$, conditionally on  $S^*_1, \ldots, S^*_k$, is a vector of independent Bernoulli random variables with parameters $(S^*_1, \ldots, S^*_k)$, which have been defined in  \Cref{thm:post};
        \item[(ii)] $Z^\prime_{n+1}  \sim \mathrm{BeP}(\mu^\prime)$ and $\mu^\prime$ is distributed according to \Cref{thm:post}.
    \end{enumerate}
\end{theorem}

Note that, in the point process language, the latter measure $Z^\prime_{n+1}$ in (ii) is obtained by first thinning the process $\Psi^\prime = \sum_{j \geq 1} \delta_{(X^\prime_j, S^\prime_j)}$ with retention probability $p(x, s) = s$, and then discarding the second component. Theorem \ref{cor:pred} shows that the next individual $Z_{n+1}$ has a positive probability of displaying the features $x^*_1, \ldots, x^*_k$ observed out of the initial sample, and it can display hitherto unseen features, which correspond to the atoms of $Z^\prime_{n+1}$.

Together, \Cref{thm:post,cor:pred} identify the mechanism through which dependence among feature labels enters prediction. The distribution of new features is determined by the distribution of $\mu^\prime$, or equivalently by that of $\Psi^\prime$, which is a tilted reduced Palm version of the prior process $\Psi$. 
Therefore, the role of the observed labels in predicting unseen features is governed entirely by the reduced Palm law: the labels $\bm x^*$ influence the prediction of unseen features if and only this law depends on $ \bm x^*$.
 This perspective both recovers and generalizes previous results. Under a \textsc{crm} prior, the dependence collapses to the sample size $n$ alone \citep[Proposition 3.2]{Jam(17)}; under the stable-beta scaled process of \cite{Cam(23)}, it extends to $n$ and $k$.
\Cref{sec:predictive_characterization} characterizes the classes of priors giving rise to these predictive structures, whereas \Cref{sec:sncp_model_and_application,sec:dpp_model_and_application} develop models in which the labels themselves contribute to prediction.

\section{Sufficientness postulates} \label{sec:predictive_characterization}


In the same spirit as sufficientness postulates for species sampling models, we now consider the problem of characterizing prior distributions in extended feature allocation models, leading to predictive distributions that satisfy specific properties.
In particular, we focus our analysis on how the Bernoulli process $Z^\prime_{n+1}$ in \eqref{eq:pred_sum} depends on the sampling information in $\bm Z$. As we will clarify in \Cref{rmk:more_general_postulates} at the end of this section, there is no gain in considering $Z_{n+1}$ in place of $Z^\prime_{n+1}$.
From point (ii) of \Cref{cor:pred}, it is clear that the distribution of  $Z^\prime_{n+1}$ is uniquely characterized by the random measure $\mu^\prime$ in \eqref{eq:post_mu_sum} with Laplace functional \eqref{eq:post_laplace_main}.
Therefore, in general, the predictive distribution of $Z^\prime_{n+1}$ depends on the sample size $n$, and on the observed feature label-frequency pairs \((\bm x^*,\bm m)\), up to the arbitrary ordering of the observed features.

Special cases of extended feature models previously studied in the literature, lead to much simpler predictive laws. For \textsc{crm} priors \cite{Jam(17)} the law of $Z^{\prime}_{n+1}$ depends only on the sample size $n$ (see also \Cref{thm:james}). Instead, in the case of the stable beta scaled process in \cite{Cam(23)} and for feature models having a product form \textsc{efpf} \citep{battiston18,ghilotti2024bayesian}, such predictions may also depend on the number of distinct features $k$.
Motivated by such examples, we characterize the classes of extended feature allocation models for which the law of $Z^\prime_{n+1}$ depends on the initial sample $\bm Z$ (i) only through the sample size $n$, and (ii) solely on $n$ and the number of observed features $k$. 

\begin{theorem}\label{teo:pred_char}
Consider a sample $\bm Z$ from the statistical model \eqref{eq:representation_theorem}. Then,
\begin{enumerate}
    \item[(i)]  The distribution of $Z^\prime_{n+1}$ in \Cref{cor:pred} depends on the observed sample $\bm Z$ solely through the sample size $n$ if and only if $\Psi$ in model \eqref{eq:representation_theorem} is a Poisson process.
    \item[(ii)] The distribution of $Z^\prime_{n+1}$  depends on the observed sample $\bm Z$ solely through $n$ and $k$ if and only if $\Psi$ in model \eqref{eq:representation_theorem} is a mixed Poisson or mixed binomial process. 
\end{enumerate}   
\end{theorem}
The proof of Theorem \ref{teo:pred_char} is deferred to the supplementary material (see Section \ref{app:pred_characterization}). Point (i) of this theorem in the feature setting is the most natural counterpart of the results by  \citet{Reg(78),Lo(91)}, who characterized the Dirichlet process as the unique species sampling prior in which the probability of observing a new species depends only on the sample size, and the probability of observing a previously recorded species depends on both the sample size and its frequency. 
Other characterizations for the Pitman--Yor process and Gibbs-type priors \citep{Zab05,Bac(17)}, show the corresponding role of the sample size and the number of observed species.  In the feature setting, point (ii) of \Cref{teo:pred_char} gives the analogous result for priors whose unseen-feature prediction depends only on \(n\) and \(K_n\).

Point (i) of Theorem \ref{teo:pred_char} fully characterizes the class of \textsc{crm} priors analyzed in \cite{Jam(17)}
(see also Section \ref{supp:Poisson}), while point (ii) of Theorem \ref{teo:pred_char} characterizes a broad class of prior distributions that includes, among others, all feature models having a product form \textsc{efpf} \citep{battiston18}, as well as the stable beta scaled processes by \cite{Cam(23)}. 
From \cite{battiston18,ghilotti2024bayesian}, it is not difficult to realize that the point process $\Psi$ associated with a feature model having a product form \textsc{efpf} could be either ($i$) a mixed binomial point processes, i.e., $\Psi = \sum_{j=1}^M \delta_{(X_j, S_j)}$, where $M$ is random, the $S_j$'s are i.i.d. beta distributed and the $X_j$'s are i.i.d. from a diffuse measure further independent of the $S_j$'s, or ($ii$) a mixed Poisson processes  $\mathrm{MP}(\nu, f_\gamma)$ with $\nu$ being the L\'evy intensity of a three-parameter beta process \citep{Teh09} and $\gamma$ a positive random variable.
These are only very special cases of mixed binomial and mixed Poisson processes. In Sections \ref{supp:mixed_Poisson} and \ref{sec:mixed_binomial} we specialize our general theorems to the classes of mixed binomial and mixed Poisson processes, providing a comprehensive analysis of all the priors leading to predictions depending on $n$ and $K_n$.
Finally, Section \ref{supp:prediction_all} suggests a straightforward strategy to induce a predictive distribution depending on the whole frequency spectrum while maintaining computational convenience.
\begin{remark} \label{rmk:more_general_postulates}
Sufficientness postulates for species sampling models characterize specific prior distributions such as the finite Dirichlet distribution, the Dirichlet process, and the Pitman-Yor process \citep{Bac(17)}, by imposing conditions on the probability of observing a new species and the probability of re-observing a species recorded in the sample. By contrast, our postulates focus only on the distribution of new features and characterize broad classes of priors.  It is then natural to wonder if, by adding further conditions on previously observed features, it is possible to restrict the characterization to specific prior distributions.
This question has a negative answer, as clarified by a simple counterexample. In fact, for any Poisson, mixed Poisson, or mixed binomial prior, the probability of re-observing a feature depends exclusively on the sample size and the frequency of that feature (see Sections \ref{supp:Poisson}, \ref{supp:mixed_Poisson} and \ref{sec:mixed_binomial}).
This excludes the possibility of distinguishing within the class of \textsc{crm}s by adding structural constraints on the predictive law for the previously observed features.
\end{remark}

Hereafter, we move beyond the traditional classes of priors characterized in \Cref{teo:pred_char}, leveraging the general results of Section \ref{sec:general_bayesian_analysis} to induce interactions across feature labels and to address applications in which such interactions are scientifically meaningful.

\section{Joint clustering and feature discovery via Cox process priors}\label{sec:sncp_model_and_application}

We now consider settings in which feature labels are informative and expected to exhibit clustering structure.
Such settings fall outside the regime characterized in \Cref{sec:predictive_characterization}, as prediction must exploit information contained in the labels themselves.
In this section we are motivated by a genomic application, where the feature labels represent \emph{embeddings} of genomic variants. Such embeddings are high-dimensional Euclidean vectors that encode key aspects of the variants, and were obtained following the pipeline of \cite{niu2025incorporating}: we obtained textual descriptions of each variant by querying online databases with genomic annotations, we processed those texts in a standardized format, and used a large language model (\textsc{llm}) to obtain sentence embeddings. See Section \ref{sec:sncp_genomics} for additional details. The statistical goal here is threefold. First, we aim at clustering the variants through their embeddings, since such clustering may reveal functional or evolutionary groupings that are biologically meaningful and may guide downstream interpretation. Second, we consider the usual unseen feature problem, whereby we predict the number of new variants discovered in an additional sample. Third, we want to quantify how many new clusters may emerge in future samples.
To address this joint inferential problem, in \Cref{sec:mixture} we introduce a novel class of extended feature models based on cluster Cox processes or, equivalently, on a mixture model for the feature labels $X_j$'s. 
We then apply the resulting framework to Glioblastoma variant data in \Cref{sec:sncp_genomics}.

\subsection{The Cox process prior} \label{sec:mixture}

Cox point processes \citep{BaBlaKa} are a natural tool when the points of interest are arranged into latent clusters. Here, we use them to induce clustering among the feature labels $X_j$'s. More specifically, $\Psi$ in \eqref{eq:mu_definition} is defined as a Poisson process with random intensity according to the following hierarchical construction
\begin{equation}\label{def:psi_sncp}
\begin{split}
    \Psi \mid \Lambda &\sim \textsc{pp}\left( \rho(\dd s) L_\Lambda (x)  \dd x \right), \qquad \text{where} \; \;  L_\Lambda (x) =  \int_{\Theta \times\R_+ } \gamma \kappa(x; \theta) \Lambda (\dd \theta \, \dd \gamma)\\
 \Lambda &\sim \textsc{pp}(\tilde{\rho}(\dd \gamma) G_0(\dd \theta)),
    \end{split}
\end{equation}
where $\kappa$ is a parametric probability kernel from $\Theta$, the parameter space, to $\X$; $\rho(\dd s)$ is a locally finite intensity measure on $(0,1]$ such that $\int_{\R_+} s \rho(\dd s) < \infty$, which governs the probability of observing each feature. $G_0$ is a probability distribution on $\Theta$, and $\tilde\rho$ is a locally integrable measure on $\R_+$ such that $\int_{\R_+} \gamma \Lambda(\dd \theta\,\dd \gamma) < \infty$ almost surely. If $\rho$ is integrable, the marginal process of feature labels $X_j$'s, is a shot noise Cox process in the sense of \cite{Moller03}.

Spelling out the point process $\Lambda =  \sum_{h\geq 1} \delta_{(\theta_h, \gamma_h)} $ in \eqref{def:psi_sncp} reveals a strong connection with Bayesian mixture models. Indeed, $L_\Lambda(x) = \sum_{h\geq 1} \gamma_h \kappa(x; \theta_h)$. Hence, the $X_j$'s are i.i.d. from the mixture density $\frac{1}{\Gamma} L_\Lambda$, where $\Gamma = \sum_{h \ge 1} \gamma_h$ is assumed to be finite almost surely. Hence, under \eqref{def:psi_sncp}, the feature labels are naturally organized into different latent clusters.
In particular, \eqref{def:psi_sncp} can be equivalently rephrased as a mixture model on the labels, which is a more common formulation among the Bayesian nonparametric literature. 
While such a formulation is intuitive, the one in \eqref{def:psi_sncp} is better suited to exploit the general results of Section \ref{sec:bnp}. See  \cite{Wang2024} for further insights on the connections between cluster processes and mixture models.

We now specialize the posterior representation in \Cref{thm:post} for the Cox process prior, given an observable sample $\bm Z$ from the statistical model \eqref{eq:representation_theorem}, displaying features $\bm x^* = (x^*_1,\ldots,x^*_k)$ with associated frequency counts $\bm m = (m_1,\ldots,m_k)$.
Let us introduce a  set of latent variables $\bm T := (T_1, \ldots , T_k)$ describing a partition of the observed features $(x^*_1,\ldots,x^*_k)$ into $C_n:=|\bm T|$ clusters such that $T_\ell= h $ if and only if $x_\ell^*$ belongs to cluster $h$. We write $\bm x_h^*: = \{ x_\ell^*: T_\ell = h\}$ for the features allocated to cluster $h$, and  $n_h$  for their cardinality. The proof of the next corollary shows $\bm T$ has distribution
   \begin{equation} \label{eq:distribution_T}
         \Pp(\bm T = \bm t) \propto \prod_{h = 1}^{c_n} \int_{\R_+} e^{-\gamma\varphi_n} \gamma^{n_h} \int_\X \prod_{\ell: t_\ell = h} \kappa( x^*_\ell; \theta) G_0(\dd \theta) \tilde{\rho}(\dd \gamma),
    \end{equation} 
    where $c_n=|\bm t|$ and $\varphi_n = \int_{(0,1]} (1 - (1-s)^n) \rho(s) \dd s$.
\begin{corollary} \label{prop:bayesian_sncp_features_posterior}
Let $\bm Z$ be a sample from \eqref{eq:representation_theorem}, where $\mu$ is the functional of the process $\Psi$ described in \eqref{def:psi_sncp}. 
Then, the posterior distribution of $\mu$ is as in \eqref{eq:post_mu_sum}, where $S^*_\ell \ind f_{S^*_\ell}(\dd s) \propto s^{m_\ell} (1 - s)^{n - m_\ell} \rho(\dd s)$, and are further independent of $\mu^\prime$. Moreover $\mu^\prime = \sum_{j \geq 1} S^\prime_j \delta_{X^\prime_j}$, where $\Psi^\prime= \sum_{j \geq 1} \delta_{(X_j^\prime , S_j^\prime)}$ satisfies the following distributional equality
    \begin{equation}\label{eq:joint_psi_prime_T}
        \Psi^\prime \mid \bm T \dequal \Psi^{(0)} + \sum_{h=1}^{C_n} \Psi^{(\bm x^*_h)}
    \end{equation}
being  $\bm T := (T_1,\ldots,T_k)$ the latent allocation variables  with distribution \eqref{eq:distribution_T}. 
Moreover,
    the processes $\Psi^{(0)}$ and $\Psi^{(\bm x^*_h)}$ in \Cref{eq:joint_psi_prime_T}, as $h = 1,\ldots, C_n$, are mutually independent conditionally to $\bm T$, having the distribution described below:
    \begin{enumerate}
        \item[(i)] $\Psi^{(0)}$ is again a Cox process having a distribution similar to $\Psi$, namely
        \begin{equation*}
            \Psi^{(0)} \mid \Lambda^{(0)} \sim\textsc{pp}\left( \rho_n(\dd s)  L_{\Lambda^{(0)}} (x) \dd x \right),   
           \qquad \Lambda^{(0)} \sim \textsc{pp}(e^{- \gamma\varphi_n} \tilde{\rho}(\dd \gamma) G_0(\dd \theta)),
        \end{equation*}
    with $\rho_n(\dd s) = (1 - s)^n \rho(\dd s)$ and $ L_{\Lambda^{(0)}}$ is the integral defined in \eqref{def:psi_sncp} where $\Lambda$ is replaced with $\Lambda^{(0)}$;
    \item[(ii)] for each $h=1,\ldots,C_n$, the process $\Psi^{(\bm x^*_h)}$ is such that 
            \begin{equation*}
            \begin{aligned}
                \Psi^{(\bm x^*_h)} &\mid \left( \zeta_{\bm x^*_h} = (\theta_{\bm x^*_h}, \gamma_{\bm x^*_h}) \right) \sim \textsc{pp}\left( \gamma_{\bm x^*_h} \rho_n(\dd s) \kappa(x; \theta_{\bm x^*_h}) \, \dd x \right),\\
                \zeta_{\bm x^*_h}  &\sim f^{(\bm x^*_h)}(\dd\theta\, \dd \gamma) \propto e^{-\gamma \varphi_n \kappa(x; \theta) \dd x} \gamma^{n_h} \tilde{\rho}(\dd \gamma)  \prod_{\ell: t_\ell = h} \kappa( x^*_\ell; \theta)  G_0(\dd \theta).
            \end{aligned}
            \end{equation*} 
    \end{enumerate}   
\end{corollary}
The proof of \Cref{prop:bayesian_sncp_features_posterior} is reported in \Cref{app:sncp_appendix}, together with the marginal distribution  (\Cref{app:add_res_cox_model}). Corollary \ref{prop:bayesian_sncp_features_posterior} gives a transparent statistical interpretation of the posterior. Conditionally on the latent partition $\bm T$, the unseen-feature process splits into two independent components. The process $\Psi^{(0)}$ represents features belonging to clusters not yet represented among the observed features, whereas the processes $\Psi^{(\bm x^*_h)}$, $h=1,\ldots,C_n$, represent additional unseen features attached to clusters already suggested by the observed features.
In particular,  \(\Psi^{(0)} \mid \Lambda^{(0)}\) in \eqref{eq:joint_psi_prime_T} is the  superposition   \(\sum_{\ell \geq 1} \Psi_\ell^{(0)}\), where \(\Lambda^{(0)} = \sum_{\ell \geq 1} \delta_{(\theta_\ell^{(0)}, \gamma_\ell^{(0)})}\), and each \(\Psi_\ell^{(0)}\) is a Poisson process with intensity $\rho_n(\dd s) \, \gamma_\ell^{(0)} \kappa(x; \theta_\ell^{(0)})\,\mathrm{d}x$.  Hence, unobserved features are organized into a collection of unobserved clusters, represented by \((\Psi_h^{(0)})_{h \geq 1}\), or may aggregate into one of the previously observed \(C_n\) clusters if they belong to a point process of the type \(\Psi^{(\bm x_h^*)}\).
The predictive distribution of the next observation $Z_{n+1}$ still has the general form \eqref{eq:pred_sum}. Under the Cox prior, the term $Z_{n+1}^\prime$ is a Bernoulli process with parameter $\mu^\prime =  \sum_{j \geq 1} S^\prime_j \delta_{X^\prime_j}$, where $((X_j^\prime, S_j^\prime))_{j \geq 1}$ are the points of the process $\Psi^\prime$ in \eqref{eq:joint_psi_prime_T}. 
In \Cref{app:add_res_cox_model} we exploit this relation recursively to obtain the $m$-step ahead prediction $K_m^{(n)}$, i.e., the number of new features that would be discovered in a future sample of size $m$. In particular, such predictive distributions depend on the observed sample through $n, K_n$ and $x^*_1, \ldots, x^*_{K_n}$ thus enriching the predictive structure of the Poisson, mixed Poisson and binomial families.
The model therefore addresses clustering and discovery simultaneously, with future observations either reinforcing existing clusters or revealing entirely new ones.

Our third inferential goal is the number of \emph{new clusters} that would be discovered in a future sample of size $m$. That is, if $\bm Z^{(n+m)}$ is a sample of size $n+m$ from \eqref{eq:representation_theorem} where we observe the first $n$ coordinates $\bm Z$ and introduce the allocation variables $\bm T^{(n)}$ and $\bm T^{(n+m)}$ with law as in \eqref{eq:distribution_T}, then the number of new clusters is $C_m^{(n)} := |\bm T^{(n+m)}| - |\bm T^{(n)}|$. The next proposition provides the distribution of $C_m^{(n)}$.
\begin{proposition}\label{prop:sncp_Cmn}
Consider a sample $\bm Z$ from the statistical model \eqref{eq:representation_theorem}, where $\mu$ is the functional of the process $\Psi$ described in \eqref{def:psi_sncp} . 
    The number of new clusters in a future sample of size $m$ is distributed as
    \[
        C_m^{(n)} \mid \bm Z \sim \mathrm{Poisson}\left( \int \{1 - \exp(-\gamma(\varphi_{n+m}- \varphi_n))\} e^{-\gamma \varphi_n} \tilde{\rho}(\dd \gamma)  \right).
    \]
\end{proposition}
The proof is reported in \Cref{app:proofs_cox}.
Observe that the number of new clusters $C_m^{(n)}$ depends on the observed sample $\bm Z$ only on the sample size $n$, which is expected since $\Lambda$ is a Poisson process. However, putting hyperpriors on the parameters of $\Lambda$ lets these predictions use more of the observed sample.

\subsection{Parameter selection and fitting details}\label{sec:ext_gaussian_ibp}

We now specify a set of standard choices for the parameters in \eqref{def:psi_sncp}. 
The intensity measure $\rho$ affects the feature occurrence probabilities. Motivated by our genomic application, we follow \cite{Mas(22)} and assume the three-parameter beta process \citep{Teh09} $\rho(\dd s)= c\,s^{-1-\alpha}(1-s)^{\beta+\alpha-1}\indicator_{(0,1]}(s) \ \dd s,$
with $c>0$, $\alpha \in [0,1)$ and $\beta>- \alpha$. This yields feature occurrence probabilities with a power law with exponent $-\alpha$, while $K_n$, a priori, grows like $n^\alpha$ \citep{Mas(22)}.
Moreover $\X = \R^d$ is the space where the embeddings are defined. To cluster the embeddings we specify $\kappa$ as a Gaussian kernel as standard in model-based clustering. 
To avoid overparametrizing our model, we assume $\kappa(\cdot;\theta) = \calN(\cdot; \mu, \sigma^2 I_d)$, where $\theta = (\mu, \sigma^2) \in \R^d \times \R_+$, and $I_d$ denotes the $d\times d$ identity matrix. We pair this with a conjugate base measure $G_0 (\dd \mu \, \dd \sigma )$ given by $\mu\mid \sigma^2 \sim\mathcal{N}(m_0,\sigma^2(\lambda_0^{-1})I_d)$ and $\sigma^2 \sim \mathrm{Inv-Gamma}(a,b)$,
which simplifies posterior computations.
Moreover, we assume that $\Lambda$ is a Gamma process, that is $\tilde{\rho}(\dd \gamma) = \tau_0 \gamma^{-1} e^{-b_0 \gamma} \indicator_{\R_+}(\gamma) \ \dd \gamma$,
with $\tau_0, b_0 > 0$.  
Hence, the feature labels $X_j$'s are generated from a Dirichlet process mixture model with mixture kernel $\kappa$.  
We refer to the model resulting from this combined choice of $\rho$, $\kappa$, $\tilde{\rho}$ and $G_0$ as \emph{extended Gaussian-\textsc{ibp}}. 
In this case, it can be easily shown  the distribution of $C_m^{(n)}$ in \Cref{prop:sncp_Cmn} boils down to 
    \[
        C_m^{(n)}\mid \bm Z \sim \mathrm{Poisson}\left( \tau_0 \log\left(\frac{b_0+\varphi_{n+m}}{b_0+\varphi_n}\right) \right).
    \]
Posterior inference for the extended feature model is performed via Markov chain Monte Carlo (\textsc{mcmc}), exploiting the posterior characterization in \Cref{prop:bayesian_sncp_features_posterior}. See \Cref{app:mcmc_sncp} for detailed description of the updates, with a focus on the extended Gaussian-\textsc{ibp} we described above. In particular, we put hyperpriors on the parameters $\tau_0,b_0$ in $\tilde{\rho}$ and $c, \alpha, \beta$ in $\rho$, as specified in \Cref{sec:sncp_genomics}, which let the
model adapt to the observed data. With these hyperpriors, the predictive laws of both $C_{m}^{(n)}$ (the number of new clusters) and $K_m^{(n)}$ (the number of new features) depend on the whole sampling information, that is the feature frequencies $\bm m$ and the feature labels $\bm x^*$.
In order to perform posterior inference on the quantities of interest like $K_m^{(n)}$ and $C_m^{(n)}$, the \textsc{mcmc} algorithm needs to sample from the joint posterior distribution of $\bm T$, $\Gamma^{(0)}$, $(\theta_{\bm x^*_h}, \gamma_{\bm x^*_h}: h = 1,\ldots, C_n )$, and the parameters $\tau_0,b_0, c, \alpha, \beta$. The details of the \textsc{mcmc} scheme are reported in \Cref{app:mcmc_sncp}.

\subsection{Analysis of genomic variants}\label{sec:sncp_genomics}

We consider an application to data on genomic variants for patients affected with Glioblastoma cancer, publicly available as part of the cancer genome atlas (\textsc{tcga}) at \url{https://portal.gdc.cancer.gov}. The data records information for total of 221 subjects, displaying 469 distinct genomic variants.
Such data were previously analyzed by \cite{Mas(22)} and \cite{Cam(23)} using a Bayesian feature model with a completely random measure and scaled process priors, respectively.
In both cases, data were represented as $Z_i = \sum_{j \geq 1} A_{i, j} \delta_{X_j}$ where $(X_j)_{j \geq 1}$ are the unique ids associated to each genomic variant. 
Crucially, \cite{Mas(22)} and \cite{Cam(23)} treated genomic variants as irrelevant labels, discarding the information encoded in the embeddings.

Instead, as discussed at the beginning of this section, we follow the pipeline of \cite{niu2025incorporating} to associate with each genomic variant a high-dimensional Euclidean vector called \emph{embedding}.
Specifically, we used annotations from the FAVOR database and a publicly available sentence transformer (\texttt{all-MiniLM-L6-v2} from HuggingFace) to obtain the variant embeddings as euclidean vectors in a 384-dimensional space.
As common practice when analyzing such embeddings, we then performed principal component analysis on the embeddings to reduce the dimensionality to 5, which retained more than 85\% of the total variability. Figure \ref{fig:tcga_data} (left panel) shows the projection on the first two principal components of the embeddings variants, clearly showing three large clusters.
\begin{figure}
    \centering
    \includegraphics[width=0.33\linewidth]{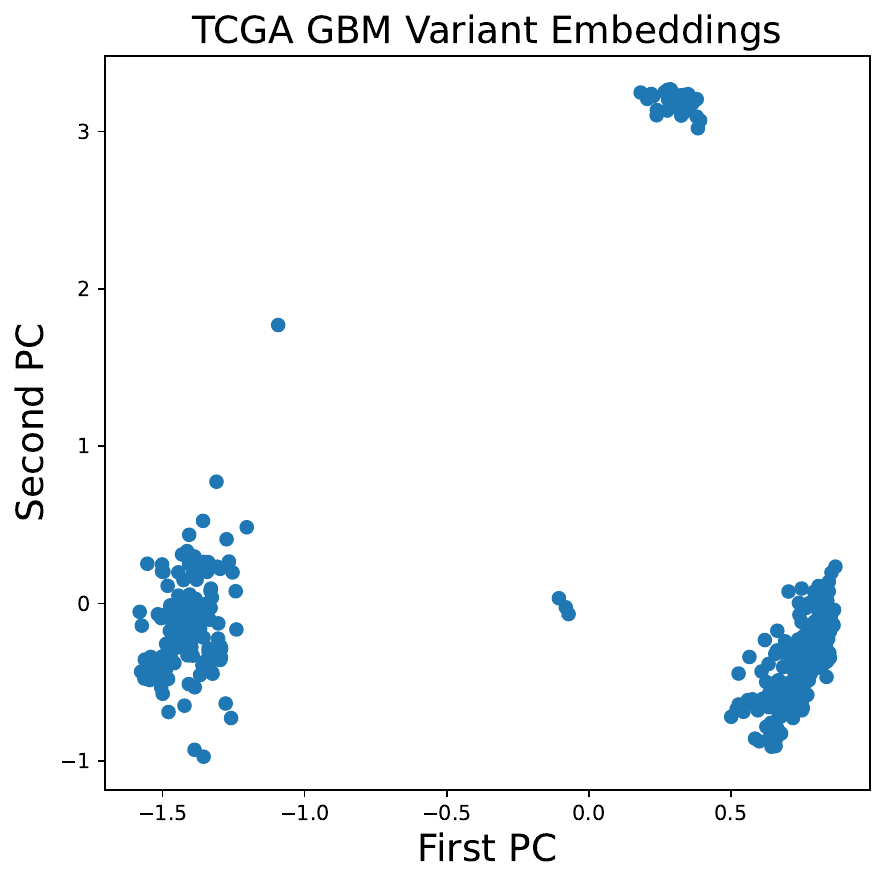}%
    \includegraphics[width=0.33\linewidth]{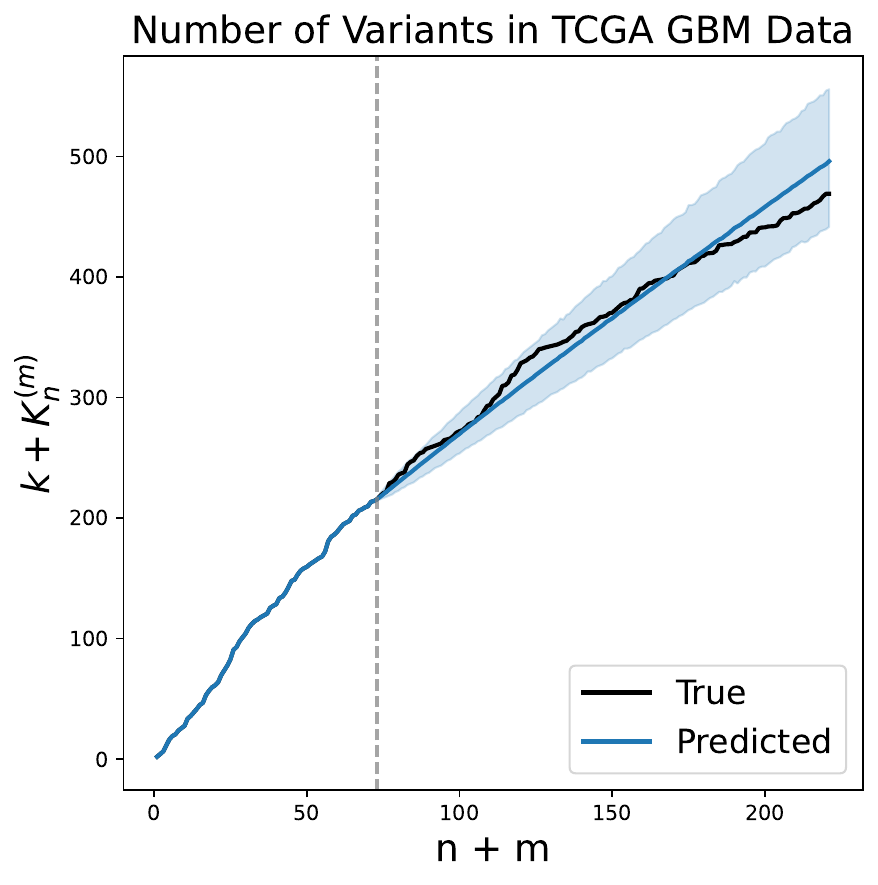}%
    \includegraphics[width=0.33\linewidth]{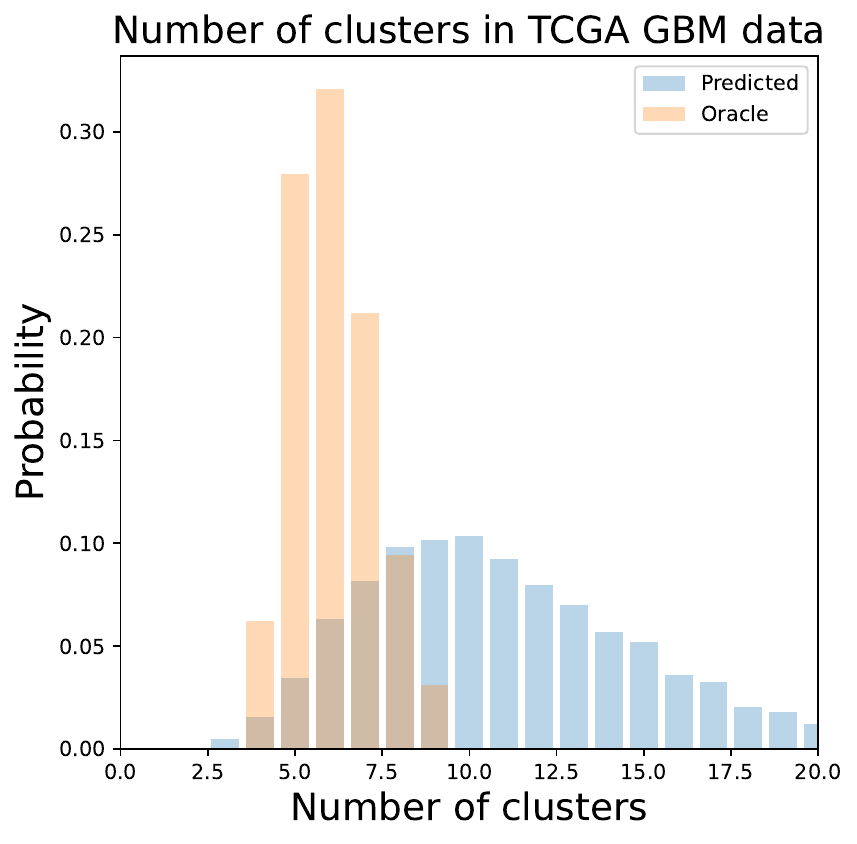}%
    \caption{(Left panel) Projection on the first two principal components of the embeddings variants. (Middle panel)  Prediction of $k + K_m^{(n)}$ as a function of $m$. (Right panel)  Posterior distribution of $C_m^{(n)}$ with $m = 221 - n$ (light blue), and ``oracle'' distribution obtained by fitting the model on the whole dataset (orange).}
    \label{fig:tcga_data}
\end{figure}

We apply our extended Gaussian-\textsc{ibp} model to this data, where $\X = \R^5$. We fix $(m_0, \lambda_0, a, b)$ as in \cite{fraley2007bayesian} and assume the following hyperpriors for all the remaining hyperparameters $
\tau_0 \sim \mathrm{Gamma}(2, 1), b_0 \sim \mathrm{Gamma}(1, 1), \alpha \sim \mathrm{Beta}(2, 2), c \sim \mathrm{Gamma}(2, 1), \beta \sim \mathrm{Gamma}(2, 1)$.

As a first check, we consider the problem of extrapolating the number of new genomic variants $K_m^{(n)}$ that would be found in an additional sample of size $m$. This is the original problem tackled by \cite{Mas(22)} and \cite{Cam(23)}. We use the first $n=73$ subjects as the sample $\bm Z$, on which we obtain the posterior inference via the \textsc{mcmc} algorithm. Then for $m=1, \ldots, 221 - n$ we form the predictions of $K_m^{(n)}$ via \Cref{prop:sncp_Kmn}.
Figure \ref{fig:tcga_data} (middle panel) shows the prediction of $k + K_m^{(n)}$ as a function of $m$. Clearly, our model correctly captures the power-law growth of $k + K_m^{(n)}$.
Similar results, although not shown here, are found with the methods in \cite{Mas(22)} and \cite{Cam(23)}.
More interestingly, our model allows us to form predictions for the number of clusters.
To this end, we evaluate the posterior distribution of $C_n + C_m^{(n)}$ with $m = 221 - n$ via \Cref{prop:sncp_Cmn}, and we compare to the ``oracle'' distribution of $C_{n+m}$ that is obtained by fitting the model on the whole dataset (Figure \ref{fig:tcga_data}, right panel).
Clearly, the predictions for the number of clusters in the additional sample agree with the oracle distribution, albeit with a slightly larger variance and a heavier right tail.
This result validates the use of the extended Gaussian-\textsc{ibp} for modeling these data.

Finally, we focus on the clustering of the genomic variants discovered by our model. To this end, we consider the full dataset, corresponding to the ``oracle'' clustering in the discussion above.  We summarize the clustering via the \texttt{salso} R package, obtaining a point estimate for the partition consisting of 5 clusters of sizes 275, 152, 39, 1, and 2, respectively. Focusing on the three largest clusters, we can appreciate how they identify variants with clearly different profiles in terms of predicted impact and genomic context, see Table \ref{tab:tcga}. 
In particular, we report in Table \ref{tab:tcga} the median value of a widely used “deleteriousness” score that aggregates many annotations into a single metric, with higher values suggesting more functionally disruptive changes \citep[CADD,][]{kircher2014framework};  the median population allele frequency (AF), a proxy for functional constraint (Conservation median) which measures how strongly the affected positions are preserved across species;
the fraction of variants overlapping so-called ``super-enhancers'', i.e., large clusters of regulatory elements controlling key genes \citep{hnisz2013superenhancers}, and the fraction overlapping known transcription-factor binding sites (TFBS).
The first cluster combines very rare variants (low AF), high CADD and high conservation in well-known cancer genes (TP53, EGFR, PTEN), consistent with functionally important, positively selected changes. 
The second cluster also targets coding regions but has much lower CADD and conservation and is enriched in long, repetitive genes (e.g., mucins), a pattern usually associated with passenger burden and technical confounding rather than strong functional effects \citep{lawrence2013mutational}. Cluster 3 has the highest CADD and conservation and is enriched for tumor-suppressor genes such as PTEN and NF1, with variants that are again rare but somewhat more recurrent across patients.

\begin{table}[htbp]
  \centering
  \begin{tabular}{c c c c c c c c}
    \toprule
     $n_{\text{variants}}$ & Top Genes &
    CADD & AF  & Conservation &
    Super-enh. & TFBS \\
    \midrule
     275 &
    \makecell{TP53, EGFR, PTEN \\ MMP3, CABP1} &
    25.6 & $1.40\times 10^{-5}$ & 0.963 &
    0.291 & 0.844 \\
     152 &
    \makecell{FLG, KLK6, SIGLEC10 \\ THSD7B, MUC3A, MUC17} &
    9.13 & $3.49\times 10^{-5}$ & 0.137 &
    0.184 & 0.697 \\
     39 &
    \makecell{PTEN, NF1, PTER \\ CASP1, ANO3} &
    36.0 & $1.08\times 10^{-5}$ & 0.986 &
    0.256 & 0.795 \\
    \bottomrule
  \end{tabular}
  \caption{Summary statistics for the three largest variant clusters.}
  \label{tab:tcga}
\end{table}

\section{Leveraging spatial regularity for feature discovery via repulsive priors}\label{sec:dpp_model_and_application}

Next, we consider a setting in which feature labels are expected to exhibit a repulsive dependence, again lying outside the regime characterized in \Cref{sec:predictive_characterization}.
In this section, the feature labels are tree locations recorded across repeated forest surveys. The inferential goals are twofold: to extrapolate how many trees remain unobserved and to identify where those missing trees are likely to be located. The second goal is outside the scope of standard feature allocation models, where labels are treated only as identifiers rather than spatial objects.
In this setting, spatial regularity, i.e. separation, is a natural assumption, since mature trees tend to spread throughout the forest to ensure adequate space for growth.
We therefore model the labels $X_j$'s via a repulsive point process. In \Cref{sec:ex_dpp} we develop an extended feature model based on determinantal point processes, which allows for feature discovery while capturing repulsive dependence among feature labels.  
We then present an ecological illustration in \Cref{sec:application_spatial} and \Cref{sec:spruces}.

\subsection{The independently marked (repulsive) determinantal process prior}\label{sec:ex_dpp}

To specify the dependence structure among the labels in our extended feature model, we assume that the feature labels $X_j$ arise as points of a determinantal point process (\textsc{dpp})   $\xi$ on a compact region $R \subset \R^d$ \citep{Hou(06), Lav(15)}. The \textsc{dpp} $\xi$  is specified by a covariance kernel $C: R \times R \rightarrow \mathbb C$, such that $M_{\xi}^{(k)}$ has density with respect to the $k$-fold product of the Lebesgue measure given by
\[
    \eta^{(k)}(x_1, \ldots, x_k) = \det\{C(x_h, x_w)_{h,w = 1,\ldots,k}\}, \qquad x_1, \ldots, x_k \in R,
\]
where $C(x_h, x_w)_{h,w = 1,\ldots,k}$ is the $k \times k$ matrix with entries $C(x_h, x_w)$. The main point process $\Psi$ in \eqref{eq:mu_definition} is obtained by associating to each feature label $X_j$ of $\xi$ its corresponding feature probability $S_j$ drawn independently as ${S}_j \mid {X}_j = x_j \sim H(\cdot \mid x_j)$, that is $\Psi$ is an \emph{independently marked point process with ground point process $\xi$ and mark kernel $H$}, according to \cite{BaBlaKa}.
We write for convenience $\Psi \sim \mathrm{imDPP}(C, H)$, where $C$ is the covariance kernel of the determinantal ground process and $H$ is the mark probability kernel.
As for any independently marked process, the $k$-th factorial moment measure equals $M_{\Psi}^{(k)}(\dd \bm x \, \dd \bm s) = M_{\xi}^{(k)}(\dd \bm x) \prod_{\ell=1}^k H(\dd s_\ell \mid x_\ell)$ and it is $\sigma$-finite, since $M_{\xi}^{(k)}$ is $\sigma$-finite. 

We also recall that the reduced Palm version of a \textsc{dpp} is still a \textsc{dpp}. Namely, for any $\bm x = (x_1, \ldots, x_k)$ such that the $x_j$'s are distinct, $\xi^!_{\bm x}$ is a \textsc{dpp} with kernel $K_{\bm x}(y_1,y_2) = C(y_1,y_2) - \tilde{c}_{\bm x}(y_1)^T \bm{\tilde{C}}^{-1} \tilde{c}_{\bm x}(y_2)$, for any $y_1,y_2 \in R$, where $\tilde{c}_{\bm x}(y) = ( C(y, x_1), \ldots, C(y, x_k))^T$ and $\bm{\tilde{C}} = C(x_h, x_w)_{h,w=1,\ldots,k}$. See, e.g., \cite{LavRub23} for further details. Therefore, the Bayesian analysis of the proposed construction easily follows from the results of Section \ref{sec:bnp}, and we can state the following.
\begin{corollary}[Bayesian analysis under the independently marked \textsc{dpp}] \label{thm:bayesian_dpp}
Consider a sample $\bm Z$ from the statistical model \eqref{eq:representation_theorem}, where $\mu$ is the functional of an independently marked point process $\Psi$, i.e., $\Psi \sim \mathrm{imDPP}(C, H)$, defined via \eqref{eq:mu_definition}. 
    \begin{enumerate}
        \item[(i)] The marginal distribution of the sample $\bm Z$ equals 
        \[
            \mathcal{L}_{\xi^!_{\bm x^*}} \left[ -\log\left\{\int_{(0,1]} (1-s)^n H(\dd s\mid x) \right\} \right] \prod_{\ell=1}^k \int_{(0,1]} s^{m_\ell} (1 - s)^{n - m_\ell} H(\dd s \mid x^*_\ell) \cdot M^{(k)}_\xi(\dd \bm x^*),
        \]
       where we remind that $ \mathcal{L}_{\xi^!_{\bm x^*}}$ denotes the Laplace functional of the \textsc{dpp} $\xi^!_{\bm x^*}$.
        \item[(ii)] The posterior distribution of $\mu$ satisfies the distributional equality in \eqref{eq:post_mu_sum}, where the weights $S^*_\ell$'s are independent random variables, further independent of $\mu^\prime$, with marginal density $f_{S^*_\ell}(\dd s) \propto s^{m_\ell} (1 - s)^{n - m_\ell} H(\dd s\mid x^*_\ell)$, as $\ell=1, \ldots , k$. 
        Moreover, $\mu^\prime$ in \eqref{eq:post_mu_sum} can be represented as $\mu^\prime = \sum_{j = 1}^{M^\prime} S^\prime_j \delta_{X^\prime_j}$, where the $X^\prime_j$'s are the atoms of a point process $\xi^\prime$ on $\X$ specified by the  Laplace functional        \begin{equation}\label{eq:laplace_xiprime}
        \mathcal L_{\xi^\prime}(f) = \frac{\mathcal L_{\xi^!_{\bm x^*}} \left\{ f - \log \int_{(0,1]} (1-s)^n H(\dd s\mid x) \right\}}{\mathcal L_{\xi^!_{\bm x^*}}\left\{ -\log\int_{(0,1]} (1-s)^n H(\dd s\mid x)  \right\}},    
        \end{equation}
        and the $S^\prime_j$'s are independent marks with conditional density $S^\prime_j \mid X^\prime_j = x^\prime_j \sim H^\prime(\, \cdot\,  \mid x^\prime_j)$, where $ H^\prime(\dd s \mid x^\prime_j) \propto (1 -s)^n H(\dd s \mid x^\prime_j)$.

        \item[(iii)] The predictive distribution of $Z_{n+1}$, given the sample $\bm Z$, satisfies the distributional equality in 
        \eqref{eq:pred_sum}, where $Z^\prime_{n+1}$ is a Bernoulli process with parameter $\mu^\prime$, and the $A_{n+1, \ell}^*$'s are independent Bernoulli random variables with parameters $S^*_\ell$'s, as $\ell = 1,\ldots, k$.
    \end{enumerate}
\end{corollary}

\Cref{thm:bayesian_dpp} shows where repulsion enters the posterior. For Poisson, mixed Poisson and mixed binomial priors, the unseen-feature component cannot use the observed locations; see also \Cref{app:examples_bayesian_analysis}. Under the \textsc{dpp} prior, the reduced Palm distribution of the ground process depends on the observed configuration $\bm x^*$. Unseen labels are therefore repelled by observed labels. The model can then estimate not only how many trees are unseen, but also where they may be located. Accordingly, $Z_{n+1}^\prime$ depends on the sample $\bm Z$ through $n$, $k$ and the feature labels $\bm x^*$, but not through the frequency counts.
We conclude with an example, where we focus on a specific choice of the distribution of the $S_j$'s.
\begin{example} \label{cor:dpp_beta}
    We consider $\Psi \sim \mathrm{imDPP}(C, H)$, where $H(\cdot \mid x) $ does not depend on $x$ and equals the beta distribution with parameters $(a,b)$, namely $S_j \iid \mathrm{Beta}(a, b)$, where $\mathrm{Beta}(a, b)$ denotes the beta distribution.
    In this particular case, the distributional results in \Cref{thm:bayesian_dpp} simplify. It is worth noticing that the posterior distribution of $\mu$ satisfies the distributional equality in \eqref{eq:post_mu_sum}, where each weight $S^*_\ell$ has a beta distribution 
    with parameters $(m_\ell +a,n-m_\ell+b)$, as $\ell=1, \ldots, k$. Moreover, the distribution of the point process $\xi^\prime$, and, as a consequence, of $\mu^\prime$ in \eqref{eq:post_mu_sum}, becomes much more manageable.  Indeed, from \eqref{eq:laplace_xiprime}, the distribution of $\xi^\prime$ has a density with respect to the distribution of $\xi^!_{\bm x^*}$ given by
    \[
           f_{\xi^\prime}(\nu) \propto \left\{\frac{B(a, b+n)}{B(a, b)}\right\}^{\nu(\X)} =: g(n; a, b)^{\nu(\X)}, 
   \]
   having denoted by $B (a,b)$ the Euler's beta function. The associated marks $S^\prime_j$'s are i.i.d. from a beta distribution with parameters $(a, b + n)$.
   In more detail, the distribution of the number of points in $\xi^\prime$,  denoted with $M^\prime$, has a density with respect to the distribution of $\xi^!_{\bm x^*}(\X)$ given by    \begin{equation}\label{eq:m_prime_dpp}
            f_{M^\prime}(m) \propto \left\{\frac{B(a, b+n)}{B(a, b)}\right\}^m.
        \end{equation}
        Finally, the mean measure $M_{\xi^\prime}$ has density  with respect to $M_{\xi^!_{\bm x^*}}$ defined as 
        \[
        m_{\xi^\prime} (y):= g(n; a, b) \E \left[ f_{\xi^\prime}\left\{ \xi^!_{(\bm x^*, y)} \right\} \right].
        \]
\end{example}

\subsection{Application to spatial statistics}\label{sec:application_spatial}

We revisit a classical problem proposed by \citet{ord78}, see also \cite{diggle2013statistical}.
Consider a forest containing an unobserved point configuration of trees. Each observation $Z_i$ is a partial survey, namely the set of tree locations recorded by the $i$-th surveyor. The original formulation in \citet{ord78} uses a single survey and focuses on the total number of trees. Here we allow repeated surveys and ask a richer question: how many trees remain unobserved, and where are they likely to stand?

Formally, we assume that $(Z_i)_{i\geq 1}$ follows \eqref{eq:representation_theorem}, where $\mu = \sum_{j \geq 1} S_j \delta_{X_j}$, and the $X_j$'s correspond to the locations of the trees in the forest, whereas the $S_j$'s are the tree-specific probabilities of observing tree $j$ in any survey.
We denote the process of all the trees by $\xi = \sum_{j \geq 1} \delta_{X_j}$.
To encode inhibition, we assume that $\xi$ is a \textsc{dpp} on a rectangular region $R \subset \mathbb R^2$, while, for simplicity, we let the $S_j$'s be i.i.d. beta random variables with common parameters $(a, b)$, as in \Cref{cor:dpp_beta}. Indeed, it is well understood that trees often exhibit repulsive behavior. 
For illustration, we assume that $\xi$ follows a Gaussian \textsc{dpp} \citep{Lav(15)} with parameters $(\rho, \alpha)$, subject to the condition $\rho < (\pi \alpha^2)^{-1}$ to ensure the process is well-defined. The covariance kernel of this Gaussian \textsc{dpp} is given by $C(x, y) = \rho \exp\{-\| (x - y) / \alpha \|^2 \}$. We refer to  \cite{MoWaBook03} for other examples of repulsive point processes.

Supposing we observe $n$ surveys $Z_1, \ldots, Z_n$, which together reveal the locations of $k$ distinct trees, we now address the problem of predicting the number and locations of the missing trees.
The posterior distribution of the total number of trees is equal to $M^\prime + k$, where the law of $M^\prime$ is given in \Cref{cor:dpp_beta}. In addition to estimating the number of trees in the forest, a natural and more challenging question is to locate the unobserved trees. 
With the notation of  \Cref{cor:dpp_beta}, the infinitesimal probability that an unobserved tree would occupy position $\dd x$ equals  $\E\{\Psi^\prime(\dd x \times (0, 1])\} = M_{\xi^\prime} (\dd x)$. Addressing both the count and location problems requires handling the distribution of $\xi^!_{\bm x}(\mathbb{X})$ for some suitable configurations $\bm x$; Appendix \ref{app:fitting_dpp_model} reports a detailed discussion on this point and related computational aspects.
We select the hyperparameters $(a,b,\rho,\alpha)$ by empirical Bayes, maximizing the marginal likelihood in point (i) of \Cref{thm:bayesian_dpp}. Appendix \ref{app:synthetic_dpp_model} reports synthetic experiments that illustrate the main behavior of the model.

\subsection{An illustration with Norwegian spruces}\label{sec:spruces}

\begin{figure}[]
    \centering
    \includegraphics[width=\linewidth]{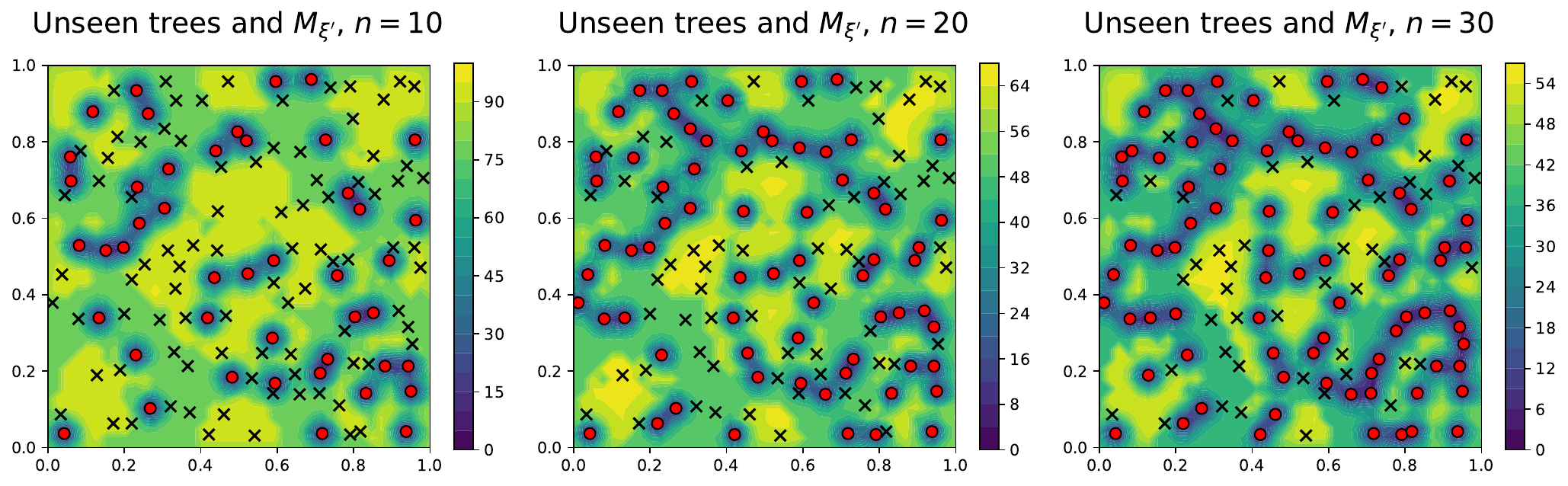}
    
    \caption{Locating the unobserved trees for $n \in \{10,20,30\}$ in the analysis of the spruces dataset of \Cref{sec:spruces}: infinitesimal probability of observing an unseen tree in a given location. The three plots report $M_{\xi^\prime}$ for the three sample sizes. The red dots represent the observed trees in the sample. The black crosses indicate the unseen trees. Note that the plots have different color scales.}
    \label{fig:mean_measure_application}
\end{figure}

We analyze the \texttt{spruces} dataset from the R package \texttt{spatstat}, which contains the spatial locations of 134 Norwegian spruce trees in a natural forest stand in Saxony, Germany.  We treat this full point pattern, denoted by $\xi_0$, as the ground-truth forest and generate repeated partial surveys from it. Specifically, each tree is assigned an i.i.d. detection probability $S_j \sim \mathrm{Beta}(1,20)$, and survey $i$ records tree $j$ independently with probability $S_j$. This semi-synthetic design is useful because it lets us assess both how well the model extrapolates the total number of trees and how accurately it localizes those not yet observed.

We consider samples of increasing sizes $n \in \{10, 20, 30\}$ and estimate model hyperparameters using the empirical Bayes approach.
We infer the total number of trees as $M^\prime + k$, where $M^\prime$ is defined in \Cref{cor:dpp_beta}.  Employing Le Cam's approximation (see Appendix \ref{app:fitting_dpp_model}), our model provides reliable predictions, although with a slight underestimation. Specifically, the expected values of $M^\prime + k$ are respectively $115, 111, 114$, for the increasing sample sizes. 
This systematic underestimation aligns with the comments of Figure \ref{fig:ntree_simulation}, where exact computations yielded more accurate results, but Le Cam's approximation is substantially faster.

The second target is the location of the missing trees. In our framework, this is summarized via $M_{\xi^\prime}$, as detailed in Section \ref{app:fitting_dpp_model}, which is reported in Figure \ref{fig:mean_measure_application}. Across all sample sizes, our predictions exhibit the peculiar repulsive structure highlighted in the simulated example. The regions predicted to most likely contain unseen trees align well with the locations of the actual unobserved trees, as indicated by the black crosses in the figure. The difference in plot scales arises from the reduction in the number of unseen trees as the sample size increases.
Crucially, the predicted locations are informed by the configuration of observed trees through the reduced Palm structure of \Cref{thm:bayesian_dpp}: an inferential gain that standard feature allocation models, treating labels only as identifiers, cannot deliver.

\section{Discussion}\label{sec:discussion}

We have developed a unified Bayesian framework for extended feature allocation models, in which feature labels and feature probabilities are modeled jointly. This yields a complete Bayesian analysis in which labels are themselves objects of inference. The sufficientness postulates of Section~\ref{sec:predictive_characterization} characterize the priors whose predictive distributions depend only on  $n$ or on $(n, K_n)$, providing both practical guidance for prior elicitation in the irrelevant-labels regime and a formal description of the limits of such standard models. These results also offer a bridge between the classical Bayesian paradigm, based on the specification of a prior and a likelihood, and the predictive approach to inference recently revived by \citet{Fong2023, Berti23}, in which the statistician directly specifies a system of predictive distributions reflecting subjective beliefs.\\
Beyond these characterizations, the central methodological contribution of the paper is to allow information encoded in feature labels to directly inform prediction. In particular, the cluster Cox process and determinantal point process models developed in Sections~\ref{sec:sncp_model_and_application} and \ref{sec:dpp_model_and_application} illustrate this idea in settings involving clustering and repulsion among labels, respectively. More generally, the proposed framework extends naturally to problems involving latent features, such as image segmentation, latent factor models, and related constructions, where structured priors over features may yield more interpretable and realistic inference. In such settings, the Bayesian theory developed here provides the foundation for posterior inference.

Our framework suggests several directions for future research.
First, the Cox process construction of Section~\ref{sec:sncp_model_and_application} can be adapted to settings with hierarchical groupings, for instance, biological taxa nested by genus and family, yielding enriched feature allocation models analogous to those developed in the mixture setting by \citet{wade2014improving}. 
Second,  \textsc{dpp}-based construction of Section~\ref{sec:dpp_model_and_application} could be extended to high-dimensional settings, making it suitable as a latent structure in factor models, following the same lines of \cite{ghilotti2025DPP}. 
Beyond extensions of these two models, other broader generalizations appear promising. 
Indeed, we envision an extension to trait allocations \citep{Campbell(18)}, where each feature is endowed with an expression level: we expect that our theorems extend to the trait allocation setting by virtue of our Palm-calculus-based framework and by adopting the spike-and-slab formulation of trait allocations in \cite{Jam(17)}.
We also plan to explore more complex models for partially exchangeable data, both generalizing the hierarchical beta process of \cite{Thi(07)} \citep[further developed by][]{james2023bayesian} and proposing alternative prior distributions inspired by the rich literature on partially exchangeable priors for species sampling models.
For instance, one could develop analogous of nested \citep{rodriguez_ndp}, additive \citep{nipoti_add} and compound \citep{griffin_compound} processes, as well as draw inspiration from the more general constructions in \cite{ascolani_24, franzolini_23, beraha_griffin} 
to design new models in the feature setting. Work on these problems and related ones is ongoing.

\section*{Acknowledgement}
FC is supported by the European Union – Next Generation EU funds, component M4C2, investment 1.1., PRIN-PNRR 2022 (P2022H5WZ9). LG gratefully acknowledges support from the Office of Naval Research (N00014-24-1-2626-P00002) and NIGMS through grant R01GM163225. The authors would like to thank Andrea Gilardi (University of Milano-Bicocca) and Jesper M\o ller (Aalborg University) for helpful discussions and suggestions.

\begin{footnotesize}
%
\begin{singlespace}
\bibliography{references}

@article{ayed2021nonnegative,
  title={Nonnegative Bayesian nonparametric factor models with completely random measures},
  author={Ayed, Fadhel and Caron, Fran{\c{c}}ois},
  journal={Statistics and Computing},
  volume={31},
  pages={1--24},
  year={2021},
  publisher={Springer}
}

@article{ghilotti2025DPP,
    author = {Ghilotti, L and Beraha, M and Guglielmi, A},
    title = {Bayesian clustering of high-dimensional data via latent repulsive mixtures},
    journal = {Biometrika},
    volume = {112},
    number = {2},
    pages = {asae059},
    year = {2025},
    issn = {1464-3510},
    doi = {10.1093/biomet/asae059}
}

@article{BaBlaKa,
  title={Random measures, point processes, and stochastic geometry},
  author={Baccelli, Fran{\c{c}}ois and B{\l}aszczyszyn, Bart{\l}omiej and Karray, Mohamed},
  year={2020},
  publisher={Inria},
  journal={HAL preprint available at https://hal.inria.fr/hal-02460214/}
}

@article{james2025poisson,
  title={Poisson Hierarchical Indian Buffet Processes for Within and Across Group Sharing of Latent Features-With Indications for Microbiome Species Sampling Models},
  author={James, Lancelot F and Lee, Juho and Pandey, Abhinav},
  journal={arXiv preprint arXiv:2502.01919},
  year={2025}
}

@article{franzolini_23,
  title={Conditional partial exchangeability: a probabilistic framework for multi-view clustering},
  author={Franzolini, Beatrice and De Iorio, Maria and Eriksson, Johan},
  journal={arXiv preprint arXiv:2307.01152},
  year={2023}
}

@article{beraha_griffin,
  title={Normalised latent measure factor models},
  author={Beraha, Mario and Griffin, Jim E},
  journal={Journal of the Royal Statistical Society Series B: Statistical Methodology},
  volume={85},
  number={4},
  pages={1247--1270},
  year={2023},
  publisher={Oxford University Press US}
}

@article{ascolani_24,
  title={Nonparametric priors with full-range borrowing of information},
  author={Ascolani, Filippo and Franzolini, Beatrice and Lijoi, Antonio and Pr{\"u}nster, Igor},
  journal={Biometrika},
  volume={111},
  number={3},
  pages={945--969},
  year={2024},
  publisher={Oxford University Press}
}

@article{griffin_compound,
  title={Compound random measures and their use in Bayesian non-parametrics},
  author={Griffin, Jim E and Leisen, Fabrizio},
  journal={Journal of the Royal Statistical Society Series B: Statistical Methodology},
  volume={79},
  number={2},
  pages={525--545},
  year={2017},
  publisher={Oxford University Press}
}

@article{nipoti_add,
author = {Antonio Lijoi and Bernardo Nipoti and Igor Pr{\"u}nster},
title = {{Bayesian inference with dependent normalized completely random measures}},
volume = {20},
journal = {Bernoulli},
number = {3},
publisher = {Bernoulli Society for Mathematical Statistics and Probability},
pages = {1260 -- 1291},
year = {2014},
doi = {10.3150/13-BEJ521},
URL = {https://doi.org/10.3150/13-BEJ521}
}

@article{rodriguez_ndp,
  title={The nested Dirichlet process},
  author={Rodriguez, Abel and Dunson, David B and Gelfand, Alan E},
  journal={Journal of the American statistical Association},
  volume={103},
  number={483},
  pages={1131--1154},
  year={2008},
  publisher={Taylor \& Francis}
}

@incollection{Pit96,
	author = {Pitman, Jim},
	booktitle = {Statistics, probability and game theory},
	doi = {10.1214/lnms/1215453576},
	mrclass = {60C05},
	mrnumber = {1481784},
	mrreviewer = {Hajime Yamato},
	pages = {245--267},
	publisher = {Inst. Math. Statist., Hayward, CA},
	series = {IMS Lecture Notes Monogr. Ser.},
	title = {Some developments of the {B}lackwell-{M}ac{Q}ueen urn scheme},
	url = {https://doi.org/10.1214/lnms/1215453576},
	volume = {30},
	year = {1996}}

@article{DeBlasi,
  title={Are Gibbs-type priors the most natural generalization of the Dirichlet process?},
  author={De Blasi, Pierpaolo and Favaro, Stefano and Lijoi, Antonio and Mena, Rams{\'e}s H and Pr{\"u}nster, Igor and Ruggiero, Matteo},
  journal={IEEE transactions on pattern analysis and machine intelligence},
  volume={37},
  number={2},
  pages={212--229},
  year={2015},
  publisher={IEEE}
}

@article{broderick_exponential,
	author = {Tamara Broderick and Ashia C. Wilson and Michael I. Jordan},
	journal = {Bernoulli},
	number = {4B},
	pages = {3181 -- 3221},
	title = {{Posteriors, conjugacy, and exponential families for completely random measures}},
	volume = {24},
	year = {2018}}

@inproceedings{Gha05,
	author = {Ghahramani, Zoubin and Griffiths, Thomas},
	booktitle = {Advances in Neural Information Processing Systems},
	editor = {Y. Weiss and B. Sch\"{o}lkopf and J. Platt},
	publisher = {MIT Press},
	title = {Infinite latent feature models and the Indian buffet process},
	volume = {18},
	year = {2005}}

@article{james2023bayesian,
  title={{Bayesian analysis of generalized hierarchical Indian buffet processes for within and across group sharing of latent features}},
  author={James, Lancelot Fitzgerald and Lee, Juho and Pandey, Abhinav},
  journal={arXiv preprint arXiv:2304.05244},
  year={2023}
}

@article{ord78,
    author = {Ord, Keith},
    title = {How Many Trees in a Forest? },
    journal = {The Mathematical Scientist},
    year = {1978},
    volume = 3,
    pages = {23--33}
}

@book{Kallenberg17,
  title={Random measures, theory and applications},
  author={Kallenberg, Olav},
  volume={1},
  year={2017},
  publisher={Springer}
}

@book{Kallenberg21,
  title={Foundations of Modern Probability},
  author={Kallenberg, Olav},
  year={2021},
  publisher={Springer}
}

@article{zhou2016priors,
  title={Priors for random count matrices derived from a family of negative binomial processes},
  author={Zhou, Mingyuan and Padilla, Oscar Hernan Madrid and Scott, James G},
  journal={Journal of the American Statistical Association},
  volume={111},
  number={515},
  pages={1144--1156},
  year={2016},
  publisher={Taylor \& Francis}
}

@article{good1956number,
  title={The number of new species, and the increase in population coverage, when a sample is increased},
  author={Good, Irving J and Toulmin, George H},
  journal={Biometrika},
  volume={43},
  number={1-2},
  pages={45--63},
  year={1956},
  publisher={Oxford University Press}
}

@article{battiston18,
	author = {Marco Battiston and Stefano Favaro and Daniel M. Roy and Yee Whye Teh},
	journal = {The Annals of Applied Probability},
	number = {3},
	pages = {1423 -- 1448},
	title = {{A characterization of product-form exchangeable feature probability functions}},
	volume = {28},
	year = {2018}}

@article{favaro2012,
	author = {Stefano Favaro and Antonio Lijoi and Igor Pr{\"u}nster},
	journal = {Bernoulli},
	number = {4},
	pages = {1267 -- 1283},
	title = {{Asymptotics for a Bayesian nonparametric estimator of species variety}},
	volume = {18},
	year = {2012}}

@article{lijoi2007bayesian,
  title={Bayesian nonparametric estimation of the probability of discovering new species},
  author={Lijoi, Antonio and Mena, Rams{\'e}s H and Pr{\"u}nster, Igor},
  journal={Biometrika},
  volume={94},
  number={4},
  pages={769--786},
  year={2007},
  publisher={Oxford University Press}
}

@article{Zab82,
	author = {Sandy L. Zabell},
	journal = {The Annals of Statistics},
	number = {4},
	pages = {1090 -- 1099},
	title = {{W. E. Johnson's "Sufficientness" Postulate}},
	volume = {10},
	year = {1982}}

@article{orlitsky2016optimal,
  title={Optimal prediction of the number of unseen species},
  author={Orlitsky, Alon and Suresh, Ananda Theertha and Wu, Yihong},
  journal={Proceedings of the National Academy of Sciences},
  volume={113},
  number={47},
  pages={13283--13288},
  year={2016},
  publisher={National Acad Sciences}
}

@book {Zab05,
    AUTHOR = {Zabell, S. L.},
     TITLE = {Symmetry and its discontents},
    SERIES = {Cambridge Studies in Probability, Induction, and Decision
              Theory},
      NOTE = {Essays on the history of inductive probability,
              With a preface by Brian Skyrms},
 PUBLISHER = {Cambridge University Press, New York},
      YEAR = {2005},
     PAGES = {xii+279},
      ISBN = {978-0-521-44912-0; 0-521-44912-X},
   MRCLASS = {01A05 (01A50 01A55 60-03 62-03)},
  MRNUMBER = {2199124},
MRREVIEWER = {Michael\ A. B. Deakin},
       DOI = {10.1017/CBO9780511614293},
       URL = {https://doi.org/10.1017/CBO9780511614293},
}

@article {Berti23,
    AUTHOR = {Berti, Patrizia and Dreassi, Emanuela and Leisen, Fabrizio and
              Pratelli, Luca and Rigo, Pietro},
     TITLE = {A {P}robabilistic {V}iew on {P}redictive {C}onstructions for
              {B}ayesian {L}earning},
  JOURNAL = {Statistical Science. A Review Journal of the Institute of
              Mathematical Statistics},
    VOLUME = {40},
      YEAR = {2025},
    NUMBER = {1},
     PAGES = {25--39},
      ISSN = {0883-4237,2168-8745},
   MRCLASS = {99-01},
  MRNUMBER = {4858638},
       DOI = {10.1214/23-STS884},
       URL = {https://doi.org/10.1214/23-STS884},
}

@book{MoWaBook03,
  title={Statistical inference and simulation for spatial point processes},
  author={M{\o}ller, Jesper and Waagepetersen, Rasmus Plenge},
  year={2003},
  publisher={CRC Press}
}

@article{LavRub23,
title = "On simulation of continuous determinantal point processes",
author = "Fr{\'e}d{\'e}ric Lavancier and Ege Rubak",
note = "Publisher Copyright: {\textcopyright} 2023, The Author(s).",
year = "2023",
month = oct,
doi = "10.1007/s11222-023-10272-w",
language = "English",
volume = "33",
journal = "Statistics and Computing",
issn = "0960-3174",
publisher = "Springer",
number = "45",
}

@article{ghilotti2024bayesian,
    author = {Ghilotti, Lorenzo and Camerlenghi, Federico and Rigon, Tommaso},
    title = {Bayesian analysis of product feature allocation models},
    journal = {Journal of the Royal Statistical Society Series B: Statistical Methodology},
    volume = {88},
    number = {2},
    pages = {540-566},
    year = {2026},
    issn = {1369-7412},
    doi = {10.1093/jrsssb/qkaf058}
}

@article{Ber26_Palm,
    author = {Beraha, M and Camerlenghi, F and Ghilotti, L},
    title = {Palm distributions of superposed point processes for statistical inference},
    journal = {Biometrika},
    pages = {asag021},
    year = {2026},
    issn = {1464-3510},
    doi = {10.1093/biomet/asag021},
    url = {https://doi.org/10.1093/biomet/asag021},
    eprint = {https://academic.oup.com/biomet/advance-article-pdf/doi/10.1093/biomet/asag021/67637136/asag021.pdf},
}

@book{diggle2013statistical,
  title={Statistical analysis of spatial and spatio-temporal point patterns},
  author={Diggle, Peter J},
  year={2013},
  publisher={CRC press}
}

@article{steele1994cam,
  title={Le Cam's inequality and Poisson approximations},
  author={Steele, J Michael},
  journal={The American Mathematical Monthly},
  volume={101},
  number={1},
  pages={48--54},
  year={1994},
  publisher={Taylor \& Francis}
}

@inproceedings{Teh09,
	author = {Teh, Yee and Gorur, Dilan},
	booktitle = {Advances in Neural Information Processing Systems},
	editor = {Y. Bengio and D. Schuurmans and J. Lafferty and C. Williams and A. Culotta},
	title = {Indian buffet processes with power-law behavior},
	volume = {22},
	year = {2009}}

@inproceedings{PeRaDu12,
	author = {Petralia, Francesca and Rao, Vinayak and Dunson, David},
	booktitle = {Advances in Neural Information Processing Systems},
	editor = {F. Pereira and C.J. Burges and L. Bottou and K.Q. Weinberger},
	publisher = {Curran Associates, Inc.},
	title = {Repulsive Mixtures},
	volume = {25},
	year = {2012},
}

@article{fraley2007bayesian,
  title={Bayesian regularization for normal mixture estimation and model-based clustering},
  author={Fraley, Chris and Raftery, Adrian E},
  journal={Journal of classification},
  volume={24},
  number={2},
  pages={155--181},
  year={2007},
  publisher={Springer}
}

@article{beraha2025online,
  title={Online activity prediction via generalized Indian buffet process models},
  author={Beraha, Mario and Masoero, Lorenzo and Favaro, Stefano and Richardson, Thomas S},
  journal={arXiv preprint arXiv:2505.19643},
  year={2025}
}

@article{xie2019bayesian,	
  title={Bayesian repulsive gaussian mixture model},	
  author={Xie, Fangzheng and Xu, Yanxun},	
  journal={Journal of the American Statistical Association},	
  pages={187--203},	
  year={2019},	
  publisher={Taylor \& Francis}	
}

@article{cremaschi2023repulsion,
  title={Repulsion, chaos, and equilibrium in mixture models},
  author={Cremaschi, Andrea and Wertz, Timothy M and De Iorio, Maria},
  journal={Journal of the Royal Statistical Society Series B: Statistical Methodology},
  pages={qkae096},
  year={2024},
  publisher={Oxford University Press UK}
}

@article{Moller03, 
    title={{Shot noise Cox processes}},
    volume={35}, 
    DOI={10.1239/aap/1059486821},
    number={3}, 
    journal={Advances in Applied Probability}, 
    author={Møller, Jesper}, 
    year={2003}, 
    pages={614–640}
}

@article{wade2014improving,
  title={Improving prediction from Dirichlet process mixtures via enrichment},
  author={Wade, Sara and Dunson, David B and Petrone, Sonia and Trippa, Lorenzo},
  journal={The Journal of Machine Learning Research},
  volume={15},
  number={1},
  pages={1041--1071},
  year={2014},
  publisher={JMLR. org}
}

@article{beraha2024,
  title={Bayesian mixture models with repulsive and attractive atoms},
  author={Beraha, Mario and Argiento, Raffaele and Camerlenghi, Federico and Guglielmi, Alessandra},
  journal={Journal of the Royal Statistical Society Series B: Statistical Methodology},
  pages={qkaf027},
  year={2025},
  publisher={Oxford University Press UK}
}

@article {Fong2023,
    AUTHOR = {Fong, Edwin and Holmes, Chris and Walker, Stephen G.},
     TITLE = {Martingale posterior distributions},
  JOURNAL = {Journal of the Royal Statistical Society. Series B.
              Statistical Methodology},
    VOLUME = {85},
      YEAR = {2023},
    NUMBER = {5},
     PAGES = {1357--1391},
      ISSN = {1369-7412,1467-9868},
   MRCLASS = {62F15 (60G09 60G42 62D10 62H05)},
  MRNUMBER = {4726953},
}

@article{dunsonstolf24,
    author = {Stolf, F and Dunson, D B},
    title = {Infinite joint species distribution models},
    journal = {Biometrika},
    volume = {112},
    number = {4},
    pages = {asaf055},
    year = {2025},
}

@article{niu2025incorporating,
  title={Incorporating LLM Embeddings for Variation Across the Human Genome},
  author={Niu, Hongqian and Bryan, Jordan and Li, Xihao and Li, Didong},
  journal={arXiv preprint arXiv:2509.20702},
  year={2025}
}

@article{lawrence2013mutational,
  author  = {Lawrence, Michael S. and Stojanov, Petar and Polak, Paz and Kryukov, Gregory V. and Cibulskis, Kristian and Sivachenko, Andrey and Carter, Scott L. and Stewart, Chip and Mermel, Craig H. and Roberts, Steven A. and et al.},
  title   = {Mutational heterogeneity in cancer and the search for new cancer-associated genes},
  journal = {Nature},
  year    = {2013},
  volume  = {499},
  number  = {7457},
  pages   = {214--218},
  doi     = {10.1038/nature12213}
}

@article{kircher2014framework,
  author  = {Kircher, Martin and Witten, Dana M. and Jain, Preti and O'Roak, Brian J. and Cooper, Gregory M. and Shendure, Jay},
  title   = {A general framework for estimating the relative pathogenicity of human genetic variants},
  journal = {Nature Genetics},
  year    = {2014},
  volume  = {46},
  number  = {3},
  pages   = {310--315},
  doi     = {10.1038/ng.2892}
}

@article{hnisz2013superenhancers,
  author  = {Hnisz, Denes and Abraham, Brian J. and Lee, Tong Ihn and Lau, Alexander and Saint-Andr{\'e}, Vincent and Sigova, Alla A. and Hoke, Heather A. and Young, Richard A.},
  title   = {Super-enhancers in the control of cell identity and disease},
  journal = {Cell},
  year    = {2013},
  volume  = {155},
  number  = {4},
  pages   = {934--947},
  doi     = {10.1016/j.cell.2013.09.053}
}

@article{Wang2024,
author = {Yixin Wang and Anthony Degleris and Alex Williams and Scott W. Linderman},
title = {Spatiotemporal Clustering with Neyman-Scott Processes via Connections to Bayesian Nonparametric Mixture Models},
journal = {Journal of the American Statistical Association},
volume = {119},
number = {547},
pages = {2382--2395},
year = {2024},
publisher = {Taylor \& Francis},
doi = {10.1080/01621459.2023.2257896},
}
\end{singlespace}

\end{footnotesize}


\clearpage

\setcounter{equation}{0}
\renewcommand\theequation{S\arabic{equation}}
\renewcommand\theHequation{S\arabic{equation}}

\setcounter{remark}{0}
\renewcommand\theremark{S\arabic{remark}}
\renewcommand\theHremark{S\arabic{remark}}

\setcounter{theorem}{0}
\renewcommand\thetheorem{S\arabic{theorem}}
\renewcommand\theHtheorem{S\arabic{theorem}}

\setcounter{corollary}{0}
\renewcommand\thecorollary{S\arabic{corollary}}
\renewcommand\theHcorollary{S\arabic{corollary}}

\setcounter{proposition}{0}
\renewcommand\theproposition{S\arabic{proposition}}
\renewcommand\theHproposition{S\arabic{proposition}}

\setcounter{lemma}{0}
\renewcommand\thelemma{S\arabic{lemma}}
\renewcommand\theHlemma{S\arabic{lemma}}

\setcounter{figure}{0}
\renewcommand\thefigure{S\arabic{figure}}
\renewcommand\theHfigure{S\arabic{figure}}

\setcounter{table}{0}
\renewcommand\thetable{S\arabic{table}}
\renewcommand\theHtable{S\arabic{table}}

\setcounter{section}{0}
\renewcommand{\thesection}{S\arabic{section}}
\renewcommand{\theHsection}{S\arabic{section}}

\begin{center}
   \LARGE Supplementary material for:\\
   ``Extended feature allocation models''
\end{center}

\section*{Organization of the supplementary material}

The supplementary material is structured as follows. \Cref{app:point_process_definition} provides some mathematical background on point processes. \Cref{app:results_extended_feature} presents useful results on extended feature allocation models, which are referenced in \Cref{rem:finite_features} and throughout the paper. \Cref{app:mb_results} provides additional results on the class of mixed binomial processes. \Cref{app:general_res_palm} recalls a key formula for working with Palm distributions, namely the Campbell-Little-Mecke (\textsc{clm}) formula, as well as a fundamental characterization result for mixed Poisson and mixed binomial processes, as stated in \cite[Theorem 5.3]{Kal(73)}. \Cref{app:proofs_main_theorems} contains the proofs of Theorems \ref{thm:marg} and \ref{thm:post}, which establish the full Bayesian analysis presented in \Cref{sec:general_bayesian_analysis}. \Cref{app:pred_characterization} includes proofs for all results related to the sufficientness postulates discussed in \Cref{sec:predictive_characterization}. 
\Cref{app:sncp_appendix} contains proofs, additional results and details of the posterior sampling algorithm for the Cox process prior model of \Cref{sec:sncp_model_and_application}.
\Cref{app:dpp_prior_model} contains proofs and additional details for the independently marked determinantal process prior of \Cref{sec:dpp_model_and_application}, including discussion about the fitting approach and numerical implementation, and  synthetic simulation studies. In \Cref{app:examples_bayesian_analysis}, we analyze some other notable classes of extended feature models, based on mixed Poisson, and mixed binomial
process priors. The proofs for such results are reported in \Cref{app:examples_computation}.

To facilitate the reading of the supplementary material, we recall the \emph{extended} feature allocation model for the exchangeable sequence of observations $(Z_i)_{i\geq 1}$, presented in model \eqref{eq:representation_theorem} of the main text. In particular, we consider $Z_i = \sum_{j \geq 1} A_{i,j} \delta_{X_j}$, for $i \geq 1$, and the statistical model is given by
\begin{equation}\label{eq:main_model_appendix}
\begin{aligned}
    Z_i \mid \mu & \iid \mathrm{BeP}(\mu),\\
        \mu & \sim \mathscr Q,
\end{aligned}    
\end{equation}
where $\mathscr Q$ denotes the law of the random measure $\mu$, which is a functional of the point process $\Psi = \sum_{j\geq 1} \delta_{(X_j,S_j)}$ on $\X \times (0,1]$. Specifically, $\mu$ is defined as  
\begin{equation}\label{eq:functional_appendix}
    \mu(B) = \int_{\X \times (0,1]} s \indicator_B (x ) \Psi(\dd x \, \dd s), \quad B \in \Xcr.
\end{equation}

\section{Mathematical background on point processes} \label{app:point_process_definition}

Here, we provide the formal definition of a point process.
Indicate by $\X$  a Polish space equipped with the corresponding Borel $\sigma$-algebra $\Xcr$. 
We say that a measure $\nu$ on $(\X, \Xcr)$ is \textit{locally finite} if $\nu(B) < \infty$ for any relative compact set $B \in \Xcr$.
Denote by $(\mathbb M_\X, \mathcal M_\X)$ the space of locally finite counting measures on $(\X, \Xcr)$, equipped with the corresponding Borel $\sigma$-algebra $\mathcal M_\X$.
A point process $\Phi$ on the space $\X$ is a measurable map from an underlying probability space $(\Omega, \Acr, \Pp)$ taking values in $(\mathbb M_\X, \mathcal M_\X)$ so that $\Phi(B)$ is a nonnegative integer almost surely (a.s.), for any relatively compact Borel set $B$.
Let $\plaw_{\Phi}:= \Pp \circ \Phi^{-1}$ denote the probability distribution of  $\Phi$, which is uniquely characterized by the Laplace functional $\mathcal{L}_{\Phi}(f):= \E [\exp \{ - \int_\X f (x) \Phi(\dd x)\}] $, for any measurable function $f: \X \to \R_+$. Notably, any point process $\Phi$ can be represented as $\Phi = \sum_{j\geq 1} \delta_{X_j}$, where $(X_j)_{j\geq 1}$ is a sequence of random variables taking values in $\X$, and  $\delta_{X_j}$ denotes the Dirac delta mass at $X_j$.
In the paper, we deal with \emph{simple} point processes, for which the atoms of $\Phi$ are a.s. distinct, i.e.,   $\Pp (X_i = X_j)=0$ for all $i \not = j$.

To define the \emph{Palm distributions}, we introduce the Campbell measure of $\Phi$ as the measure on $\Xcr \times \mathcal{M}_\X$ defined by $\mathscr C_\Phi(C \times L) : = \E[\Phi(C) \indicator(\Phi \in L) ]$ for $C \in \Xcr$ and $L \in \mathcal M_\X$.  Provided that the mean measure $M_\Phi$ is $\sigma$-finite, the Palm kernel $\{\plaw_{\Phi}^x\}_{x \in \X}$ of $\Phi$ is defined as the (a.s.) unique disintegration probability kernel of $\mathscr C_\Phi$ with respect to $M_\Phi$, i.e., 
\[
    \mathscr C_\Phi(C \times L) = \int_{C} \plaw_{\Phi}^x(L) M_{\Phi}(\dd x).
\]

In a similar fashion, under the assumption that the $k$-th factorial moment measure $M_\Phi^{(k)}$ is $\sigma$-finite, it is possible to construct the family of  $k$-th order Palm distributions $\{\plaw_{\Phi}^{\bm x}\}_{\bm x \in \X^k}$, and the generic probability measure $\plaw_{\Phi}^{\bm x}$ can be interpreted as the distribution of $\Phi$ given that $\bm x = (x_1, \ldots, x_k)$ are atoms of $\Phi$. 
Again, by removing the trivial atoms $(x_1, \ldots , x_k)$, we obtain the reduced Palm distributions $\plaw_{\Phi^!}^{\bm x}$, namely the probability law of
\[
\Phi^!_{\bm x} : = \Phi_{\bm x} -\sum_{j =1}^k \delta_{x_j}.
\]

\section{Some useful results on extended feature allocation models}\label{app:results_extended_feature}

We start by providing some considerations about the common assumption for feature allocation models (which we retain for the extended case) imposing that any subject displays a.s. a finite number of features, i.e., $Z_i(\X) < \infty$ a.s..
\begin{remark}\label{rem:finite_features}
The property $Z_i(\X) < \infty$ a.s.  is always guaranteed when $\Psi$ has a.s. a finite number of points, for example, when $\Psi$ is a mixed binomial process. Some considerations are needed for priors which entail an infinite number of points for $\Psi$. Since $Z_i$ is obtained from thinning $\Psi$ and discarding the second component, it follows that the mean measure of any $Z_i$ is $M_Z(B) = \int_{(0,1]} s M_\Psi(B \times \dd s)$, $B \in \Xcr$. Therefore, a sufficient condition for $Z_i(\X) < \infty$ a.s. stems from $\E [Z_i(\X)] < \infty$, which corresponds to $\int_{\X\times (0,1]} s M_{\Psi}(\dd x \, \dd s) < \infty$. However,  this is not a necessary condition in general. 
It can be proved that  a necessary condition on the mean measure of $\Psi$ to have  $Z_i(\X) < \infty$ a.s. is the following:
\[
\int_{\X\times (0,1]} (1 - s) M_{\Psi}(\dd x \, \dd s) = \infty.\] 
See \Cref{prop:nec_condition} for the formal statement and proof. Notably, if $\Psi$ is a Poisson process with (infinite) mean measure $\nu$ or a mixed Poisson $\mathrm{MP}(\nu, f_\gamma)$, then the condition $\int_{\X\times(0,1]} s \nu(\dd x \, \dd s) < \infty$ is both necessary and sufficient for $Z_i(\X) < \infty$ a.s., as stated in \Cref{prop:nec_suff_poisson}. We defer to \Cref{prop:poisson_details} further implications of the assumption $Z_i(\X) < \infty$ a.s. under Poisson processes. 
\end{remark}

The following proposition formally states the necessary condition, discussed in \Cref{rem:finite_features}, on the mean measure of $\Psi$ to ensure $Z_i(\X)<\infty$ a.s. for any generic element $Z_i$ in the sequence $(Z_i)_{i\geq 1}$ defined in \eqref{eq:main_model_appendix}.

\begin{proposition}\label{prop:nec_condition}
Let $Z_i$ be the generic element of the sequence $(Z_i)_{i\geq1}$ defined in \eqref{eq:main_model_appendix}, where $\mu$ is the functional of a point process $\Psi$ defined via \eqref{eq:functional_appendix}. Assume that $\Psi$ has infinite points a.s. and let $M_{\Psi}$ denote the mean measure of $\Psi$. If $Z_i(\X) < \infty$ a.s., then
\[
\int_{\X\times (0,1]} (1 - s) M_{\Psi}(\dd x \, \dd s) = \infty.\]
\end{proposition}
\begin{proof}
We proceed by proving that $\int_{\X\times (0,1]} (1 - s) M_{\Psi}(\dd x \, \dd s) < \infty$ implies $\Pp(Z_i(\X) < \infty) = 0$. Indeed, define the sequence of events $V_j = \{A_{i,j} = 0\} \in \Acr$, $j \geq 1$ and observe that the event $Z_i(\X) < \infty$ coincides with the event $\liminf V_j$. We now prove that $\Pp(\liminf V_j) \leq \Pp(\limsup V_j) = 0$. Consider
\begin{equation*}
        \sum_{j\geq 1} \Pp(V_j) = \sum_{j\geq 1} \Pp(A_{i,j} = 0) = \sum_{j\geq 1} \E(1 - S_j) =  \E \sum_{j\geq 1} (1 - S_j),  
\end{equation*}
and define the process $\mathcal{T} = \sum_{j\geq 1} \delta_{(X_j, \pi_j)}$ on $\X \times [0,1)$, where $\pi_j = 1 - S_j$. We have that 
\begin{equation}\label{eq:borel_cantelli_V_j}
    \sum_{j\geq 1} \Pp(V_j) =  \E \sum_{j\geq 1} \pi_j = \E \left\{\int_{\X \times [0,1)} t  \mathcal{T}(\dd x \, \dd t)\right\} = \int_{\X \times [0,1)} t  M_{\mathcal{T}}(\dd x \, \dd t),
\end{equation}
where the last equality follows from Campbell averaging formula. Now, defining the function $g: \X\times(0,1] \rightarrow \X\times[0,1)$ as $g(x,s) = (x, 1-s)$, we have that $\mathcal{T} = \Psi \circ g^{-1}$ is the image of $\Psi$ by $g$. Consequently, $M_{\mathcal{T}}(\dd x \, \dd t) = M_{\Psi}(g^{-1}(\dd x, \dd t)) = M_{\Psi}(\dd x \, \dd s)$, where $s = 1 -t$. Thus, it follows that
\begin{equation}\label{eq:integral_equality}
\int_{\X \times [0,1)} t  M_{\mathcal{T}}(\dd x \, \dd t) = \int_{\X \times (0,1]} (1-s)  M_{\Psi}(\dd x \, \dd s) < \infty,    
\end{equation}
where the inequality holds by hypothesis. Therefore, from \eqref{eq:borel_cantelli_V_j} and \eqref{eq:integral_equality}, we have that $\sum_{j\geq 1} \Pp(V_j)$ $ < \infty$. By applying the Borel-Cantelli lemma, we obtain $\Pp(\limsup V_j) = 0$ and the proof is complete.
\end{proof}

The following proposition establishes the necessary and sufficient condition, also discussed in \Cref{rem:finite_features}, on the mean measure of $\Psi$ to ensure that $Z_i(\X)<\infty$ a.s., under the assumption that $\Psi$ is a Poisson or a mixed Poisson process. In particular, when $\Psi$ follows either of these processes, the necessary condition given in \Cref{prop:nec_condition} is also sufficient.

\begin{proposition}\label{prop:nec_suff_poisson}
Let $Z_i$ be the generic element of the sequence $(Z_i)_{i\geq1}$ defined in \eqref{eq:main_model_appendix}, where $\mu$ is the functional of a point process $\Psi$ defined via \eqref{eq:functional_appendix}. Assume $\Psi$ is a Poisson process with (infinite) mean measure $\nu$ or a mixed Poisson process $\mathrm{MP}(\nu, f_\gamma)$. Then, $Z_i(\X) < \infty$ a.s. if and only if $\int_{\X\times(0,1]} s \nu(\dd x \, \dd s) < \infty$.    
\end{proposition}
\begin{proof}
We start by proving the statement for $\Psi$ distributed as a Poisson process with mean measure $\nu$. As discussed in \Cref{rem:finite_features}, the condition $\int_{\X\times (0,1]} s M_{\Psi}(\dd x \, \dd s) < \infty$ is sufficient for $Z_i(\X) < \infty$ a.s.. Under the Poisson assumption for $\Psi$, this sufficient condition writes as $\int_{\X\times(0,1]} s \nu(\dd x \, \dd s) < \infty$. We are left to show the inverse implication. Consider the Laplace transform of the random variable $Z_i(\X)$ evaluated in $1$, 
\begin{equation*}
\begin{aligned}
    \E\left\{ e^{-Z_i(\X)} \right\} &= \E\left[ \E\left\{ e^{-Z_i(\X)} \mid \mu  \right\} \right] = \E\left[ \prod_{j \geq 1} \left\{ 1 - S_j ( 1 - e^{-1})  \right\} \right]\\
    &= \E\left( \exp\left[ \sum_{j\geq 1} \log\left\{ 1 - S_j ( 1 - e^{-1}) \right\} \right] \right)\\
    &= \exp\left\{ -  ( 1 - e^{-1}) \int_{\X\times(0,1]} s \nu(\dd x \, \dd s)\right\}.
\end{aligned}
\end{equation*}
Therefore, if $\int_{\X\times(0,1]} s \nu(\dd x \, \dd s) = \infty$, then $ \E\left\{ e^{-Z_i(\X)} \right\} = 0$ and $Z_i(\X) = \infty$ a.s. and the proof is complete.

To prove the result when $\Psi$ is a mixed Poisson process $\mathrm{MP}(\nu, f_\gamma)$, we leverage the result just established for the Poisson process. In particular, $Z_i(\X) < \infty$ a.s. is equivalent to 
\[
1 = \Pp\left( Z_i(\X) < \infty \right) = \E\left\{ \Pp\left( Z_i(\X) < \infty \mid \gamma \right) \right\},
\]
which holds true if and only if $\Pp\left( Z_i(\X) < \infty \mid \gamma \right) = 1$ a.s.. Conditionally to $\gamma$, the process $\Psi$ is a Poisson process with mean measure $\gamma\nu$. Thus, leveraging on the result just shown for the Poisson case, it holds that $\Pp\left( Z_i(\X) < \infty \mid \gamma \right) = 1$ if and only if $\nu$ satisfies the condition stated for a Poisson process, namely $\int_{\X\times(0,1]} s \nu(\dd x \, \dd s) < \infty$. 
\end{proof}

The next proposition examines notable consequences of satisfying the condition $Z_i(\X) < \infty$ a.s. under the assumption that $\Psi$ is a Poisson process. 
\begin{proposition}\label{prop:poisson_details}
Let $Z_i$ be the generic element of the sequence $(Z_i)_{i\geq1}$ defined in \eqref{eq:main_model_appendix}, where $\mu$ is the functional of a point process $\Psi$ defined via \eqref{eq:functional_appendix} and $\Psi$ is a Poisson process with (infinite) mean measure $\nu$. If $Z_i(\X) < \infty$ a.s., then:
\begin{enumerate}
    \item[(i)] the $k$-th factorial moment measure $M_\Psi^{(k)} = \nu^k$ is $\sigma$-finite, for any $k$;

    \item[(ii)] the following disintegration holds: $\nu(\dd x \, \dd s) = \kappa(\dd x \mid s) A_0(\dd s)$, where $\kappa$ is a probability kernel from $(0,1]$ to $\X$ and $A_0(\dd s) := \nu(\X \times \dd s)$ is $\sigma$-finite. Consequently, $\Psi$ can be seen as an independently marked point process with Poisson ground process on $(0,1]$ with mean measure $A_0$ and marks on $\X$ with mark probability kernel $\kappa$. 
\end{enumerate}
If the mean measure $\nu$ is finite, points (i) and (ii) always hold with additional trivial simplifications.
\end{proposition}
\begin{proof}
To prove point (i), we observe that $\nu(\X \times (\epsilon,1]) < \infty$, for any $\epsilon>0$. Indeed, when $\Psi$ is a Poisson process with mean measure $\nu$, the condition $Z_i(\X) < \infty$ a.s. is equivalent to $\int_{\X\times(0,1]} s \nu(\dd x \, \dd s) < \infty$ (\Cref{prop:nec_suff_poisson}). Therefore, for any $\epsilon>0$,
\[
\nu(\X \times (\epsilon,1]) = \int_{\X\times(\epsilon,1]} \nu(\dd x \, \dd s) \leq \epsilon^{-1} \int_{\X\times(\epsilon,1]} s \nu(\dd x \, \dd s) < \infty.
\]
Thus, it follows that $\nu$ is $\sigma$-finite, as well as any $\nu^k$. 

We now focus on point (ii). Consider the projected measure $\nu(\X \times \dd s) =: A_0(\dd s)$ on $(0,1]$. From point (i), $A_0$ is a $\sigma$-finite measure. Then, thanks to \cite[Theorem 3.4]{Kallenberg21}, the following disintegration holds: $\nu(\dd x \, \dd s) = \kappa(\dd x \mid s) A_0(\dd s)$, where $\kappa$ is a probability kernel from $(0,1]$ to $\X$.    
\end{proof}

\section{Auxiliary results on mixed binomial processes}\label{app:mb_results}
This section presents three distributional results related to mixed binomial processes. The first proposition establishes that the $k$-th reduced Palm version of any mixed binomial process remains a mixed binomial process. This result is crucial for proving \Cref{lem:mbp} (see Section \ref{app:pred_characterization}) and \Cref{thm:mixed_binom_bayesian} (see Section \ref{app:examples_computation}).

\begin{proposition}[Palm distribution of mixed binomial processes]\label{prop:palm_mb}
Let $\Phi$ be a mixed binomial process $\mathrm{MB}(\nu, q_M)$ defined on $\X$. Then, the reduced Palm version of $\Phi$ at $x$, denoted with $\Phi^!_x$, is a mixed binomial  process $\mathrm{MB}(\nu, q_{\tilde{M}})$ where
\[
    q_{\tilde{M}}(m) = \frac{(m+1)}{\E(M)} q_M(m+1).
\]
Similarly, the $k$-th reduced Palm version of $\Phi$ at $\bm x = (x_1, \ldots, x_k)$ (distinct points), denoted with $\Phi^!_{\bm x}$, is a mixed binomial process $\mathrm{MB}(\nu, q^{(k)}_{\tilde{M}})$ where
\[
     q^{(k)}_{\tilde{M}}(m) = \frac{ (m+k)!}{\E\{M^{(k)}\} m!} q_M (m+k).
\]
\end{proposition}
\begin{proof}
    Let us focus on the reduced Palm distribution of order $k=1$.
    Let $\mathcal L_\Phi$ be the Laplace functional of $\Phi$.
    By \cite[Proposition 3.2.1]{BaBlaKa}, for any measurable functions $f, g: \X \to \R_+$, the Palm distribution of a point process satisfies
    \begin{equation*}
        \evalat{\frac{\partial}{\partial t} \mathcal{L}_{\Phi}(f + tg)}{t=0} = - \int_{\X} g(x) \mathcal{L}_{\Phi_x}(f) M_{\Phi}(\dd x).
    \end{equation*}
    Since $\Phi$ is a mixed binomial process $\mathrm{MB}(\nu, q_M)$, its Laplace functional writes as    \begin{equation}\label{eq:lap_mbp}
        \mathcal{L}_{\Phi}(f) = \E\left[ \E\left\{e^{- f(X)}\right\}^M \right],
    \end{equation}
    where $M$ has probability mass function $q_M$ and $X$ has law $\nu$, so that
    \[
      \evalat{\frac{\partial}{\partial t} \mathcal{L}_{\Phi}(f + tg)}{t=0} = - \int_{\X} g(x) e^{-f(x)} \left[ \sum_{m \geq 1} m q_M(m)  \E\left\{e^{-f(X)}\right\}^{m-1}\right]  \nu(\dd x).
    \]
    Multiplying and dividing by $\E(M)$, and thanks to a change of index $k = m-1$ in the inner summation, we recognize the following expression
    \[
      \evalat{\frac{\partial}{\partial t} \mathcal{L}_{\Phi}(f + tg)}{t=0} = - \int_{\X} g(x) e^{-f(x)} \left[ \sum_{k \geq 0} \frac{(k+1) q_M(k+1) }{\E(M)} \E\left\{e^{-f(X)}\right\}^{k}\right]  M_{\Phi}(\dd x).
    \]
    Hence, $\Phi_x = \delta_{x} + \Phi^!_{x}$ where $\Phi^!_{x}$ is distributed as in the statement.

    The proof for the Palm distribution of order $k$ follows by induction, by using the Palm algebra property \citep[Proposition 3.3.9]{BaBlaKa}, i.e., $\Phi^!_{(x_1, x_2)} \dequal \left(\Phi^!_{x_1}\right)^!_{x_2}$, for $M_{\Phi^{(2)}}$-a.a. $(x_1,x_2)$.
\end{proof}

The second proposition examines the effect of thinning a mixed binomial process and establishes that its probability law remains that of a mixed binomial process. This result is applied in the proof of \Cref{thm:mixed_binom_bayesian}; see the mixed binomial case of Section \ref{app:examples_computation}.

\begin{proposition}[Thinning of mixed binomial processes] \label{prop:mbp_thin}
Let $\Phi$ be a mixed binomial process $\mathrm{MB}(\nu, q_M)$ defined on $\X$. Let the retention probability $p: \X \rightarrow [0, 1]$ be a measurable function. Then the thinning of $\Phi$ by $p$, denoted by $\Phi_p$, is a mixed binomial process $\mathrm{MB}(\nu_p, q_{M_p})$, where $\nu_p(\dd x) = p(x) \nu(\dd x) / c_p$ is a probability distribution and
\[
    q_{M_p}(m) = \sum_{z \geq m} \binom{z}{m} q_M(z) c_p^{m} (1 - c_p)^{z-m}.
\]
\end{proposition}
\begin{proof}
Given a point process $\Phi$, the thinned process $\Phi_p$ is characterized by the Laplace functional described in  \cite[Proposition 2.2.6]{BaBlaKa}, that is
    \[
        \mathcal L_{\Phi_p}(f) = \mathcal{L}_{\Phi}\left[ -\log\left\{p(x) e^{-f(x)} + 1 - p(x)\right\} \right].
    \]
    Under the assumption that $\Phi$ is a mixed binomial process, its 
    Laplace functional is expressed in \eqref{eq:lap_mbp}. Consequently,
    \[
        \mathcal L_{\Phi_p}(f) = \E\left\{\left( 1 + \E\left[ p(X)\left\{e^{- f(X)} - 1\right\}\right] \right)^M \right\},
    \]
    where $M$ has probability mass function $q_M$ and $X$ has law $\nu$.
    Let $Z$ be distributed according to $\nu_p$, where $\nu_p$ is as in the statement. It follows that 
    \[
        \E\left[ p(X)\left\{e^{- f(X)} - 1\right\}\right] = c_p \E\left\{e^{-f(Z)} \right\} - c_p.
    \]
    An application of the binomial theorem leads to
    \[
        \mathcal L_{\Phi_p}(f) = \E\left[ \sum_{j=0}^M \binom{M}{j} {c_p}^j \E\left\{e^{-f(Z)}\right\}^j (1 - c_p)^{M - j} \right]
    \]
    and the result follows from Fubini's theorem.
\end{proof}

Finally, we highlight a simple yet useful result stating that mixed Poisson processes $\mathrm{MP}(\nu, f_\gamma)$, where $\nu$ is finite, can be viewed as mixed binomial processes. 

\begin{lemma}\label{lemma:MP_as_MB}
Let $\Phi \sim \mathrm{MP}(\nu, f_\gamma)$, where $\nu$ is a finite measure on $\X$ and some $f_\gamma$. Then $\Phi \sim \mathrm{MB}(\nu, q_M)$ for some $q_M$. In particular, letting $\varrho(t) := \E(e^{-t\gamma})$ and $\varphi(\nu(\X)(1 - z)) := \E(z^M)$, it holds that $\varrho = \varphi$ on $[0, \nu(\X)]$.
\end{lemma}
\begin{proof}
Assume that $\Phi \sim \mathrm{MP}(\nu, f_\gamma)$. Define a mixed binomial process 
\[
\tilde{\Phi}\mid M = \sum_{j=1}^M \delta_{\tilde{X}_j},
\]
with $M\mid \gamma \sim \mathrm{Poi}(\gamma \nu(\X))$ and $\gamma \sim f_\gamma$. Moreover, consider $\tilde{X}_1,\ldots,\tilde{X}_M\mid M \iid \nu(\cdot)/\nu(\X)$. The probability-generating function of $M$ can be written as 
\[
\mathcal{G}_M(z) = \E(z^M) = \E[\exp\{ -\gamma \nu(\X)(1-z) \}]
\]
and the Laplace transform of $f(\tilde{X}_1)$ is given by
\begin{equation*}\label{eq:laplace_fX_1}
    \mathcal{L}_{f(\tilde{X}_1)}(t) = \int_{\X} \exp\{ -t f(x) \} \nu(\dd x)/\nu(\X).
\end{equation*}
Consequently, the Laplace functional of $\tilde{\Phi}$ equals 
\begin{equation*}
        \mathcal{L}_{\tilde{\Phi}}(f) = \mathcal{G}_M\left(\mathcal{L}_{f(\tilde{X}_1)}(1)\right) 
        = \E\left[ \exp\left\{ - \int_{\X} (1- e^{-f(x)}) \gamma \nu(\dd x) \right\} \right],
\end{equation*}
therefore $\tilde{\Phi} \sim \mathrm{MP}(\nu, f_\gamma)$ and in particular $\Phi \dequal \tilde{\Phi}$. By definition, $\Phi \dequal \tilde{\Phi} \sim \mathrm{MB}(\nu, q_M)$, with $\varphi(\nu(\X)(1-z)) := \E(z^M) = \E[\exp\{ -\gamma \nu(\X)(1-z) \}] $, thus $\varphi(t) = \E(e^{-t\gamma}) =: \varrho(t)$ for $t \in [0,\nu(\X)]$.  
\end{proof}

\section{Key results from point process theory and Palm distributions}\label{app:general_res_palm}

First, we recall the primary technical tool used in the proof of Theorems \ref{thm:marg} and \ref{thm:post}, namely the Campbell-Little-Mecke formula. This formula can be viewed as an extension of Fubini's theorem to the case when both expectation and the integration are taken with respect to a point process.

\begin{lemma}[Campbell-Little-Mecke (\textsc{clm}) formula]\label{lem:clm}
    Let $\Phi$ be a point process over $(\X, \Xcr)$ such that $M_{\Phi}^{(k)}$ is $\sigma$-finite. For all measurable $f: \X^k \times \mathbb M_\X \rightarrow \R_+$, it holds
    \[
        \E\left\{\int_{\X^k} f \left(\bm x, \Phi -  \sum_{j=1}^k \delta_{x_j}\right) \Phi^{(k)}(\dd \bm x) \right\} =  \int_{\X^k} \E \left\{f(\bm x, \Phi^{!}_{\bm x})\right\} M_{\Phi}^{(k)}(\dd \bm x).
    \]
\end{lemma}

Next, we present a key characterization result \cite[Theorem 5.3]{Kal(73)} for mixed Poisson and mixed binomial processes, which relies on properties of the reduced Palm distributions. This result will be crucial for the proof of \Cref{lem:mbp}. However, it is important to note that in Kallenberg’s theorem, a point process is defined as a random boundedly finite counting measure, whereas we consider a point process as a random locally finite counting measure.
This distinction, while subtle, is crucial. Working with boundedly finite measures excludes Poisson processes with infinite activity. Specifically, in infinite-activity Poisson processes considered in the literature, $\Psi(B \times (0, r]) = \infty$ for any bounded Borel set $B \subset \X$, meaning that  $\Psi$ is not a boundedly finite measure.
Therefore, we verify that the next result remains valid under our definition of a point process.

\begin{lemma}[Theorem 5.3 in \cite{Kal(73)}]\label{lem:kallenberg}
Let $\Phi$ be a point process with locally finite mean measure $M_\Phi$. Then the following assertions are equivalent. 
\begin{enumerate}
        \item[(i)] The distribution of $\Phi^!_x$ is independent of $x$, for $M_\Phi$-a.a. $x$.
        \item[(ii)] $\Phi$ is either a mixed Poisson or mixed binomial process.
    \end{enumerate}
\end{lemma}
\begin{proof} We proceed by demonstrating the validity of the result in two steps, by separately proving the two implications.

\noindent \emph{Proof of (ii) $\implies$ (i)}.
We start by remarking that if $\Phi$ is either a mixed Poisson or a mixed binomial process, then $\Phi^!_x$ is independent of $x \in \X$, for $M_\Phi$-a.a. $x$. Assume $\Phi$ to be a mixed Poisson, i.e., $\Phi\mid \gamma \sim \mathrm{PP}(\gamma \nu)$ and $\gamma \sim f_\gamma$ is an almost surely positive random variable. Then, by \citep[Lemma 2.2.27]{BaBlaKa}, its Laplace functional equals
\begin{equation*}
    \mathcal L_\Phi(f) = \E\left[ \E\left\{ \mathcal L_{\Phi_\gamma}(f) \mid \gamma \right\}  \right] = \E\left( \exp\left[ - \int_{\X} \left\{1-e^{-f(x)}\right\} \gamma \nu(\dd x)\right] \right).
\end{equation*}
By \cite[Proposition 3.2.1]{BaBlaKa}, the Palm distribution of $\Phi$ satisfies    \begin{equation}\label{eq:der_laplace_palm}
        \evalat{\frac{\partial}{\partial t} \mathcal{L}_{\Phi}(f + tg)}{t=0} = - \int_{\X} g(x) \mathcal{L}_{\Phi_x}(f) M_{\Phi}(\dd x).
    \end{equation}
We proceed by computing the left-hand side of \eqref{eq:der_laplace_palm}, under the assumption that $\Phi$ is a mixed Poisson process. First, note that $M_\Phi(B) = \E\left[\E\left\{ \Phi(B)\mid \gamma \right\} \right] = \nu(B) \E(\gamma)$, for $B \in \Xcr$. Hence, one has 
\begin{equation*}
    \evalat{\frac{\partial}{\partial t} \mathcal{L}_{\Phi}(f + tg)}{t=0} = \E\left( \E\left[ h^\prime(0) \exp\{ h(0) \} \mid \gamma \right] \right),
\end{equation*}
where $h(t) := - \int_\X (1-e^{-f(x) - t g(x)}) \gamma \nu(\dd x)$, and thus $h^\prime(t) := - \int_\X e^{-f(x) - t g(x)} g(x) \gamma \nu(\dd x)$. It follows that
\begin{equation*}
    \begin{aligned}
        \evalat{\frac{\partial}{\partial t} \mathcal{L}_{\Phi}(f &+ tg)}{t=0} \\
        &= \E\left[- \int_\X e^{-f(x)} g(x) \gamma \nu(\dd x) \cdot \exp\left\{ - \int_\X (1-e^{-f(x)}) \gamma \nu(\dd x) \right\} \right]\\
        &= - \int_{\R_+} \int_\X e^{-f(x)} g(x) \exp\left\{ - \int_\X (1-e^{-f(y)}) \gamma \nu(\dd y) \right\} \gamma \nu(\dd x) f_\gamma(\dd \gamma)\\
        &= - \int_\X g(x) \left[ e^{-f(x)} \int_{\R_+} \exp\left\{ - \int_\X (1-e^{-f(y)}) \gamma \nu(\dd y) \right\} \frac{\gamma}{\E(\gamma)} f_\gamma(\dd \gamma) \right] M_\Phi(\dd x).
    \end{aligned}
\end{equation*}
Therefore, by comparing the previous expression with \eqref{eq:der_laplace_palm}, we deduce that, for $M_\Phi$-a.a. $x \in \X$,
\begin{equation*}
    \mathcal{L}_{\Phi_x}(f) = e^{-f(x)} \int_{\R_+} \exp\left\{ - \int_\X (1-e^{-f(y)}) \gamma \nu(\dd y) \right\} \frac{\gamma}{\E(\gamma)} f_\gamma(\dd \gamma) 
\end{equation*}
and the Laplace functional of the Palm version is
\begin{equation*}
    \mathcal{L}_{\Phi^!_x}(f) = \int_{\R_+} \exp\left\{ - \int_\X (1-e^{-f(y)}) \gamma \nu(\dd y) \right\} \frac{\gamma}{\E(\gamma)} f_\gamma(\dd \gamma), 
\end{equation*}
which does not depend on $x$. In particular, it turns out that $\Phi^!_x$ is a mixed Poisson process such that $\Phi^!_x\mid \tilde{\gamma} \sim \mathrm{PP}(\tilde{\gamma}\nu)$ and $\tilde{\gamma} \sim f_{\tilde{\gamma}}$ with $f_{\tilde{\gamma}} (\dd \gamma) \propto \gamma f_{\gamma}(\dd \gamma)$.

If instead $\Phi$ is a mixed binomial process, $\Phi^!_x$ is described by \Cref{prop:palm_mb}: it is still a mixed binomial process with law independent of $x$, for $M_\Phi$-a.a. $x \in \X$.

\noindent \emph{Proof of (i) $\implies$ (ii)}. Conversely, to prove the opposite implication, we need some preliminary lemmas, which are of independent interest. We start by recalling \cite[Lemma 5.1]{Kal(73)}. Given the point process $\Phi$ and the set $C \in \Xcr$, we indicate with $\Phi^C$ the restriction of $\Phi$ on $C$, i.e., $\Phi^C(B) = \Phi(B \cap C)$, for any $B \in \Xcr$. Moreover, let $\N$ denote the set of natural numbers $\{1,2,\ldots\}$.

\begin{lemma}\label{lemma:5.1kall}
If $\Phi$ is a point process with locally finite mean measure $M_\Phi$, then
\begin{equation*}
    \Pp(\Phi \in L \mid \Phi(C) = n, \eta \in \dd x) = \Pp(\Phi_x \in L\mid \Phi_x(C) = n),
\end{equation*}
for $L \in \mathcal M_\X, C \in \Xcr, n \in \N$, and $\lambda_{C,n}$-a.a. $x \in C$, where $\eta$ is the position of a randomly chosen atom of $\Phi^C$ and 
\[
\lambda_{C,n}(\dd x) = \E\{\Phi(\dd x)\mid \Phi(C) = n\}\Pp(\Phi(C) = n) = \Pp(\Phi_x(C) = n) M_\Phi(\dd x).
\]

\end{lemma}
\begin{proof}
Preliminarily, we need to remark an alternative and limiting characterization of Palm versions in terms of the so-called \emph{local Palm probabilities}, denoted by $\mathrm{P}^x$ and defined on $(\Omega, \Acr)$ \citep[see][]{BaBlaKa}. Specifically, for $M_\Phi$-a.a. $x \in \X$, the following definition is given, for $A \in \Acr$,
\[
\mathrm{P}^x(A) = \frac{\E(\Phi(\dd x)\indicator_A)}{M_\Phi(\dd x)}.
\]
Define the family of processes $\{{\Phi_x^{loc}\}}_{x \in \X}$ whose laws are given as follows: $\Pp(\Phi_x^{loc} \in L) = \mathrm{P}^x(\Phi^{-1}(L))$, for any $L \in \mathcal{M}_\X$. Following \cite{BaBlaKa}, we observe that $\Phi_x^{loc} \dequal \Phi_x$, thus the Palm versions of $\Phi$ can be characterized as follows 
\begin{equation}\label{eq:palm_local_charact}
    \Pp(\Phi_x \in L) = \frac{\E(\Phi(\dd x)\mid \Phi \in L)\Pp(\Phi \in L)}{M_\Phi(\dd x)}.
\end{equation}
To prove the lemma, consider $x \in C$, then
\begin{equation}\label{eq:random_point_kall}
    \begin{aligned}
        \Pp(\Phi \in L, &\Phi(C) = n, \eta \in \dd x) = \Pp(\eta \in \dd x \mid \Phi \in L, \Phi(C) = n)  \Pp(\Phi \in L, \Phi(C) = n)\\
        &= \E(\delta_\eta(\dd x) \mid \Phi \in L, \Phi(C) = n) \Pp(\Phi \in L, \Phi(C) = n)\\
        &= \frac{1}{n}\E(\Phi(\dd x) \mid \Phi \in L, \Phi(C) = n) \Pp(\Phi \in L, \Phi(C) = n),
    \end{aligned}
\end{equation}
where the last equality holds since $\eta$ is a randomly chosen atom of $\Phi$. Moreover, applying \eqref{eq:palm_local_charact} to the event $L^\prime := \{\mu \in \mathbb M_\X: \mu \in L, \mu(C) = n\}$, we get
\[
\E(\Phi(\dd x) \mid \Phi \in L, \Phi(C) = n) \Pp(\Phi \in L, \Phi(C) = n) = \Pp(\Phi_x \in L, \Phi_x (C) = n) M_\Phi(\dd x). 
\]
Consequently, plugging the last expression into \eqref{eq:random_point_kall}, it follows that
\begin{equation}\label{eq:lemma51_first}
         \Pp(\Phi \in L, \Phi(C) = n, \eta \in \dd x)
         = \frac{1}{n}\Pp(\Phi_x \in L, \Phi_x(C) = n) M_\Phi(\dd x)
\end{equation}
and, in particular, choosing $L = \mathbb M_\X$,
\begin{equation}\label{eq:lemma51_second}
    \Pp( \Phi(C) = n, \eta \in \dd x) = \frac{1}{n}\Pp( \Phi_x(C) = n) M_\Phi(\dd x).
\end{equation}
Finally, 
\begin{equation*}
    \begin{aligned}
        \Pp(\Phi \in L\mid \Phi(C) = n, \eta \in \dd x) &= \frac{\Pp(\Phi \in L, \Phi(C) = n, \eta \in \dd x)}{\Pp(\Phi(C) = n, \eta \in \dd x)}\\
        &= \frac{\Pp(\Phi_x \in L, \Phi_x(C) = n)}{\Pp( \Phi_x(C) = n) }\\
        &= \Pp(\Phi_x \in L\mid \Phi_x(C) = n),
    \end{aligned}
\end{equation*}
where the second equality follows from \eqref{eq:lemma51_first} and \eqref{eq:lemma51_second}. Clearly, the previous computation holds when $\Pp(\Phi(C) = n, \eta \in \dd x) > 0$, that is $\lambda_{C,n}(\dd x) = \Pp(\Phi_x(C) = n) M_\Phi(\dd x) > 0$.
\end{proof}

Secondly, we state another key result which we need for the proof, corresponding to \cite[Theorem 3.7]{Kallenberg17}.
\begin{lemma}\label{lemma:5.2kall}
    Let $\Phi$ be a point process on $\X$ and let $C_j\uparrow \X$, $C_j \in \Xcr$ and relatively compact, $j \geq 1$. Then $\Phi$ is either a mixed Poisson or mixed binomial process if and only if $\Phi^{C_j}$ is a mixed binomial process for every $j \geq 1$. 
\end{lemma}

We can now prove that (i) implies (ii). Assume that $\Phi^!_x$ is independent of $x$, that is $\Phi^!_x = \Phi_x - \delta_x \dequal \xi$, for some point process $\xi$, for $M_\Phi$-a.a. $x \in \X$. Consider $C \in \Xcr$ relatively compact set and $n \in \N$ such that $\Pp(\xi(C) = n) > 0$. Let $k \in \N$ and $\bm B = (B_1,\ldots, B_k) \in \Xcr^{\otimes k}$ whose components are a partition of $C$. Let $\bm z \in (\N \cup \{0\})^k$ and $\bm e_1 = (1,0,\ldots,0)$ with $k-1$ zeros. Let $\eta$ as in Lemma \ref{lemma:5.1kall}. Assume, without loss of generality, that $x \in B_1$; for $\lambda_{C,n}$-a.a. $x$, it holds that
\begin{equation*}
    \begin{aligned}
        \Pp((\Phi - \delta_\eta)(\bm B) &= \bm z \mid \Phi(C) = n, \eta \in \dd x) =  \Pp( \Phi(\bm B) = \bm z + \bm e_1 \mid \Phi(C) = n, \eta \in \dd x)\\
        &= \Pp( \Phi_x(\bm B) = \bm z + \bm e_1 \mid \Phi_x(C) = n) \\
        &= \Pp( (\Phi_x - \delta_x)(\bm B) = \bm z \mid (\Phi_x - \delta_x)(C) = n - 1)\\
        &= \Pp( \xi(\bm B) = \bm z \mid \xi(C) = n -1),
    \end{aligned}
\end{equation*}
where the second equality follows from Lemma \ref{lemma:5.1kall} and the last one from the definition of $\xi$.
Therefore, the following holds: 
\begin{equation}\label{eq:fact_cond_indep}
    \begin{aligned}
        \Pp((\Phi - \delta_\eta)(\bm B) &= \bm z, \eta \in \dd x\mid \Phi(C) = n) \\
        &=\Pp((\Phi - \delta_\eta)(\bm B) = \bm z \mid \Phi(C) = n, \eta \in \dd x) \Pp( \eta \in \dd x \mid \Phi(C) = n)  \\
        &=  \Pp( \xi(\bm B) = \bm z \mid \xi(C) = n -1) \Pp(\eta \in \dd x \mid \Phi(C) = n).
    \end{aligned}
\end{equation}
Then, define $\bm B = (B_1,\ldots, B_k)$ as follows: let $k = \Bar{k} + 1$, $\Bar{k} \leq n$, consider $\Bar{k}$ points in $C$, denoted with $\Bar{x}_1,\ldots, \Bar{x}_{\Bar{k}}$, and let $B_j = \dd\Bar{x}_j$, $j \leq \Bar{k}$, be a neighborhood of $\Bar{x}_j$, such that $B_i \cap B_j = \emptyset$, for any $i,j \leq \Bar{k}, i \neq j$. Finally, define $B_k = C \setminus \cup_{j=1}^{\Bar{k}}\dd \Bar{x}_j$ and $\bm z = (z_1,\ldots, z_{\Bar{k}}, 0)$ such that $\sum_{j=1}^{\Bar{k}} z_j = n - 1$. Then, the factorization in \eqref{eq:fact_cond_indep} gives
\begin{equation*}
    \begin{aligned}
        &\Pp\Big( \bigcap_{j=1}^{\Bar{k}} (\Phi - \delta_\eta)(\dd \Bar{x}_j) = z_j, (\Phi - \delta_\eta)( C \setminus \cup_{j=1}^{\Bar{k}}\dd \Bar{x}_j) = 0, \eta \in \dd x\mid \Phi(C) = n \Big) \\
        &=  \Pp\Big(\bigcap_{j=1}^{\Bar{k}} \xi(\dd \Bar{x}_j) = z_j, \xi( C \setminus \cup_{j=1}^{\Bar{k}}\dd \Bar{x}_j) = 0 \mid \xi(C) = n -1\Big) \Pp( \eta \in \dd x\mid \Phi(C) = n),
    \end{aligned}
\end{equation*}
from which we conclude that the distribution of $\eta$, conditionally to $\Phi(C) = n$, is independent of the position of the remaining $n-1$ points. Then, we conclude that a randomly chosen point of $\Phi - \delta_\eta$ is conditionally independent of the others, given that $(\Phi - \delta_\eta)(C) = n-1$, by repeating the same argument above for the process $\Phi - \delta_\eta$. Note that such an argument applies to $\Phi - \delta_\eta$ since $(\Phi - \delta_\eta)_x - \delta_x = \Phi_x -\delta_x -\delta_\eta = \xi -\delta_\eta$ which does not depend on $x$. Wrapping up, we have just shown that, conditionally on $\Phi(C) = n$, the $n$ points of $\Phi$ are independent between them. Equivalently, $\Phi^C$ is a binomial process given $\Phi(C) = n$, that is $\Phi^C$ is marginally a mixed binomial process. Then, taking a sequence of sets $C_j \uparrow \X, C_j \in \Xcr$ and relatively compact, $j \geq 1$, it holds that $\Phi^{C_j}$ is a mixed binomial process for every $j \geq 1$. Finally, by an application of \Cref{lemma:5.2kall}, we conclude that $\Phi$ is either a mixed Poisson or mixed binomial process.
\end{proof}

\section{Proof of Theorems \ref{thm:marg} and \ref{thm:post}}\label{app:proofs_main_theorems}

The proof is based on the study of the Laplace functional of $\mu$. We remind that, for any measurable function $f: \X \rightarrow \R_+$, the Laplace functional of $\mu$ equals
\[
    \E\left[ \exp\left\{ - \int_\X f(x) \mu (\dd x)\right\} \right] = \E\left[ \exp\left\{ - \int_{\X \times (0, 1]} s f(x) \Psi(\dd x \, \dd s )\right\} \right].
\]
In the following, we will use the shorthand notations $\mu (f) = \int_\X f(x) \mu (\dd x)$ and $\Psi(sf) = \int_{\X \times (0, 1]} s f(x) \Psi(\dd x \, \dd s)$.
Let $\bm Z := (Z_1, \ldots, Z_n)$ and denote by $\bm z$ its realization, with associated likelihood function $L(\bm z \mid \mu)$. An application of Bayes' theorem entails
\begin{equation}\label{eq:post_laplace}
    \E\left\{e^{-\mu(f)} \mid \bm Z = \bm z \right\} = \frac{\E\left\{ e^{-\Psi(sf)} L(\bm z \mid \mu)\right\}}{\E\left\{L(\bm z \mid \mu)\right\}},
\end{equation}
where, at the denominator, we recognize the marginal likelihood.

Let $x^*_1, \ldots, x^*_k$ be the $k$ distinct features displayed in $\bm z$ and denote by $m_1, \ldots, m_k$ the corresponding frequency counts.
Let $\Theta = \X \times (0, 1]$; arguing as in \cite[Appendix A]{Jam(17)}, the likelihood function can be shown to be
\[
  L(\bm z \mid \mu) =  \int_{\Theta^k}   e^{n \int \log(1 - t) \Psi(\dd z \, \dd t)} \prod_{\ell=1}^{k} \frac{s_\ell^{m_\ell}}{(1 - s_\ell)^{m_\ell}} \delta_{x^*_\ell}(x_\ell) \Psi^k(\dd \bm x \, \dd \bm s)
\]
where $\Psi^k$ is the $k$-th power of $\Psi$.
Note that the term $\prod_{\ell=1}^k \delta_{x^*_\ell}(x_\ell)$ entails that the integrand is zero on sets of the type
\[
    \{(\bm x, \bm s) \in \Theta^k : \; x_i = x_j \text{ for } i \neq j \}.
\]
Therefore we can replace $\Psi^k$ with the $k$-th factorial power $\Psi^{(k)}$. Then, an application of the \textsc{clm} formula (cf. \Cref{lem:clm}) yields the following expression for the numerator of \eqref{eq:post_laplace}, 
\begin{align*}
    \E\left\{ e^{-\Psi(sf)} L(\bm z \mid \mu)\right\} &=  \E \left\{ \int_{\Theta^k } e^{ -\Psi(tf - n\log(1 - t))} \prod_{\ell=1}^{k} \frac{s_\ell^{m_\ell}}{(1 - s_\ell)^{m_\ell}} \delta_{x^*_\ell}(x_\ell) \Psi^{(k)}(\dd \bm x \, \dd \bm s)\right\} \\
    &= \int_{\Theta^k } \E\left\{e^{ -\Psi^!_{\bm x, \bm s}(tf - n\log(1 - t))}\right\} \\
    & \qquad \times \prod_{\ell=1}^k e^{-s_\ell f(x_\ell) + n \log(1 - s_\ell)} \frac{s_\ell^{m_\ell}}{(1 - s_\ell)^{m_\ell}} \delta_{x^*_\ell}(x_\ell) M_{\Psi}^{(k)} (\dd \bm x \, \dd \bm s).
\end{align*}
By assumptions, the disintegration in \eqref{eq:mpsi_disinteg} holds true. Moreover, integrating with respect to the $x_\ell$'s, we obtain
\begin{equation} \label{eq:numerator_general}
\begin{split}
   \E\left\{ e^{-\Psi(sf)} L(\bm z \mid \mu)\right\} &= \int_{(0, 1]^k} \E\left\{e^{ -\Psi^!_{\bm x^*, \bm s}(tf - n\log(1 - t))}\right\} \\
   & \qquad \times\prod_{\ell=1}^k e^{- s_\ell f(x^*_\ell) }  s_\ell^{m_\ell} (1 - s_\ell)^{n - m_\ell} \rho^{(k)}(\dd \bm s \mid \bm x^*) \tilde m^{(k)}_{\xi}(\dd \bm x^*).
\end{split}
\end{equation}
The expression in Theorem \ref{thm:marg} follows by setting $f=0$ in the equation above, which coincides with the marginal likelihood of the data.\\

As for the proof of Theorem \ref{thm:post}, we can simply evaluate \eqref{eq:post_laplace}, where the numerator equals \eqref{eq:numerator_general}, while the denominator coincides with the marginal likelihood, i.e., Equation \eqref{eq:numerator_general} when 
$f=0$. Thus, the posterior Laplace functional in \eqref{eq:post_laplace} boils down to
\[
\begin{split}
& \E\left\{e^{-\mu(f)} \mid \bm Z = \bm z \right\} \\
& \qquad = \frac{\int_{(0, 1]^k} \E\left\{e^{ -\Psi^!_{\bm x^*, \bm s}(tf - n\log(1 - t))}\right\} \prod_{\ell=1}^k e^{- s_\ell f(x^*_\ell) }  s_\ell^{m_\ell} (1 - s_\ell)^{n - m_\ell} \rho^{(k)}(\dd \bm s \mid \bm x^*)}{\int_{(0, 1]^k} \E\left\{e^{\Psi^!_{\bm x^*, \bm s}(n\log(1 - t))}\right\} \prod_{\ell=1}^k  s_\ell^{m_\ell} (1 - s_\ell)^{n - m_\ell} \rho^{(k)}(\dd \bm s \mid \bm x^*)}.
\end{split}
\]
To conclude the proof of Theorem \ref{thm:post}, we need to show that the right-hand side above coincides with the Laplace transform of the measure 
\[
    \eta := \sum_{\ell=1}^k S^*_\ell \delta_{x^*_\ell} + \mu^\prime
\]
where the laws of $S^*_1, \ldots, S^*_k$ and $\mu^\prime$ are as in the statement of  Theorem \ref{thm:post}. This is indeed true, since the Laplace functional of $\eta $ is 
\begin{align*}
    \E\left\{e^{- \eta (f)}\right\} &= \E\left\{e^{- \sum_{\ell=1}^k S^*_\ell f(x^*_\ell) -\mu^\prime(f)}\right\} \\
    &= \int_{(0, 1]^k} \E\left\{e^{-\mu^\prime(f)} \mid \bm S^* = \bm s\right\} \prod_{\ell=1}^k e^{- s_\ell f(x^*_\ell)} f_{\bm S^*}(\dd \bm s)  \\
    &= \int_{(0, 1]^k} \frac{\E\left\{e^{- \Psi^!_{\bm x^*, \bm s} ( t f - n \log(1- t))}\right\}}{\E\left\{e^{\Psi^!_{\bm x^*, \bm s} ( n \log(1- t))}\right\}} 
    \\ & \quad \times \frac{\E\left\{e^{\Psi^!_{\bm x^*, \bm s} ( n \log(1- t))}\right\} \prod_{\ell=1}^k  e^{- s_\ell f(x^*_\ell)} s_\ell^{m_\ell} (1 - s_\ell)^{n-m_\ell} \rho^{(k)}(\dd \bm s \mid \bm x^*)}{ \int_{(0, 1]^k} \E\left\{e^{\Psi^!_{\bm x^*, \bm s} (n \log(1- t))}\right\} \prod_{\ell=1}^k  s_\ell^{m_\ell} (1 - s_\ell)^{n-m_\ell} \rho^{(k)}(\dd \bm s \mid \bm x^*)}
\end{align*}
which is exactly the Laplace functional of $\mu$ that we have found before, and this concludes the proof.

\section{Results and proofs of Section \ref{sec:predictive_characterization}}\label{app:pred_characterization}

In order to prove the main result of Section \ref{sec:predictive_characterization}, two lemmas are necessary and are provided in the following. In particular, central to the proof of point (i) is a novel characterization of the Poisson point process in terms of its (higher-order) reduced Palm distribution that might be of independent interest. We report it in the following lemma.

\begin{lemma}\label{lem:poi}
    Let $\Phi$ be a point process with locally finite mean measure $M_\Phi$. Then the following assertions are equivalent.
    \begin{enumerate}
        \item[(i)] $\Phi$ is a Poisson  process.
        \item[(ii)] The law of $\Phi^!_{\bm x}$ does not depend on $\bm x$. That is, for $M_\Phi^{(k)}$-almost all $\bm x = (x_1, \ldots, x_k)$ and $M_\Phi^{(m)}$-almost all $\bm y = (y_1, \ldots, y_m)$, it holds that $\Phi^!_{\bm x} \dequal \Phi^!_{\bm y}$.
    \end{enumerate}
    Moreover, if (ii) holds, then $\Phi^!_{\bm x} \dequal \Phi^!_{\bm y} \dequal \Phi$. 
\end{lemma}
\begin{proof}
    If $\Phi$ is a Poisson process, then $\Phi^!_{\bm x} \dequal \Phi^!_{\bm y} \dequal  \Phi$ by the multivariate Mecke equation \cite[Theorem 4.4]{Las(17)}. Thus, this shows that (i) implies (ii).

On the other hand, assume that (ii) holds true, we need to show that $\Phi$ is a Poisson process. To this end, 
consider $\bm x = (x_1, \ldots, x_k)$ and $\bm y = (x_1, \ldots, x_k, y)$, i.e., $\bm y$ is defined by adding a single point $y$ to $\bm x$.
Then, by the Palm algebra \cite[Proposition 3.3.9]{BaBlaKa} we have
\[
    \Phi^!_{\bm y} \dequal \left(\Phi^!_{\bm x}\right)^!_{y} \dequal \Phi^!_{\bm x}
\]
where the last equality holds by hypothesis.
That is, we have proven that $\Phi^!_{\bm x}$ is a Poisson process using Slivnyak-Mecke \cite[Theorem 3.2.4]{BaBlaKa}.
Setting $k=1$, we have that $\Phi^!_x$ is Poisson and by hypothesis, the law does not depend on $x$. Therefore, we conclude that $\Phi$ is a Poisson process since the family of Palm distributions of a point process characterizes its law. See, e.g.,  \cite[Proposition 3.1.17]{BaBlaKa}.
\end{proof}

Moreover, to prove point (ii) of Theorem \ref{teo:pred_char}, we need a characterization of mixed Poisson and mixed binomial processes. We provide it in the next lemma, whose proof is a slight extension of a result in \cite{Kal(73)}.
\begin{lemma}\label{lem:mbp}
    Let $\Phi$ be a point process with locally finite mean measure $M_\Phi$. Then the following statements are equivalent.
    \begin{enumerate}
        \item[(i)] $\Phi$ is a mixed Poisson or a mixed binomial point process.
        \item[(ii)] The law of $\Phi^!_{\bm x}$ depends on $\bm x$ only through its cardinality.
        That is, for $M_\Phi^{(k)}$-almost all $\bm x = (x_1, \ldots, x_k)$ and $\bm y = (y_1, \ldots, y_k)$, it holds that $\Phi^!_{\bm x} \dequal \Phi^!_{\bm y}$.
    \end{enumerate}
\end{lemma}
\begin{proof}
    The proof of our lemma requires the use of \Cref{lem:kallenberg}. 
Using this characterization, proceed as follows to prove \Cref{lem:mbp} . First, (ii) implies (i) by using \cite[Theorem 5.3]{Kal(73)}. Indeed, by setting $k=1$, condition (ii) states that $\Phi^!_{x}$ does not depend on $x$, which implies $\Phi$ to be either a mixed Poisson or a mixed binomial process (\Cref{lem:kallenberg}). On the other side, (i) implies (ii). Indeed, if $\Phi$ is a mixed Poisson process, then $\Phi^!_{x_1}$ does not depend on $x_1$ (\Cref{lem:kallenberg}) and it is still a mixed Poisson (see the proof of the inverse implication of \Cref{lem:kallenberg}). Conversely, if $\Phi$ is a mixed binomial process, then $\Phi^!_{x_1}$ is still a mixed binomial process which does not depend on $x_1$ (see \Cref{prop:palm_mb}). Therefore, in both cases, $\Phi^!_{(x_1, x_2)} = \left(\Phi^!_{x_1} \right)^!_{x_2}$ does not depend on $(x_1,x_2)$ and it is still either a mixed Poisson or a mixed binomial process. Continuing with this argument, the proof follows for any $k$.  
\end{proof}

\subsection{Proof of Theorem \ref{teo:pred_char}}

\noindent\emph{Proof of point (i)}.
It is clear that the predictive distribution of $Z^\prime_{n+1}$ depends on the sampling information as the law of $\mu^\prime$, or equivalently $\Psi^\prime$, in Theorem \ref{thm:post} does.
Therefore, the statement of Theorem \ref{teo:pred_char}, point (i), is equivalent to saying that $\Psi^\prime$ depends only on $n$ if and only if $\Psi$ is a Poisson process.
From \eqref{eq:density_psiprime}, it is clear that $\Psi^\prime$ depends on $n$ and $\Psi^!_{\bm x^*, \bm s^*}$, where $\bm x^*$ are the observed distinct feature labels and $\bm s^*$ is distributed as in Theorem \ref{thm:post}.
Then, the proof follows from Lemma \ref{lem:poi} which characterizes the Poisson process as the unique point process for which the reduced Palm kernel $\Psi^!_{\bm x^*, \bm s^*}$ does not depend on $(\bm x^*, \bm s^*)$.\\

\noindent\emph{Proof of point (ii)}.
The proof follows by arguing as in the proof of Theorem \ref{teo:pred_char}, point (i), above but invoking Lemma \ref{lem:mbp} instead of Lemma \ref{lem:poi}.

\section{Proofs and additional details of  \Cref{sec:sncp_model_and_application}}\label{app:sncp_appendix}

\subsection{Proofs of Section \ref{sec:sncp_model_and_application}}\label{app:proofs_cox}

We first characterize the reduced Palm distribution for the Cox process in \eqref{def:psi_sncp}. This follows from an application of Theorem 3 of \cite{Ber26_Palm}. Let 
\[
    \eta(y_1, \ldots, _k) = \int \gamma^k \prod_{j=1}^k \kappa(y_j; \theta) \nu (\dd \theta \, \dd \gamma).
\]
\begin{lemma}\label{thm:red_palm_gensncp_ibp} 
Let $\Psi$ as in \eqref{def:psi_sncp} and $(\bm{x}, \bm{s}) = ((x_1,s_1), \ldots, (x_k,s_k)) \in (\X \times (0,1])^k$. Then, the reduced Palm version of $\Psi$ at $(\bm{x}, \bm{s})$ admits the following representation
\begin{equation*}
    \Psi^!_{(\bm{x}, \bm{s})}\mid \bm{T} \dequal \Psi + \sum_{h = 1}^{|\bm{T}|} \Psi_{\zeta_{{\bm{x}}_h}},\qquad \prob(\bm{T} = \bm{t}) \propto \prod_{h=1}^{|\bm{t}|} \eta(\bm{x}_h),
\end{equation*}
where $\bm{T} := (T_1,\ldots,T_k)$ are latent indicators describing a partition of $\bm{x}$ into $|\bm{T}|$ clusters and
\begin{equation*}
    \begin{aligned}
        \Psi_{\zeta_{{\bm{x}}_h}} \mid \zeta_{{\bm{x}}_h} = (\theta_{{\bm{x}}_h}, \gamma_{{\bm{x}}_h}) &\sim \textsc{pp}(\rho(\dd s) \gamma_{{\bm{x}}_h} \kappa(x; \theta_{{\bm{x}}_h}) \,\dd x)\\
        \zeta_{{\bm{x}}_h}  &\sim f_{{\bm{x}}_h}(\dd \theta\, \dd \gamma) \propto \gamma^{n_h} \prod_{x_j \in 
        \bm{x}_h}\kappa(x_j;\theta) \nu(\dd \theta\, \dd \gamma),
    \end{aligned}
\end{equation*}
where $\bm{x}_h = (x_j: T_j = h)$ and $n_h$ is the cardinality of $\bm{x}_h$.
Finally, the processes $\Psi$ and $\Psi_{\zeta_{{\bm{x}}_h}}$, $h = 1,\ldots, |\bm{T}|$, are mutually independent conditionally to $\bm{T}$.
\end{lemma}
\begin{proof}
If $\rho((0,1])<\infty$, then $\Psi$ is the independent $\rho/\rho((0,1])$-marking of the ground SNCP $\Phi=\Psi(\cdot\times(0,1])$. The result follows from  Proposition 3.2.14 in \cite{BaBlaKa}, together with Theorem 3 in \cite{Ber26_Palm}. 

Assume now that $\rho((0,1])=\infty$ and $\rho((\varepsilon,1])<\infty$ for every $\varepsilon>0$.
 Let $A_\varepsilon:=\X\times(\varepsilon,1]$, $\rho_\varepsilon:=\rho_{|(\varepsilon,1]}$, $c_\varepsilon:=\rho((\varepsilon,1])$, and $\Psi^{(\varepsilon)} := r_\varepsilon(\Psi) = \Psi_{|A_\varepsilon}$.
Fix $\varepsilon<\min_i s_i$. Then, conditionally on $\Lambda$, $\Psi^{(\varepsilon)}\mid\Lambda \sim \textsc{pp}\left(\rho_{|(\varepsilon,1]}(ds) L_\Lambda(x) \dd x\right)$. Hence, the finite-mark case
applies to $\Psi^{(\varepsilon)}$.
Namely,
\begin{equation}\label{eq:palm_psi_eps}
    (\Psi^{(\varepsilon)})^!_{(\bm x,\bm s)}\mid\bm T\dequal\Psi^{(\varepsilon)}+\sum_{h=1}^{|\bm T|}\Psi^{(\varepsilon)}_{\zeta_{\bm x_h}}, \qquad \prob(\bm{T} = \bm{t}) \propto \prod_{h=1}^{|\bm{t}|} \eta_\varepsilon(\bm{x}_h),
\end{equation}
where 
\[
    \eta_\varepsilon(\bm x_h):=\int_{\Theta\times\R_+}(c_\varepsilon\gamma)^{n_h}\prod_{x_j\in\bm x_h}\kappa(x_j;\theta)\nu(\dd\theta\dd\gamma)=c_\varepsilon^{n_h}\eta(\bm x_h) = c_\varepsilon^k\prod_{h=1}^{|\bm t|}\eta(\bm x_h).
\]
In particular, the term $\prod_{h=1}^{|\bm t|}c_\varepsilon^{n_h} = c_\varepsilon^k$ in the law of $\bm T$ cancels from the normalizing constant of the law of $\bm T$, so that $\prob(\bm{T} = \bm{t}) $ in \eqref{eq:palm_psi_eps} coincides with the one stated in \Cref{thm:red_palm_gensncp_ibp}.
Similarly, one can show that the density of $\zeta_{\bm x_h}$ in \eqref{eq:palm_psi_eps} concides with the one in the statement.
We now identify the reduced Palm distribution of the truncated process with the restriction of the reduced Palm distribution of the full process.
Let $\tilde f$ be any nonnegative measurable function on $A_\varepsilon^k\times \M_{A_\varepsilon}$. Let $\bm z=(z_1,\ldots,z_k)\in(\X\times(0,1])^k$, with $z_j=(u_j,r_j)$  and define $f(\bm z,\nu) := \indicator_{A_\varepsilon^k} (\bm z)\tilde f(\bm z,r_\varepsilon(\nu))$. Since 
$r_\varepsilon(\Psi-\sum_{j=1}^k\delta_{z_j})=r_\varepsilon(\Psi)-\sum_{j=1}^k\delta_{z_j}$ whenever $\bm z\in A_\varepsilon^k$, we have
\[
    \E\int_{(\X \times (0, 1])^k} f\left(\bm z,\Psi-\sum_{j=1}^k\delta_{z_j}\right) \Psi^{(k)}(\dd\bm z)  = \E\int_{A_\varepsilon^k} \tilde f\left(\bm z,r_\varepsilon(\Psi)-\sum_{j=1}^k\delta_{z_j}\right) (\Psi^{(\varepsilon)})^{(k)}(\dd\bm z).
\]
Applying the reduced CLM formula on both sides of the equation above gives
\[
    \int_{A_\varepsilon^k} \E\left[ \tilde f(\bm z,r_\varepsilon(\Psi^!_{\bm z})) \right] M_\Psi^{(k)}(d\bm z) = \int_{A_\varepsilon^k} \E\left[\tilde f(\bm z,(\Psi^{(\varepsilon)})^!_{\bm z}) \right]
    M_{\Psi^{(\varepsilon)}}^{(k)}(d\bm z).
\]
Since $M_{\Psi^{(\varepsilon)}}^{(k)} = M_\Psi^{(k)}{}_{|A_\varepsilon^k}$, uniqueness of the Campbell disintegration yields equality in distribution between $r_\varepsilon(\Psi^!_{\bm z})$ and $(\Psi^{(\varepsilon)})^!_{\bm z}$ (on $(\M_{A_\varepsilon}, \mathcal M_{A_\varepsilon})$) $M_{\Psi^{(\varepsilon)}}^{(k)}\text{-a.e. }\bm z$.

It remains only to remove the restriction. Let us view $r_\varepsilon(\Psi^!_{\bm z})$ as a measure on $\X\times(0,1]$ by extending it by zero outside $A_{\varepsilon_m}$ and noting that  for every Borel set $B\subseteq\X\times(0,1]$ and any measure $\xi$,
\[
    \xi(B)=\lim_{m\to\infty}r_{\varepsilon_m}(\xi)(B).
\]
Since the evaluation maps $\xi\mapsto\xi(B)$ generate the $\sigma$-field $\mathcal{M}_{\X \times (0, 1]}$, the results follows by a monotone-class argument.
\end{proof}

Second, we provide an expression for the disintegration $M_{\Psi}^{(k)}(\dd \bm x \, \dd \bm s) =  \rho^{(k)}(\dd \bm s \mid \bm x) \tilde m^{(k)}_{\xi}(\dd \bm x)$, which is a direct consequence of Theorem S3 in \cite{Ber26_Palm}.

\begin{lemma}\label{lem:cox_disintegration}
Let $\Psi$ be the Cox process in \eqref{def:psi_sncp}, and set $\nu(\dd\theta\,\dd\gamma):=\tilde\rho(\dd\gamma)G_0(\dd\theta)$. 
Let $\mathcal T_k$ be the set of vectors $\bm t=(t_1,\ldots,t_k)\in\N^k$ such that $t_1=1$ and $t_\ell\leq 1+\max_{j<\ell}t_j$, for $\ell=2,\ldots,k$. 
For $\bm t\in\mathcal T_k$, set $|\bm t|:=\max_{\ell\leq k}t_\ell$, $\bm x_h=(x_\ell:t_\ell=h)$ and $n_h=|\bm x_h|$, for $h=1,\ldots,|\bm t|$. Moreover, for any $\bm y=(y_1,\ldots,y_r)\in\X^r$, define
\[
    \eta(\bm y) := \int_{\Theta\times\R_+}\gamma^r\prod_{\ell=1}^r\kappa(y_\ell;\theta)\nu(\dd\theta\,\dd\gamma)
\]
Then, the $k$-th factorial moment measure of $\Psi$ admits the disintegration $M_\Psi^{(k)}(\dd\bm x\,\dd\bm s)=\rho^{(k)}(\dd\bm s\mid \bm x)\tilde m_\xi^{(k)}(\dd\bm x)$,
where
\[
    \rho^{(k)}(\dd\bm s\mid \bm x)=\prod_{\ell=1}^k\rho(\dd s_\ell), \qquad \tilde m_\xi^{(k)}(\dd\bm x)=\left\{\sum_{\bm t\in\mathcal T_k}\prod_{h=1}^{|\bm t|}\eta(\bm x_h)\right\}\prod_{\ell=1}^k\dd x_\ell.
\]
Here $\mathcal P_k$ denotes the set of partitions of $\{1,\ldots,k\}$.
\end{lemma}

\begin{proof}
Let $g:(\X\times(0,1])^k\to\R_+$ be measurable. Then, 
\begin{align*}
    \int g(\bm x,\bm s)M_\Psi^{(k)}(\dd\bm x\,\dd\bm s)& =\E\left[\int g(\bm x,\bm s)\prod_{\ell=1}^k\rho(\dd s_\ell)L_\Lambda(x_\ell)\dd x_\ell\right] \\
    &= \int g(\bm x,\bm s) \E\left[\prod_{\ell=1}^kL_\Lambda(x_\ell)\right]\prod_{\ell=1}^k\rho(\dd s_\ell)\dd x_\ell \\
    &= \int  g(\bm x,\bm s)  \prod_{\ell=1}^k\rho(\dd s_\ell)\dd x_\ell  \sum_{\bm t\in\mathcal T_k}\prod_{h=1}^{|\bm t|} \int_{\Theta\times\R_+}\gamma^{n_h}\prod_{\ell\in B}\kappa(x_\ell;\theta)\nu(\dd\theta\,\dd\gamma) \dd x_\ell
\end{align*}
where the first equality follows from the definition of $\Psi$, the second by Tonelli's theorem, and the third from the moment measure formula for Poisson processes.
The claim follows recognizing the definition of $\eta(\bm x_B)$.
\end{proof}

Now, we provide the proof of \Cref{prop:bayesian_sncp_features_posterior}, exploiting \Cref{thm:red_palm_gensncp_ibp}.\\

\noindent \emph{Proof of \Cref{prop:bayesian_sncp_features_posterior}}. The posterior distribution of $\mu$, given a sample $\bm{Z}$ displaying $k$ features with labels $\bm{x^*} = (x^*_1, \ldots, x^*_k)$ and corresponding vector of  frequency counts $\bm{m}:= (m_1, \ldots, m_k)$, is obtained by specializing \Cref{thm:post} in \Cref{sec:general_bayesian_analysis}, which provides the following distributional equality
\[
    \mu \mid \bm Z \dequal  \sum_{\ell=1}^{k} S^*_{\ell} \delta_{x^*_{\ell}} +  \mu^{\prime},
\]
where $\bm{S^*} := (S^*_1, \ldots, S^*_k)$ is a vector of positive random variables with joint distribution 
       \[
        f_{\bm{S^*}}(\dd \bm s) \propto  \E\left\{e^{\int_{\X \times (0, 1]} n \log(1-t) \Psi^!_{\bm{x^*}, \bm s}(\dd z \, \dd t)}\right\} \prod_{\ell=1}^k s_\ell^{m_\ell} (1 - s_\ell)^{n - m_\ell} \rho^{(k)}(\dd \bm s \mid \bm{x^*}).
       \]
From \Cref{thm:red_palm_gensncp_ibp}, the reduced Palm version $\Psi^!_{\bm{x^*}, \bm s}$ does not depend on $\bm s$, thus the $S^*_{\ell}$'s turn out to be independent random variables with marginal law $f_{S^*_{\ell}}(\dd s) \propto s^{m_\ell} (1 - s)^{n - m_\ell} \rho(\dd s)$, as $\ell=1, \ldots , k$. Conditionally to $\bm{S^*}$, from \Cref{thm:post}, $\mu^\prime$ can be represented as $\mu^\prime = \sum_{j\geq 1} S^\prime_j \delta_{X^\prime_j}$, where $\Psi^\prime  = \sum_{j \geq 1} \delta_{(X^\prime_j, S^\prime_j)}$ is characterized by the Laplace functional 
\begin{equation}\label{eq:laplace_psi_prime}
    \mathcal{L}_{\Psi^\prime\mid \bm s^*}(g) = \frac{\mathcal{L}_{\Psi^!_{\bm{x^*}, \bm s^*}} ( g(x,s) - n \log(1-s))}{\mathcal{L}_{\Psi^!_{\bm{x^*}, \bm s^*}} ( - n \log(1-s))},
\end{equation}
for any measurable function $g:\X\times (0,1] \to \R$. As already observed from \Cref{thm:red_palm_gensncp_ibp}, none of the terms depend on $\bm s^*$. Using the expression of the reduced Palm version $\Psi^!_{\bm{x^*}, \bm s^*}$ provided by \Cref{thm:red_palm_gensncp_ibp} in \eqref{eq:laplace_psi_prime}, the thesis follows.\\

\noindent \emph{Proof of \Cref{prop:sncp_Cmn}}.
From \Cref{prop:bayesian_sncp_features_posterior}, a new cluster appears if a point of $\Lambda^{(0)}$, say $(\theta_0, \gamma_0)$, generates a process $\Psi^{(0)}_{(\theta_0, \gamma_0)} = \sum_{j \geq 1}  \delta_{(X^\prime_j, S^\prime_j)} \mid (\theta_0, \gamma_0) \sim \mathrm{PP}(\gamma_0 \rho_n(\dd s)  \kappa(x; \theta_0) \dd x)$ for which at least one of the $X^\prime_j$'s appears in (at least one of) the observations $Z_{n+1}, \ldots, Z_{n+m}$. Specifically, the probability that feature $X^\prime_j$ is picked by at least one of the $m$ next observations is $1 - (1 - S^\prime_j)^m$. Consequently, the process collecting those features picked by at least one of the $m$ next observations, conditionally to $(\theta_0, \gamma_0)$, is obtained by independently thinning $\Psi^{(0)}_{(\theta_0, \gamma_0)}$ with retention probability $p(x, s) = 1 - (1 -s)^m$. By the properties of the Poisson process, such a process follows a Poisson process with intensity $\gamma_0 (1 - (1 -s)^m) \rho_n(\dd s) \kappa(x; \theta_0) \dd x$, conditionally to $(\theta_0, \gamma_0)$. Therefore, the point $(\theta_0, \gamma_0)$ gives rise to a new cluster if the such a Poisson process has at least one point, which happens with probability equal to $q(\gamma_0) := 1 - \exp(-\gamma_0 (\varphi_{n+m}- \varphi_{n}))$.
Finally, the process collecting the points $(\theta_0, \gamma_0)$ of $\Lambda^{(0)}$ that generate a new cluster in the next $m$ observations is obtained by independently thinning $\Lambda^{(0)}$ with retention probability $p(\theta,\gamma) = q(\gamma)$. That is, such a process is Poisson with intensity $q(\gamma) e^{-\gamma \varphi_n} \tilde{\rho}(\dd \gamma) G_0(\dd \theta)$, which has a Poisson number of points with parameter $\int (1 - e^{-\gamma (\varphi_{n+m}- \varphi_{n})}) e^{-\gamma \varphi_n} \tilde{\rho}(\dd \gamma)$.

\subsection{Additional theoretical results}\label{app:add_res_cox_model}

First, we provide the marginal distribution of the sample $\bm Z$ under the Cox model prior in \eqref{def:psi_sncp}.

\begin{proposition}\label{prop:bayesian_sncp_features_marginal}
Let $\bm Z$ be a sample from \eqref{eq:representation_theorem}, where $\mu$ is the functional of the process $\Psi$ described in \eqref{def:psi_sncp}. Set
\[
    \varphi_n:=\int_{(0,1]}\{1-(1-s)^n\}\rho(\dd s).
\]    
Then, the marginal distribution of the sample $\bm Z$ is
\begin{equation*}
\begin{aligned}
&\exp\left\{-\int_{\Theta\times\R_+}(1-e^{-\gamma\varphi_n})\tilde\rho(\dd\gamma)G_0(\dd\theta)\right\} \times\prod_{\ell=1}^k\int_{(0,1]}s^{m_\ell}(1-s)^{n-m_\ell}\rho(\dd s)\\
&\qquad\times\sum_{\bm t\in\mathcal T_k}\prod_{h=1}^{|\bm t|}\int_{\Theta\times\R_+}e^{-\gamma\varphi_n}\gamma^{n_h}\left\{\prod_{\ell:t_\ell=h}\kappa(x_\ell^*;\theta)\right\}\tilde\rho(\dd\gamma)G_0(\dd\theta),
\end{aligned}
\end{equation*}
where $\mathcal T_k$ is the set of canonical allocation vectors $\bm t=(t_1,\ldots,t_k)$ describing partitions of $\{1,\ldots,k\}$ as in \Cref{lem:cox_disintegration}, and for $\bm t\in\mathcal T_k$, we write $|\bm t|=\max_{\ell\leq k}t_\ell$, $\bm x_h^*=(x_\ell^*:t_\ell=h)$ and $n_h=|\bm x_h^*|$.
\end{proposition}

\begin{proof}
By Theorem \ref{thm:marg}, the marginal distribution of $\bm Z$ is obtained from
\begin{equation*}
\int_{(0,1]^k}\E\left\{\exp\left(\int_{\X\times(0,1]}n\log(1-s)\Psi^!_{\bm x^*,\bm s}(\dd x\,\dd s)\right)\right\}\prod_{\ell=1}^k s_\ell^{m_\ell}(1-s_\ell)^{n-m_\ell}\rho^{(k)}(\dd\bm s\mid\bm x^*)\tilde m_\xi^{(k)}(\dd\bm x^*).
\end{equation*}
By \Cref{lem:cox_disintegration},
\[
    \rho^{(k)}(\dd\bm s\mid\bm x^*)=\prod_{\ell=1}^k\rho(\dd s_\ell), \quad \text{and} \quad \tilde m_\xi^{(k)}(\dd\bm x^*)=\left\{\sum_{\bm t\in\mathcal T_k}\prod_{h=1}^{|\bm t|}\eta(\bm x_h^*)\right\}\prod_{\ell=1}^k\dd x_\ell^*.
\]     
Thus, it remains to compute the term
\[
    \E\left\{\exp\left(\int_{\X\times(0,1]}n\log(1-s)\Psi^!_{\bm x^*,\bm s}(\dd x\,\dd s)\right)\right\}
\]
For ease of notation, set $H_n(\xi):=\exp\left\{\int_{\X\times(0,1]}n\log(1-s)\xi(\dd x\,\dd s)\right\}$.
From \Cref{thm:red_palm_gensncp_ibp}, 
\[
    \E[H_n (\Psi^!_{\bm x^*,\bm s})] = \E\left[H_n(\Psi) + \sum_{h=1}^{|\bm T|}H_n(\Psi_{\zeta_{\bm x_h^*}}) \right]
\]  
First, it is clear that $\E[\Psi_{\zeta_{{\bm{x}}_h}} \mid \zeta_{{\bm{x}}_h} = (\theta, \gamma)] = e^{-\gamma\varphi_n}$.
Second, by the Cox construction,
\begin{align*}
    \E[H_n(\Psi)] &= \E\left[\E\{H_n(\Psi)\mid\Lambda\}\right] \\
    &= \E\left[\exp\left\{-\varphi_n\int_{\Theta\times\R_+}\gamma\Lambda(\dd\theta\,\dd\gamma)\right\}\right]\\
    &= \exp\left\{-\int_{\Theta\times\R_+}(1-e^{-\gamma\varphi_n})\nu(\dd\theta\,\dd\gamma)\right\}.
\end{align*}   
Therefore, putting things together,
\begin{multline*}
\E\{H_n(\Psi^!_{\bm x^*,\bm s})\}  \tilde m_\xi^{(k)}(\dd\bm x^*) = 
\E\{H_n(\Psi^!_{\bm x^*,\bm s})\} \sum_{\bm t\in\mathcal T_k}\prod_{h=1}^{|\bm t|}\eta(\bm x_h^*)  \prod_{\ell=1}^k\dd x_\ell^* \\
=\exp\left\{-\int_{\Theta\times\R_+}(1-e^{-\gamma\varphi_n})\nu(\dd\theta\,\dd\gamma)\right\}\sum_{\bm t\in\mathcal T_k}\prod_{h=1}^{|\bm t|}\int_{\Theta\times\R_+}e^{-\gamma\varphi_n}\gamma^{n_h}\left\{\prod_{\ell:t_\ell=h}\kappa(x_\ell^*;\theta)\right\}\nu(\dd\theta\,\dd\gamma) \prod_{\ell=1}^k\dd x_\ell^*,
\end{multline*}
and the thesis follows by multiplying both sides for  $\prod_{\ell=1}^k\int_{(0,1]}s^{m_\ell}(1-s)^{n-m_\ell}\rho(\dd s)$.   
\end{proof}

Second, we focus on the $m$-step ahead prediction. In particular,
let $K_m^{(n)}$ be the number of new features discovered in a future sample of size $m$.

\begin{proposition}\label{prop:sncp_Kmn}
Conditional on $\bm T$, $\Gamma^{(0)} = \int \gamma \Lambda^{(0)}(\dd \theta\, \dd \gamma)$, and $(\theta_{\bm x^*_h}, \gamma_{\bm x^*_h}: h = 1,\ldots, C_n )$, the law of $K_m^{(n)}$ follows
     \[
K_m^{(n)} \mid \bm T, \Gamma^{(0)},(\theta_{\bm x^*_h}, \gamma_{\bm x^*_h}: h = 1,\ldots, C_n ) \sim \mathrm{Poisson}\!\left((\varphi_{n+m} - \varphi_n )\left( \Gamma^{(0)} +\sum_{h = 1}^{C_n}\gamma_{\bm x^*_h}\right)\right).
\]
\end{proposition}
\begin{proof}
Note that, conditionally to $\bm T$, $\Lambda^{(0)}$, and $(\theta_{\bm x^*_h}, \gamma_{\bm x^*_h}: h = 1,\ldots, C_n )$, the process $\Psi^\prime$ in \eqref{eq:joint_psi_prime_T} is a Poisson point process with mean measure 
\[
    \rho_n(\dd s) \times \left[ L^{(0)}(x) + \sum_{h=1}^{C_n} \gamma_{\bm x^*_h} k(x; \theta_{\bm x^*_h}) \right]\dd x.
\]
The process describing the new features appearing in the future sample of size $m$ is obtained by retaining each point in $\Psi^\prime$ independently with probability $1- (1-S^\prime_j)^m$. This results in a number of new features $K_m^{(n)}$ distributed as in statement of the thesis.   
\end{proof}

\subsection{Posterior sampling algorithm}\label{app:mcmc_sncp}

Recall the shot-noise representation $L_\Lambda (x)=\sum_{h\ge 1}\gamma_h \kappa(x; \theta_h)$ and the Cox process $\Psi=\sum_{j\geq 1} \delta_{(X_j,S_j)}$. Following \Cref{prop:bayesian_sncp_features_posterior}, the posterior introduces:
\begin{itemize}
\item feature-specific occurrence probabilities $S^*_\ell$ for the $k$ observed features $\bm x^* = (x^*_1,\ldots,x^*_k)$;
\item a clustering of the $k$ observed features into $C_n$ clusters, encoded by the vector of latent allocation variables
$\bm{T} =(T_1,\dots,T_k)$, with $T_\ell\in\{1,\dots,C_n\}$. Moreover, features block $\bm x_h^* = \{ x_\ell^*: T_\ell = h\}$, with size $n_h=|\bm x_h^*|$, collects the features allocated to cluster $h$, for $h=1,\ldots,C_n$; 
\item for each cluster $h = 1,\ldots,C_n$, a parent shot-noise parameter $\zeta_{\bm x_h^*}=(\theta_{\bm x_h^*},\gamma_{\bm x_h^*})$, where $\theta_{\bm x_h^*}$ indexes the kernel $\kappa(\cdot;\theta_{\bm x_h^*})$ and $\gamma_{\bm x_h^*}>0$ is the cluster amplitude.
\end{itemize}

Specific to the \textsc{tcga} application, we consider the \emph{extended Gaussian-\textsc{ibp}} described in \Cref{sec:ext_gaussian_ibp}, which takes $\theta=(\mu,\sigma^2)\in\mathbb{R}^d\times\mathbb{R}_+$ and $\kappa(\cdot; \theta)=\mathcal{N}(\cdot; \mu,\sigma^2 I_d)$, with $d=5$.
A key scalar quantity entering the posterior is
\begin{equation*}
\varphi_n=\int_{(0,1]}\bigl(1-(1-s)^n\bigr)\rho(s)\,ds,
\qquad
\rho(s)=c\,s^{-1-\alpha}(1-s)^{\beta+\alpha-1}\indicator_{(0,1]}(s),
\end{equation*}
which can be evaluated stably via the finite expansion
\begin{equation}\label{eq:phi-n-series}
\varphi_n
=c\sum_{i=1}^n(-1)^{i+1}\binom{n}{i}\,B(i-\alpha,\beta+\alpha),
\end{equation}
where $B(\cdot,\cdot)$ denotes the Beta function. The quantity $\varphi_n$ depends on the hyperparameters $(\alpha,\beta, c)$ of the model.
In the following, we describe the Gibbs scheme for posterior sampling from the \emph{extended Gaussian-\textsc{ibp}} model described in \Cref{sec:ext_gaussian_ibp}. 
Remind that, at each retained iteration, we obtain posterior draws of predictive quantities using \Cref{prop:sncp_Kmn} and \Cref{prop:sncp_Cmn}. In particular, in \Cref{prop:sncp_Kmn} under the assumed model, the total mass $\Gamma^{(0)}=\int \gamma\,\Lambda^{(0)} (\dd \theta \dd \gamma)$ has conditional distribution $\Gamma^{(0)}\mid \tau_0,b_0,\varphi_n \sim \mathrm{Gamma}(\tau_0,b_0+\varphi_n)$.

\subsubsection{Gibbs updates for the main parameters}\label{app:gibbs_main_params}

For notational convenience, let $\varepsilon := (\alpha,\beta,c,\tau_0,b_0)$ collect the hyperparameters of the model. Here, we report the Gibbs steps for updating the main parameters of the model conditionally to such hyperparameters.\\

\noindent\emph{Update of the observed feature probabilities $\bm S^*$.}
From \Cref{prop:bayesian_sncp_features_posterior} and the form of $\rho$, the full-conditionals factorize as
\begin{equation}\label{eq:q-fullcond}
S^*_{\ell}\mid \bm Z,\varepsilon \ind \mathrm{Beta}(m_\ell-\alpha,n-m_\ell+\beta+\alpha),
\qquad \ell=1,\dots,k.
\end{equation}
This update does not depend neither on the clustering $\bm T$ nor on the cluster-specific parameters $\zeta_{\bm x_h^*}=(\theta_{\bm x_h^*},\gamma_{\bm x_h^*})$, for $h=1,\ldots,C_n$.\\

\noindent\emph{Update of the clustering $\bm T$ (cluster allocations).}
We use a collapsed Gibbs scan over $\ell=1,\dots,k$. Let ${\bm T}_{-\ell}$ denote the allocations with feature $\ell$ removed. For any existing cluster $h$, let $\bm x_{h,-\ell}^*$ be the set of feature labels belonging to the cluster (excluding feature $\ell$), with size $n_{h,-\ell}$.
The conditional probability of assigning feature $\ell$ to an existing cluster $h$ is proportional to the ratio of cluster marginal likelihoods induced by \eqref{eq:distribution_T}:
\begin{equation}\label{eq:assign-existing}
\Pr(T_\ell=h\mid {\bm T}_{-\ell},x^*_\ell,\varepsilon)\propto
n_{h,-\ell} \,\,
p(x^*_{\ell}\mid \bm x_{h,-\ell}^*),
\end{equation}
where $p(x^*_{\ell}\mid \bm x_{h,-\ell}^*)=\int \kappa(x^*_{\ell};\theta)\,\pi(\dd \theta\mid \bm x_{h,-\ell}^*)$ is the posterior predictive density of $x^*_\ell$ under the base prior $G_0$ updated with the feature labels $\bm x_{h,-\ell}^*$ in cluster $h$.
The probability of creating a new cluster is
\begin{equation}\label{eq:assign-new}
\Pr(T_\ell=\text{new}\mid {\bm T}_{-\ell},x^*_\ell,\varepsilon)\propto
\tau_0 \,\,p_0(x^*_{\ell}),
\end{equation}
where $p_0(x^*_\ell)=\int \kappa(x^*_\ell;\theta)\,G_0(\dd \theta)$ is the prior predictive density of $x^*_\ell$.
After normalization across all existing clusters and the new-cluster option, we sample $T_\ell$ accordingly and relabel clusters to keep labels contiguous.\\

\noindent\emph{Update of cluster-specific parameters $\zeta_{\bm x_h^*}$.}
Conditionally to the clustering $\bm T$ (and \emph{rest}), cluster-specific parameters are independent. In particular, each $\zeta_{\bm x_h^*}$ is updated independently following the two following steps.

\begin{itemize}
    \item \emph{Update of $\gamma_{\bm x_h^*}$.}
Under the Gamma process prior for $\Lambda$, i.e.,  $\tilde{\rho}(\dd \gamma)=\tau_0 \gamma^{-1}e^{-b_0\gamma} \indicator_{\R_+}(\gamma) \dd \gamma$, the full-conditional for $\gamma_{\bm x_h^*}$ is
\begin{equation}\label{eq:gamma-fullcond}
\gamma_{\bm x_h^*}\mid \bm T,\bm Z,\varepsilon \sim \mathrm{Gamma}\bigl(n_h,b_0+\varphi_n\bigr),
\qquad h=1,\dots,C_n,
\end{equation}
where $\mathrm{Gamma}(a,b)$ denotes a Gamma distribution with shape $a$ and rate $b$.

\item \emph{Update of $\theta_{\bm x_h^*}=(\mu_{\bm x_h^*},\sigma_{\bm x_h^*}^2)$.}
Since $\kappa(\cdot;\theta)$ is Gaussian and $G_0$ normal-inverse-gamma, the updates are conjugate. Let $\bar{\bm x}_h^*$ be the empirical mean of feature labels in $\bm x_h^*$ and let 
\[
V_h=\sum_{x^*_\ell\in \bm x_h^*}\lVert x^*_{\ell}-\bar{\bm x}_h^*\rVert^2,
\]
where $\lVert \cdot \rVert$ denotes the Euclidean norm in $\R^d$. Then, the updates follow
\begin{equation}\label{eq:tau-fullcond}
\mu_{\bm x_h^*}\mid \sigma_{\bm x_h^*}^2,\bm T,\bm Z,\varepsilon \sim \mathcal{N}\Bigl(m_{\bm x_h^*},\sigma_{\bm x_h^*}^2(\lambda_{\bm x_h^*}^{-1})I_d\Bigr), \quad
\sigma_{\bm x_h^*}^2\mid \bm T,\bm Z,\varepsilon \sim \mathrm{Inv-Gamma}(a_{\bm x_h^*},b_{\bm x_h^*}),
\end{equation}
with 
\begin{align*}
\lambda_{\bm x_h^*}&=\lambda_0+n_h,
&
m_{\bm x_h^*}&=\frac{\lambda_0 m_0+n_h\bar{\bm x}_h^*}{\lambda_{\bm x_h^*}},
\\
a_{\bm x_h^*}&=a+\frac{n_h d}{2},
&
b_{\bm x_h^*}&=b+\frac{1}{2}\Bigl(V_h+\frac{\lambda_0 n_h}{\lambda_{\bm x_h^*}}\lVert \bar{\bm x}_h^*-m_0\rVert^2\Bigr).
\end{align*}

\end{itemize}

\subsubsection{Gibbs updates for global hyperparameters}

As discussed in \Cref{sec:ext_gaussian_ibp}, we put hyperpriors on $\varepsilon = (\alpha,\beta,c,\tau_0,b_0)$ in order to let the model adapt to the observed data. In particular, consider the hyperpriors detailed in \Cref{sec:sncp_genomics} for the \textsc{tcga} application. 
Here, we report the updating step for $\varepsilon$, conditionally to $\bm Z$. We denote the prior distribution on $\varepsilon$ by $p(\varepsilon)$. Therefore, we aim at drawing $\varepsilon$ from 
\begin{equation}\label{eq:post_hyper}
    p(\varepsilon\mid \bm Z) \propto p(\bm Z\mid \varepsilon) \, p(\varepsilon),
\end{equation}
where $p(\bm Z\mid \varepsilon)$ denotes the marginal distribution of $\bm Z$, conditionally to the hyperparameters $\varepsilon$, which is given in \Cref{prop:bayesian_sncp_features_marginal}. 
\\

From \Cref{prop:bayesian_sncp_features_marginal}, it holds
\[
p(\bm Z\mid \varepsilon) \propto e^{-A_n(\varepsilon)} \left\{ \prod_{\ell=1}^k q_\ell(\varepsilon) \right\}
\sum_{\bm t\in\mathcal{T}_k} \prod_{h=1}^{|\bm t|} w_{\bm x^*_h}(\varepsilon),
\]
where $\propto$ is up to constant not depending on $\varepsilon$, and 
\[
A_n(\varepsilon) = \int_{\R_+} \{1-e^{-\gamma\varphi_n(\varepsilon)} \} \tilde{\rho}(\dd \gamma), \qquad 
q_\ell(\varepsilon) = \int_{(0,1]} s^{m_\ell}(1-s)^{n-m_\ell}\rho(\dd s),
\]
and
\[
r_{\bm x^*_h}(\varepsilon)
=
\int
\prod_{\ell: t_\ell = h }
\kappa(x_\ell^*;\theta)
G_0(\dd\theta), \qquad w_{\bm x^*_h}(\varepsilon)
=
r_{\bm x^*_h}(\varepsilon)
\int_{\R_+}
e^{-\gamma\varphi_n(\varepsilon)}
\gamma^{n_h}
\tilde{\rho}(\dd\gamma).
\]
The expression of $p(\bm Z\mid \varepsilon)$ suggests to consider the augmented probability distribution
\[
p_{\mathrm{aug}}(\bm Z, \bm t \mid \varepsilon)
\propto
e^{-A_n(\varepsilon)}
\left\{
\prod_{\ell=1}^k q_\ell(\varepsilon)
\right\}
\prod_{h=1}^{|\bm t|}
w_{\bm x^*_h}(\varepsilon),
\qquad
 \bm t\in\mathcal{T}_k,
\]
where $\propto$ is up to constant not depending on $\varepsilon$.

Now, instead of directly targeting the posterior $p(\varepsilon\mid \bm Z)$ in \eqref{eq:post_hyper}, we target the augmented posterior
\begin{equation}\label{eq:post_hyper_aug}
 p(\varepsilon,\bm t \mid \bm Z)
\propto
p_{\mathrm{aug}}(\bm Z, \bm t \mid \varepsilon)\, p(\varepsilon),   
\end{equation}
whose marginal distribution in $\varepsilon$ coincides with the desired posterior distribution in \eqref{eq:post_hyper}.
The \textsc{mcmc} update then requires the sampling from $p(\varepsilon,\bm t \mid \bm Z)$ in \eqref{eq:post_hyper_aug}, which is then performed alternating two steps: (i) sampling from $p(\bm t \mid \varepsilon, \bm Z)$, (ii) sampling from $p(\varepsilon\mid \bm t, \bm Z)$. Specifically, we have that
\begin{equation}\label{eq:post_t_givenhyper}
  p(\bm t \mid \varepsilon, \bm Z) \propto 
\prod_{h=1}^{C_n}
w_{\bm x^*_h}(\varepsilon),  
\end{equation}
and 
\begin{equation}\label{eq:post_hyper_givenT}
    p(\varepsilon\mid \bm t, \bm Z) \propto e^{-A_n(\varepsilon)}
\left\{
\prod_{\ell=1}^k q_\ell(\varepsilon)
\right\}
\prod_{h=1}^{C_n}
w_{\bm x^*_h}(\varepsilon) \times p(\epsilon).
\end{equation}
Notably, the draw from \eqref{eq:post_t_givenhyper} does not need to be actually performed, since the sampled vector of latent allocation variables $\bm T$ in \eqref{app:gibbs_main_params} is drawn from the same distribution, so we condition to that value $\bm T$ when sampling in \eqref{eq:post_hyper_givenT}.\\

Specializing the distribution \eqref{eq:post_hyper_givenT} to the \emph{extended Gaussian-\textsc{ibp}} model, the involved quantities write as follows:
\[
A_n(\varepsilon)
=
\tau_0
\log\left(\frac{b_0+\varphi_n(\varepsilon)}{b_0}\right), \qquad q_\ell(\varepsilon)
=
c\,B(m_\ell-\alpha,n-m_\ell+\beta+\alpha),
\]
and $r_{\bm x^*_h}(\varepsilon)$ does not depend on $\varepsilon$, so that
\[
w_{\bm x^*_h}(\varepsilon)
=
r_{\bm x^*_h}
\int_0^\infty
e^{-\gamma\varphi_n(\varepsilon)}
\gamma^{n_h}
\tau_0\gamma^{-1}e^{-b_0\gamma}\dd\gamma
=
r_{\bm x^*_h}\,
\tau_0
\frac{\Gamma(n_h)}
{\{b_0+\varphi_n(\varepsilon)\}^{n_h}}.
\]
Therefore, the update in \eqref{eq:post_hyper_givenT} specializes to
\begin{equation*}
\begin{aligned}
    p(\varepsilon\mid \bm t, \bm Z) &\propto \exp \left\{- \tau_0
\log\left(\frac{b_0+\varphi_n(\varepsilon)}{b_0}\right) \right\}\,  c^k  \prod_{\ell=1}^k B(m_\ell-\alpha,n-m_\ell+\beta+\alpha) \\
&\quad \times {\tau_0}^{C_n} \{b_0+\varphi_n(\varepsilon)\}^{-k}   \times p(\varepsilon)  
\end{aligned}
\end{equation*}
since $\prod_{h=1}^{C_n} \{ r_{\bm x^*_h} \Gamma(n_h) \}$ does not depend on $\varepsilon$. In particular, the hyperparameters in $\varepsilon$ are updated sequentially through the Gibbs sampling approach. Under the prior $\tau_0 \sim \mathrm{Gamma}(a_{\tau_0},b_{\tau_0})$, then the full-conditional of $\tau_0$ is
\begin{equation}\label{eq:theta0-fullcond}
\tau_0\mid \bm t,\bm Z, \alpha, \beta, c, b_0 \sim \mathrm{Gamma}\left(a_{\tau_0}+C_n,b_{\tau_0}+ \log\left(\frac{b_0+\varphi_n(\varepsilon)}{b_0}\right) \right).
\end{equation} 
The remaining hyperparameters $(b_0,\alpha,\beta,c)$ are updated by Metropolis-Hastings steps.

\section{Proofs and additional details of \Cref{sec:dpp_model_and_application}}\label{app:dpp_prior_model}

\subsection{The independently marked (repulsive) determinantal process prior: proof of \Cref{thm:bayesian_dpp} and details of \Cref{cor:dpp_beta}}

\noindent \emph{Proof of point (i) of \Cref{thm:bayesian_dpp}}.
The marginal distribution of $\bm Z$ is recovered from \Cref{thm:marg} as follows. 
Since $\Psi$ is an independently marked process with ground process $\xi$ and mark kernel $H$, then from \cite[Proposition 3.2.14]{BaBlaKa}, the reduced Palm version $\Psi^!_{x,s}$ is still an independently marked process, with ground process $\xi^!_x$ and mark kernel $H$, thus it does not depend on $s$.
By Palm algebra, we can extend this property by claiming that $\Psi^!_{\bm x,\bm s}$ is an independently marked process, with ground process $\xi^!_{\bm x}$ and mark kernel $H$. 
Then,
\begin{equation}\label{eq:ev_marginal_dpp}
   \begin{aligned}
    \E\left\{ e^{ \int_{\X \times (0, 1]} n \log(1 - t) \Psi^!_{\bm x^*, \bm s}(\dd z \, \dd t)} \right\} &= \mathcal{L}_{\Psi^!_{\bm x^*, \bm s}}(-n\log(1-t))\\
    &= \mathcal{L}_{\xi^!_{\bm x^*}} \left[ -\log\left\{\int_{(0,1]} (1-t)^n H(\dd t\mid x) \right\} \right],
\end{aligned} 
\end{equation}
where the second equality follows from \cite[Proposition 2.2.20]{BaBlaKa}. The thesis in point (i) of the statement follows.\\

\noindent \emph{Proof of point (ii) of \Cref{thm:bayesian_dpp}}.
For the posterior distribution of $\mu$, expressed in \Cref{thm:post}, we need to determine the law of the vector $\bm S^* = (S^*_1,\ldots,S^*_k)$ and the law of $\mu^\prime$, conditionally to $\bm S^*$. From point (i) of \Cref{thm:post} and \eqref{eq:ev_marginal_dpp}, which does not depend on $\bm s$, the $S^*_\ell$'s are independent with marginal laws $f_{S^*_\ell}(\dd s) \propto s^{m_\ell}(1 - s)^{n-m_\ell} H(\dd s\mid x^*_\ell)$. Moreover, by point (ii) in \Cref{thm:post}, the law of $\Psi^\prime$, conditionally to $\bm S^*$, has Laplace functional given by
\begin{equation*}
    \begin{aligned}         \mathcal{L}_{\Psi^\prime\mid \bm S^* = \bm s^*}(g) &= \frac{\mathcal{L}_{\Psi^!_{\bm x^*, \bm s^*}}(g(x,s) - n\log(1-s))}{ \mathcal{L}_{\Psi^!_{\bm x^*, \bm s^*}}(- n\log(1-s))}\\
         &= \frac{\mathcal{L}_{\xi^!_{\bm x^*}} \left[ -\log\left\{\int_{(0,1]} e^{-g(x,s)} (1-s)^n H(\dd s\mid x) \right\} \right]}{\mathcal{L}_{\xi^!_{\bm x^*}} \left[ -\log\left\{\int_{(0,1]} (1-s)^n H(\dd s\mid x) \right\} \right]}.
    \end{aligned}
\end{equation*}
By \cite[Proposition 2.2.20]{BaBlaKa}, this equals the Laplace transform of the independently marked point process $\Psi^\prime = \sum_{j \geq 1} \delta_{(X^\prime_j, S^\prime_j)}$ where $S^\prime_j \mid X^\prime_j = x^\prime_j \sim H^\prime(\cdot \mid x^\prime_j) \propto (1 -s)^n H(\dd s \mid x^\prime_j)$ and $\xi^\prime = \sum_{j\geq 1} \delta_{X^\prime_j}$ has Laplace transform as in the statement of \Cref{thm:bayesian_dpp}.\\

\noindent \emph{Details of \Cref{cor:dpp_beta}}.
The results in \Cref{cor:dpp_beta} are obtained by specializing the treatment to the case where the mark kernel $H(\cdot\mid x )$ corresponds to the law of a beta with parameters $(a,b)$. Under this assumption, we get that the law of $M^\prime$ can be determined as follows
\begin{align*}
    \mathcal L_{M^\prime}(u) &= \E\left\{e^{-u \xi^\prime(\X)}\right\} =\mathcal L_{\xi^\prime}(u \indicator_{\X} ) \\
    & = \frac{\mathcal{L}_{\xi^!_{\bm x^*}} \left[ u -\log\left\{\int_{(0,1]}  (1-s)^n \mathrm{Beta}(\dd s \,; a, b) \right\} \right]}{\mathcal{L}_{\xi^!_{\bm x^*}} \left[ -\log\left\{\int_{(0,1]} (1-s)^n \mathrm{Beta}(\dd s \,; a, b) \right\} \right]} \\
    & = \frac{\mathcal{L}_{\xi^!_{\bm x^*}}  \left[ u - \log\left\{B(a, b+n) / B(a, b)\right\}\right]}{\mathcal{L}_{\xi^!_{\bm x^*}}  \left[ - \log\left\{ B(a, b+n) / B(a, b)\right\}\right]}.
\end{align*}
Moreover, we focus on the mean measure of $\xi^\prime$, denoted with $M_{\xi^\prime}$. Indicating with $f_{\xi^\prime}$ the density of the law of $\xi^\prime$ with respect to the law of $\xi^!_{\bm x^*}$, we have
\[
M_{\xi^\prime} (A) = \E\left\{f_{\xi^\prime}(\xi^!_{\bm x^*}) \xi^!_{\bm x^*}(A) \right\} = \E\left\{ \int_{\X} \indicator_A(y) f_{\xi^\prime}(\xi^!_{\bm x^*}) \xi^!_{\bm x^*}(\dd y) \right\}.
\]
By applying the alternative statement of the \textsc{clm} formula in \Cref{lem:clm} in terms of the Palm distributions, we obtain that the mean measure of $\xi^\prime$ equals
\begin{align*}
M_{\xi^\prime} (A) &= \int_{\X} \E \left[ \indicator_A(y) f_{\xi^\prime}\left\{\left(\xi^!_{\bm x^*} \right)_y \right\}  \right] M_{\xi^!_{\bm x^*}}(\dd y) \\
&= \int_{\X} \E \left[ \indicator_A(y) f_{\xi^\prime}\left\{\left(\xi^!_{\bm x^*} \right)^!_y + \delta_y \right\}  \right] M_{\xi^!_{\bm x^*}}(\dd y) \\
&= \int_{\X} \E \left[ \indicator_A(y) f_{\xi^\prime}\left\{\xi^!_{(\bm x^*, y)} + \delta_y \right\}  \right] M_{\xi^!_{\bm x^*}}(\dd y) \\
 &= g(n; a, b) \int_{A} \E \left[ f_{\xi^\prime}\left\{ \xi^!_{(\bm x^*, y)} \right\} \right] M_{\xi^!_{\bm x^*}}(\dd y).
\end{align*}

\subsection{Fitting details and numerical implementation} \label{app:fitting_dpp_model}

As mentioned in \Cref{sec:spruces}, the two main tasks in this setting, i.e., predicting the number and locations of the missing trees, require to handle the distribution of $\xi^!_{\bm x}(\X)$, for some specific points $\bm x$. Here, we go into the details of this point, and related computational aspects.
From \Cref{thm:bayesian_dpp}, $\xi^!_{\bm x}$ is a \textsc{dpp} on $R$ with kernel $ K_{\bm x}$, and associated Mercer decomposition $K_{\bm x}(y_1, y_2) = \sum_{k \geq 1} \lambda^*_k \varphi^*_k(y_1) \overline{\varphi^*_k(y_2)}$. 
Then, from \cite{Hou(06)} \cite[see also][]{Lav(15)}, we have that $\xi^!_{\bm x}(\X)$ follows a Poisson-binomial distribution with parameters  $(\lambda^*_k)_{k \geq 1}$.
Unfortunately, as discussed in \cite{LavRub23}, the eigendecomposition of $ K_{\bm x}$ is generally not analytically available. We consider two approaches: the first consists of approximating the eigendecomposition numerically, and the second exploits an approximation of the Poisson-binomial distribution that does not require such an eigendecomposition.

To numerically approximate the $\lambda^*_k$'s, we proceed as follows.
Let $R_g$ be a grid of $(N_g)^2$ equispaced points in $R$, define  the $N_g^2 \times N_g^2$ matrix $\bm{\tilde K}$ by $\bm{\tilde K}_{i, j} = K_{\bm x}(y_{1,i}, y_{2,j})$ for $(y_{1,i}, y_{2,j}) \in R_g \times R_g$. Let $\Delta = \prod_{i=1}^2 L_i / N_g$ where $L_i$ is the length of the $i$-th side of $R$. 
Of course, $\bm{\tilde K}$ is positive definite since $K_{\bm x}$ is a covariance kernel, with eigenvalues $\tilde \lambda_1, \ldots, \tilde \lambda_{N^2_g}$. 
Then, we set $\lambda^*_k \approx \tilde \lambda_k \Delta$ for $k=1, \ldots, N^2_g$ and $\lambda^*_k = 0$ for $k >N^2_g$, and approximate the law of $\xi^!_{\bm x}(\X)$ with a Poisson-binomial distribution of parameters $(\tilde \lambda_1 \Delta, \ldots,  \tilde \lambda_{N^2_g} \Delta)$. 
In our experiments presented in the subsequent sections, the decreasingly ordered sequence of $ \lambda^*_k$ decreases extremely fast with $k$ so that the truncation error is negligible. 
We checked the accuracy of the numerical approximation of $\lambda^*_k$ against the analytical value in the case of a Gaussian covariance, reporting errors of the order of $0.01$ for the largest $50$ eigenvalues.
To compute the probability mass function of a Poisson-binomial distribution, we use the Python package \texttt{fast-poibin}. Evaluating the distribution on the whole support takes less than one second for $N^2_g = 2500$. Further speed-ups can be achieved by truncating the series of eigenvalues earlier, for example, by keeping only those eigenvalues exceeding a pre-specified threshold. 

An alternative strategy would focus on approximating the Poisson-binomial distribution via Le Cam's theorem \citep{steele1994cam}, i.e., approximating the law of $\xi^!_{\bm x}(\X)$ with a Poisson distribution with parameter $\sum_{k \geq 1} \lambda^*_k$. Since the sum of eigenvalues is equal to the trace of $K_{\bm x}$, i.e., $\int_{R} K_{\bm x}(y, y) \dd y$, using Le Cam's approximation does not require performing any numerical eigendecomposition. 
However, Le Cam's approximation introduces a non-negligible error. In particular, the total variation between the true distribution of $\xi^!_{\bm x}(\X)$ and its approximation is bounded by above by $\sum_{k\geq1} (\lambda^*_k)^2$.

\subsection{Synthetic scenarios}\label{app:synthetic_dpp_model}

\begin{figure}[ht]
    \centering
    \includegraphics[width=\linewidth]{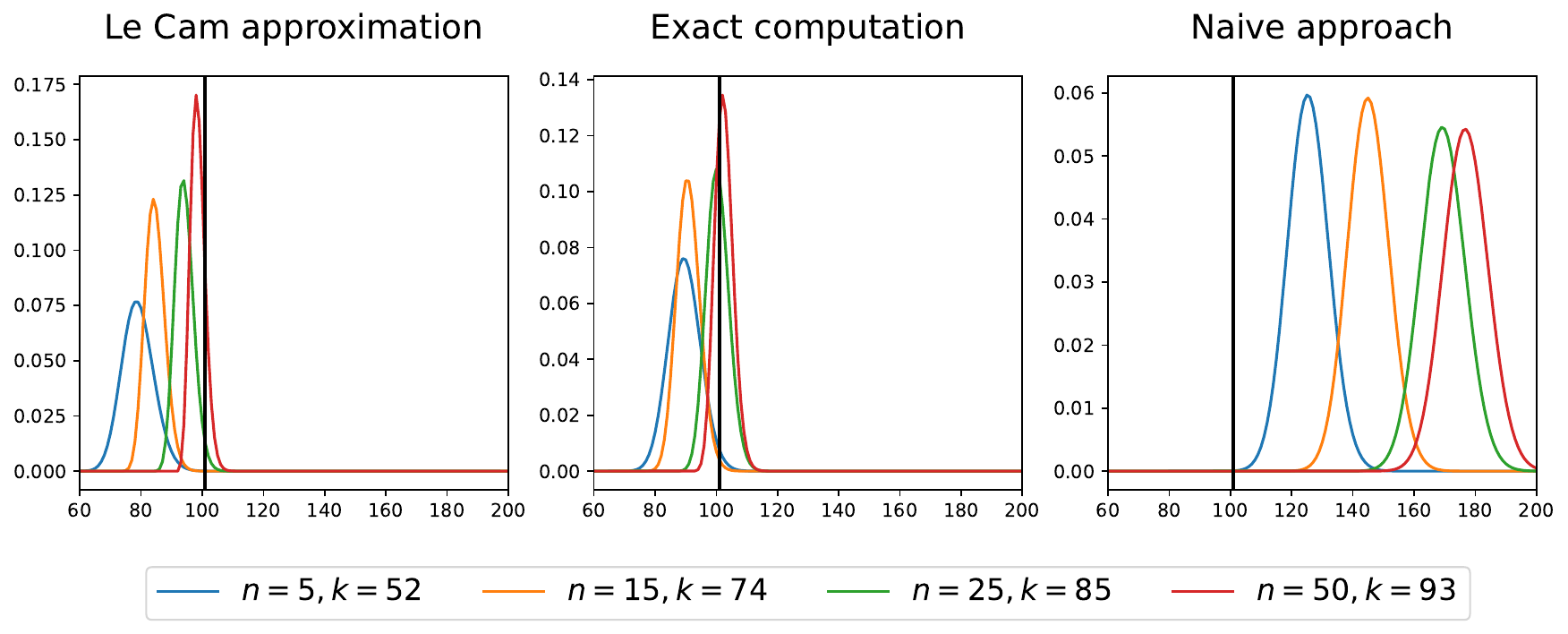}
    \caption{Posterior distribution of the total number of trees in the synthetic scenario. From left to right: calculations performed using Le Cam's approximation of the Poisson-binomial, exact computations, and posterior of $\xi^!_{\bm x^*}(\X)$. Different line colors correspond to different sample sizes; the black vertical line indicates the true number of trees. }
    \label{fig:ntree_simulation}
\end{figure}

\begin{figure}[ht]
    \centering
    \includegraphics[width=\linewidth]{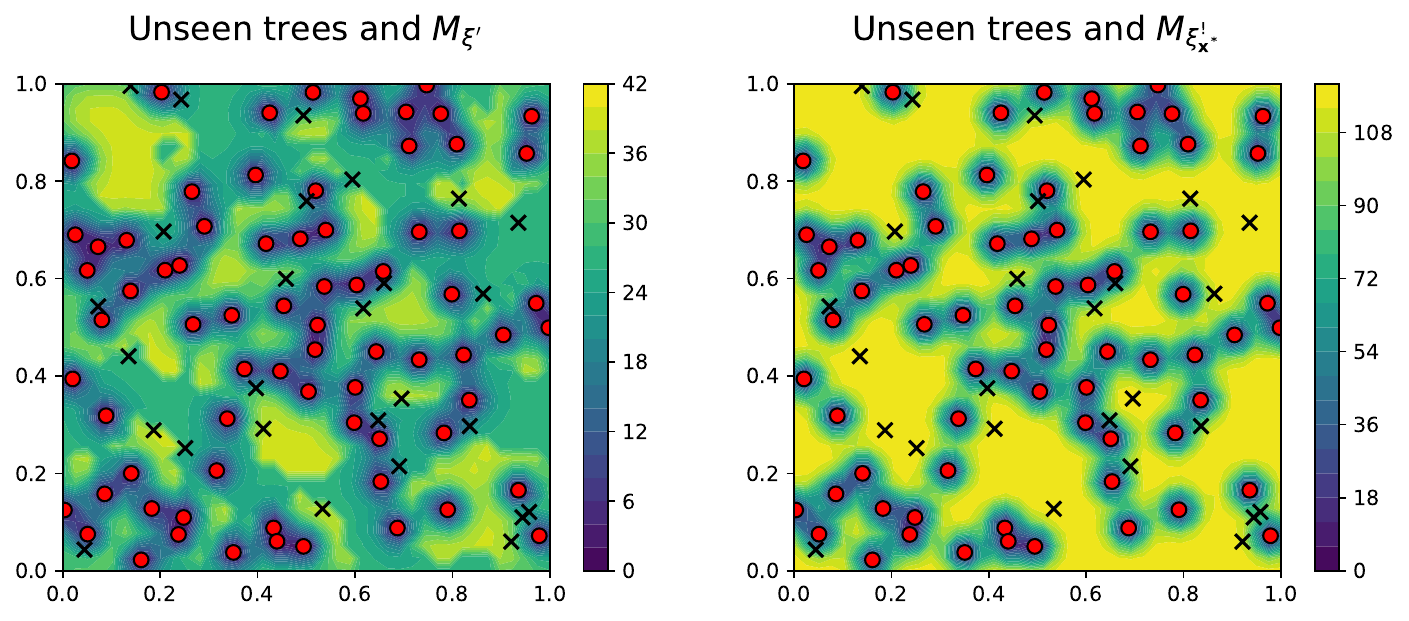}
    
    \caption{Locating the unobserved trees for $n = 15$ in the synthetic scenario: infinitesimal probability of observing an unseen tree in a given location. Left plot: the mean measure of $\xi^\prime$. Right plot: the mean measure of $\xi^!_{\bm x^*}$. The red dots represent the observed trees in the sample. The black crosses indicate the unseen trees. Note that the color scales of the two plots are different. }
    \label{fig:mean_measure_simulation}
\end{figure}

We generate data by simulating the true point process of the ``trees'', denoted with $\xi_{0}$, from a Gaussian \textsc{dpp} on $[0, 1]^2$ with intensity parameter $\rho = 100$ and scale $\alpha = 0.0535$. To the points in $\xi_0$ we attach i.i.d. marks $S_j$'s from the beta distribution of parameters $(1, 5)$, therefore obtaining a realization for $\Psi$ and $\mu$. Observations $Z_i$'s are obtained by simulating i.i.d. Bernoulli processes conditionally to $\mu$, mimicking the collection of data from the surveyors. 
 
For different sample sizes $n \in \{5, 15, 25, 50\}$, we compute the distribution of the total number of trees as $ M^\prime + k$, where $M^\prime$ is defined in \Cref{cor:dpp_beta}, and compare it with the total number of trees in $\xi_0$. 
We compare inference obtained using the exact computation of the Poisson-binomial distribution and Le Cam's approximation. Additionally, we consider a naive approach where we disregard the $Z_i$'s, retaining only the distinct locations $\bm x^*$ and computing the law of $\xi_{\bm x^*}^!$. Beyond predicting the total number of trees, we address the problem of locating the unobserved trees through $M_{\xi^\prime}$, as discussed in Section \ref{app:fitting_dpp_model}. When adopting the naive approach, this task is tackled by considering the mean measure of $\xi^!_{\bm x^*}$. All results presented here are obtained by estimating the hyperparameters of the model via the empirical Bayes approach described in Section \ref{app:fitting_dpp_model}. For completeness,  \Cref{app:details_spatial} provides the corresponding analyses under an oracle scenario, where all hyperparameters are fixed at their true values.

Figures \ref{fig:ntree_simulation} and \ref{fig:mean_measure_simulation} highlight the key features of our predictions: in estimating the total number of trees, Figure \ref{fig:ntree_simulation} clearly shows that the naive approach performs poorly across all sample sizes. Le Cam's approximation, albeit faster, introduces some errors that lead to slightly underestimating the number of trees. In contrast, the exact posterior looks centered around the true number of trees, with its variance shrinking as the sample size increases. Figure \ref{fig:mean_measure_simulation} addresses the problem of locating unseen trees for a sample size of $n=15$, comparing our proposed model with the naive approach. Refer to Section \ref{app:details_spatial} for the corresponding plots when $n \in \{5, 25, 50\}$. Beyond the substantial discrepancy between the two scales, caused by the naive approach significantly overestimating the number of trees, further remarks are needed. Specifically, consider the infinitesimal probability of observing an unseen tree at location $\dd x$. Since $\xi^!_{\bm x^*}$ is a \textsc{dpp}, the naive approach assumes that the farther $\dd x$ is from the observed trees $\bm x^*$, the more likely it is to contain an unseen tree. On the other hand, the repulsive pattern of $\xi^\prime$ is peculiar, leading to a different behavior: our model predicts that it is unlikely to find a tree at $\dd x$ if it is too close to an observed tree, as expected. However, it also predicts that locations ``too far'' from observed trees may have small probabilities of containing a tree, as they might not align with the accumulation patterns estimated from the data.

Finally, we briefly discuss the utility of imposing repulsiveness between the points $X_j$'s of $\Psi$. Consider any prior process for $\Psi$ where the points $X_j$'s are marginally i.i.d. from some diffuse distribution $G_0$, such as the Poisson, mixed Poisson, or mixed binomial processes. From the predictive characterizations in \Cref{teo:pred_char} and \Cref{teo:pred_char}, we know that the locations of the unseen trees $X_j^\prime$'s in $\Psi^\prime$ are independent of the observed locations $\bm x^*$. Furthermore, these locations remain marginally i.i.d. from $G_0$, as demonstrated in \Cref{thm:james}, \Cref{thm:mixed_poisson_bayesian}, and \Cref{thm:mixed_binom_bayesian}. Therefore, for these models, the task of locating the unobserved trees is answered by $G_0$, which provides no informative structure about the spatial arrangement of the unseen trees. This highlights the advantage of incorporating repulsiveness in this spatial illustration, as our model provides a more informative and structured prediction of tree locations.

\subsubsection{Additional details about the synthetic scenarios}\label{app:details_spatial}

We present here the analysis of the synthetic scenarios from \Cref{app:synthetic_dpp_model} under the oracle strategy, which assumes knowledge of the true values of all hyperparameters used to generate the data. First, we examine inference on the total number of trees, $M^\prime + k$, across different sample sizes, $n \in \{5, 15, 25, 50\}$. Figure \ref{fig:ntree_oracle_simulation} compares results obtained via the exact computation of the Poisson-binomial distribution and Le Cam's approximation. Additionally, we assess the naive approach described in  \Cref{app:synthetic_dpp_model}. The same observations made for Figure \ref{fig:ntree_simulation} apply here. Specifically, the naive approach performs poorly across all sample sizes when estimating the total number of trees. While Le Cam's approximation is computationally faster, it introduces some errors, leading to a slight underestimation of the tree count. In contrast, the exact posterior distribution is centered around the true number of trees, with its variance decreasing as the sample size increases.

Next, we address the problem of locating unobserved trees through 
$M_{\xi^\prime}$, as discussed in Section \ref{app:fitting_dpp_model}. Under the naive approach, this task is performed using the mean measure of 
$\xi^!_{\bm x^*}$. Figure \ref{fig:mean_measure_oracle_simulation} compares inference from our model and the naive alternative for 
$n=15$, while Figure \ref{fig:mean_measure_multiple_training_oracle_simulation} presents results from our model for 
$n \in \{5,25,50\}$. Similar insights to those discussed for Figure \ref{fig:mean_measure_simulation} apply. In particular, Figure \ref{fig:mean_measure_oracle_simulation} shows that the naive approach assumes a higher probability of finding an unobserved tree at locations farther from the observed trees 
$\bm x^*$, as inference is based on 
$\xi^!_{\bm x^*}$, which is a \textsc{dpp}. In contrast, Figures \ref{fig:mean_measure_oracle_simulation} and \ref{fig:mean_measure_multiple_training_oracle_simulation} reveal the distinctive repulsive structure of 
$\xi^\prime$ across different sample sizes. Specifically, our model predicts a low probability of finding a tree at 
$\dd x$ if it is too close to an observed tree, as expected. However, it also suggests that locations "too far" from observed trees may have a small probability of containing a tree, as they might not align with the accumulation patterns inferred from the data.

For completeness, Figure \ref{fig:mean_measure_multiple_training_fitted_simulation} presents inference on the locations of unobserved trees through $M_{\xi^\prime}$ for $n \in \{5,25,50\}$ under the scenario where hyperparameters are estimated via the empirical Bayes approach from \Cref{app:synthetic_dpp_model}. As aimed for, the results are consistent with those obtained under the oracle strategy, as evidenced by the clear similarity between Figures \ref{fig:mean_measure_multiple_training_fitted_simulation} and~\ref{fig:mean_measure_multiple_training_oracle_simulation}.

\begin{figure}[ht]
    \centering
    \includegraphics[width=\linewidth]{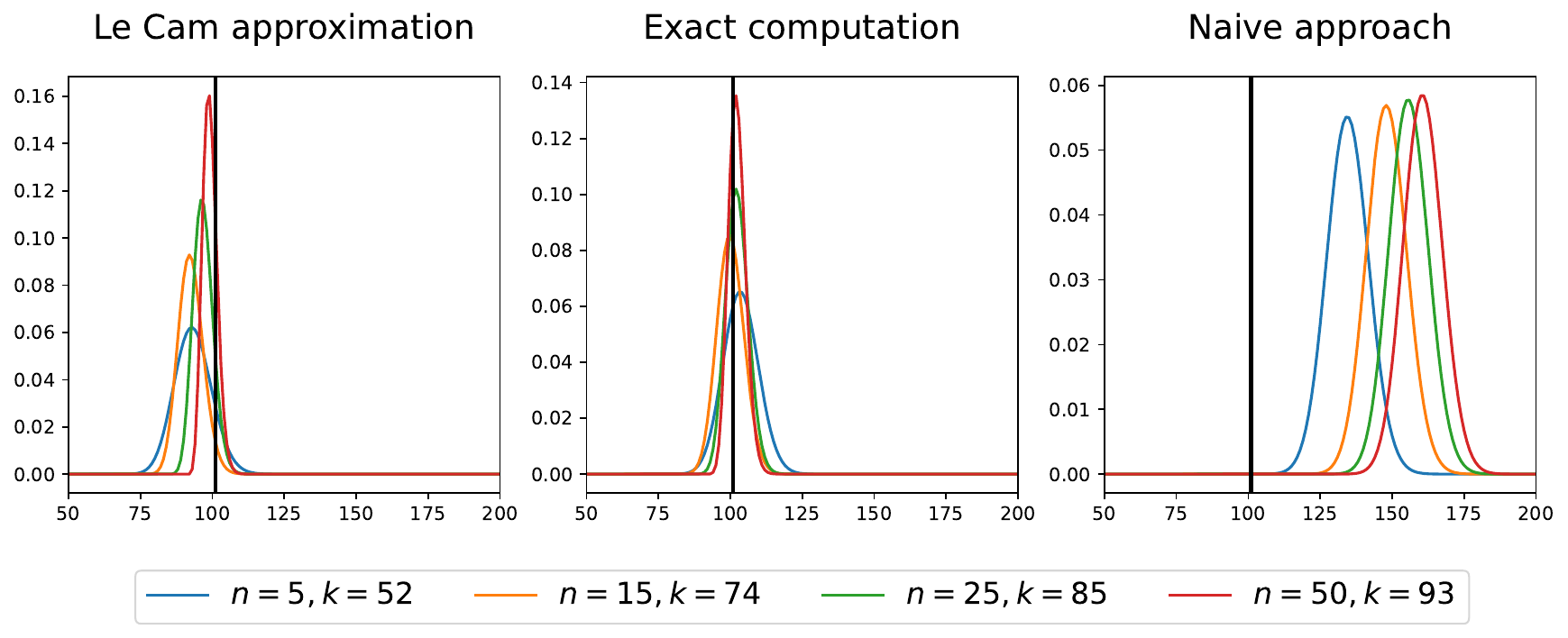}

    \caption{Posterior distribution of the total number of trees. From left to right: calculations performed using Le Cam's approximation of the Poisson-binomial, exact computations, and posterior of $\xi^!_{\bm x^*}(\X)$. Different line colors correspond to different sample sizes; the black vertical line indicates the true number of trees.}
    \label{fig:ntree_oracle_simulation}
\end{figure}

\begin{figure}[ht]
    \centering
    \includegraphics[width=\linewidth]{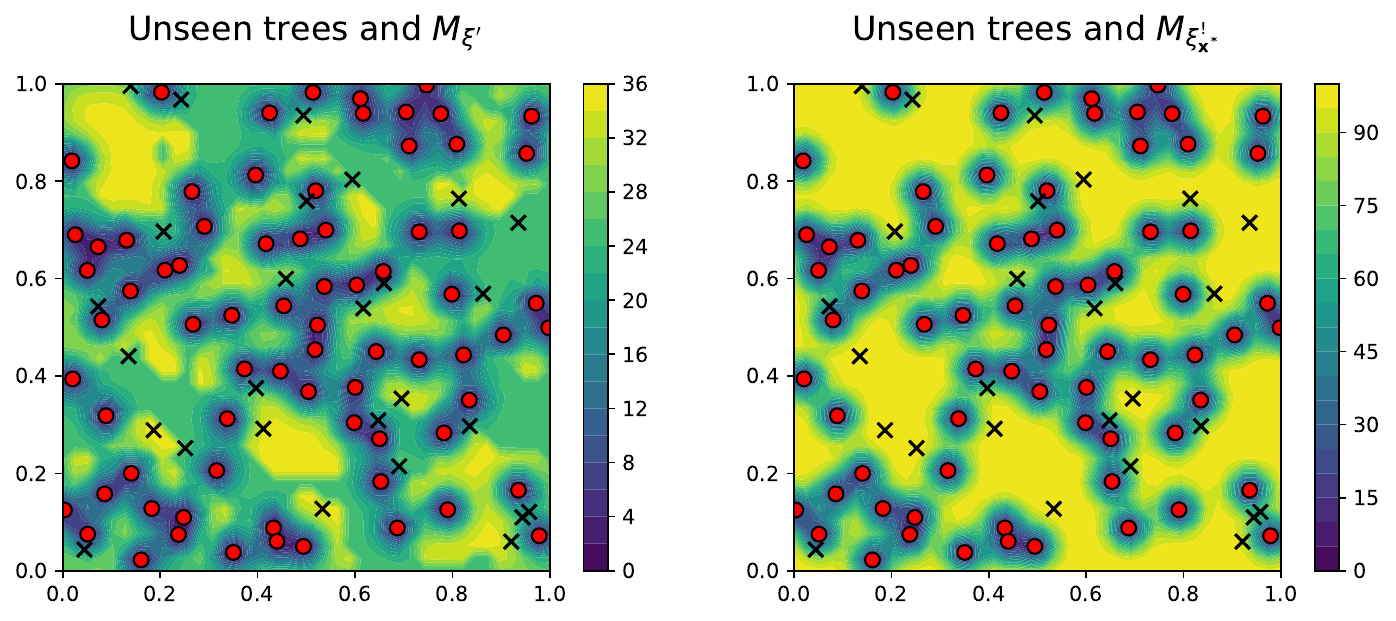}
    
    \caption{Locating the unobserved trees for $n = 15$: infinitesimal probability of observing an unseen tree in a given location. Left plot: the mean measure of $\xi^\prime$. Right plot: the mean measure of $\xi^!_{\bm x^*}$. The red dots represent the observed trees in the sample. The black crosses indicate the unseen trees. Note that the color scales of the two plots are different. }
    \label{fig:mean_measure_oracle_simulation}
\end{figure}

\begin{figure}[ht]
    \centering
    \includegraphics[width=\linewidth]{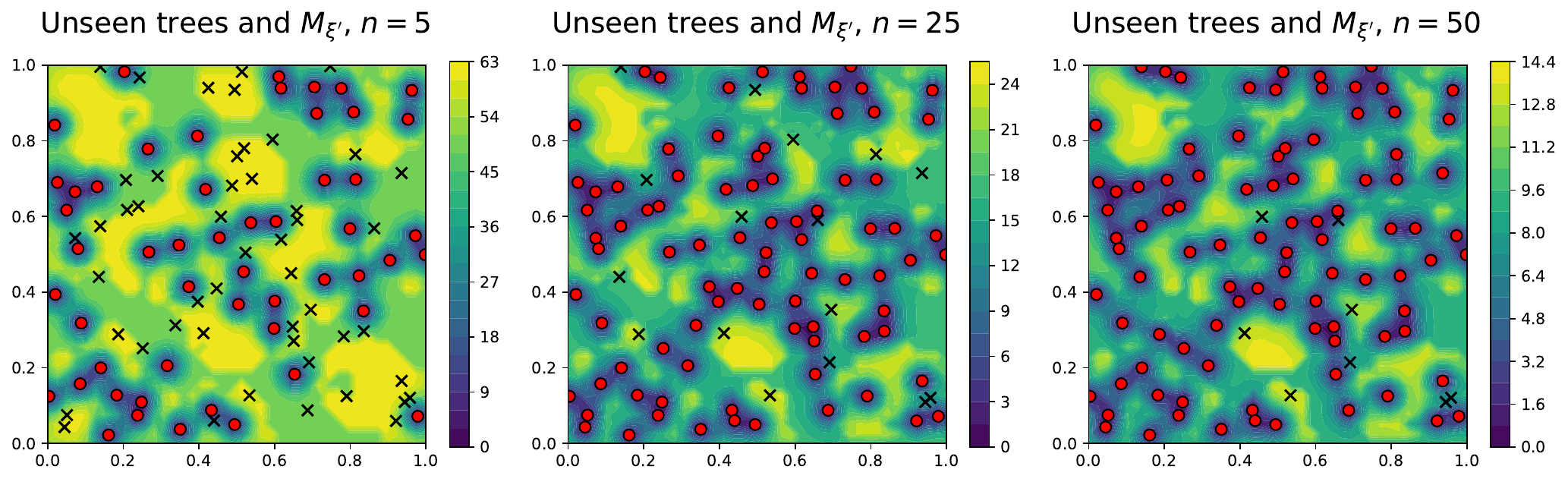}
    
    \caption{Locating the unobserved trees for $n \in \{5, 25,50\}$: infinitesimal probability of observing an unseen tree in a given location. The three plots report $M_{\xi^\prime}$ for the three sample sizes. The red dots represent the observed trees in the sample. The black crosses indicate the unseen trees. Note that the color scales of the plots are different. }
    \label{fig:mean_measure_multiple_training_oracle_simulation}
\end{figure}

\begin{figure}[ht]
    \centering
    \includegraphics[width=\linewidth]{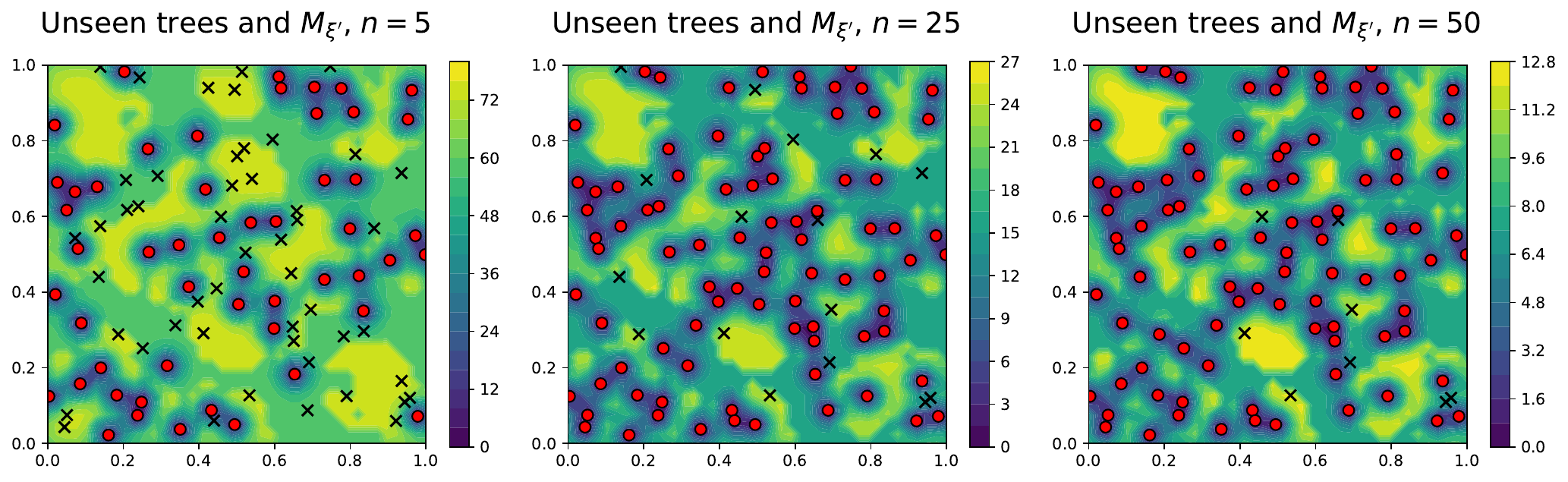}
    
    \caption{Locating the unobserved trees for $n \in \{5, 25,50\}$: infinitesimal probability of observing an unseen tree in a given location. The three plots report $M_{\xi^\prime}$ for the three sample sizes. The red dots represent the observed trees in the sample. The black crosses indicate the unseen trees. Note that the color scales of the plots are different. }
    \label{fig:mean_measure_multiple_training_fitted_simulation}
\end{figure}

\section{Detailed analysis of specific extended feature models}\label{app:examples_bayesian_analysis}

\Cref{sec:predictive_characterization} highlights the distinctive role of Poisson, mixed Poisson, and mixed binomial process priors in shaping the dependencies of the induced predictive distributions on the observed sample. 
In this section, we specialize the general Bayesian analysis developed in Section \ref{sec:general_bayesian_analysis} to these notable classes of prior distributions.
Finally, we introduce a tractable model leading to dependence on the whole frequency spectrum.

\subsection{The Poisson process prior} \label{supp:Poisson}

Consider the model \eqref{eq:representation_theorem}, where $\mu$ is as in \eqref{eq:mu_definition} and $\Psi$ is a Poisson process on $\X \times (0,1]$ with an (infinite) mean measure $\nu$. This is one of the most popular cases in the literature, and the Indian buffet process arises as a specific example. 
Under the standard assumption that each observation $Z_i$ exhibits a finite number of features a.s., the $k$-th factorial moment measures are $\sigma$-finite, as shown in \Cref{prop:poisson_details}, point (i). The disintegration in \eqref{eq:mpsi_disinteg} writes as $\nu(\dd x \, \dd s) = \rho(\dd s \mid x) G_0(\dd x)$, where $\rho$ is a kernel from $\X$ to $(0,1]$ and $G_0$ is a $\sigma$-finite measure on $\X$. The Bayesian analysis of model \eqref{eq:representation_theorem} under a Poisson process prior for $\Psi$ can be addressed by specializing theorems of Section \ref{sec:bnp} as in the next result. Let us introduce some shorthand notations that will be useful throughout this section: $\rho_n(\dd s \mid x) = (1 - s)^n \rho(\dd s \mid x)$.

\begin{corollary}[Bayesian analysis under the Poisson process] \label{thm:james} 
Consider a sample $\bm Z$ from the statistical model \eqref{eq:representation_theorem}, where $\mu$ is the functional of a Poisson point process $\Psi$ with intensity measure $\nu$ defined via \eqref{eq:mu_definition}.  
    \begin{enumerate}
        \item[(i)] The marginal distribution of the sample $\bm Z$, equals
        \[
            e^{- \varphi_n } \prod_{\ell=1}^k \int_{(0,1]} (1 - s)^{n - m_\ell} s^{m_\ell} \rho(\dd s\mid x^*_\ell)  G_0(\dd x^*_\ell),
        \]
        where we defined
        \begin{equation}\label{eq:phi_def}
        \varphi_n 
        = \sum_{i=0}^{n-1}\int_{\X \times (0,1]} s  \rho_i(\dd s\mid x) G_0(\dd x).    
        \end{equation}
        \item[(ii)] The posterior distribution of $\mu$ satisfies the distributional equality in \eqref{eq:post_mu_sum},  where in this case $\mu^\prime$ is a \textsc{crm} with L\'evy intensity measure $\rho_n(\dd s\mid x) G_0(\dd x)$. The weights $S^*_\ell$'s of previously observed features are independent random variables, and independent of $\mu^\prime$, with marginal density $f_{S^*_\ell}(\dd s) \propto s^{m_\ell} (1 - s)^{n - m_\ell} \rho(\dd s\mid x^*_\ell)$, as $\ell=1, \ldots  , k$.
        \item[(iii)] The predictive distribution of a future observation $Z_{n+1}$, given the sample $\bm Z$, equals  the distribution of the measure defined in \eqref{eq:pred_sum},  where $Z^\prime_{n+1}$ is a  Poisson process on $\X$ with (finite) intensity measure given by $\int_{(0,1]} s \rho_n(\dd s\mid x) G_0(\dd x)$, and the
        $A^*_{n+1, \ell}$'s are independent Bernoulli random variables with parameters $S^*_\ell$'s, as $\ell=1, \ldots , k$.
    \end{enumerate}
\end{corollary}
The previous corollary is in agreement with the theory developed by \citet{Jam(17)} for the Poisson process case.
We first point out that the Poisson process $Z_{n+1}^\prime$ in point (iii) depends on the observable sample $\bm Z$ only through the sample size $n$, and this is not surprising by virtue of Theorem \ref{teo:pred_char}. 
We also observe that if  $\rho(\dd s\mid x) = \rho(\dd s)$ is independent of the location, then $G_0$ can be taken as a probability measure, as evident from \Cref{prop:poisson_details}, point (ii), and $\mu$ turns out to be a homogeneous \textsc{crm}. In this case,  $Z_{n+1}^\prime$ is a mixed binomial process $\mathrm{MB}(G_0, q_{K^\prime})$, where  $q_{K^\prime}$ is the probability mass function of the Poisson distribution with mean $\lambda_n = \int_{(0,1]} s \rho_n(\dd s)$, where $\rho_n(\dd s) = (1-s)^n \rho(\dd s)$. In other words, 
$Z^\prime_{n+1}$ has $K^\prime \sim \mathrm{Poi}(\lambda_n)$ points, represented as
$
Z^\prime_{n+1}= \sum_{\ell=1}^{K^\prime} \delta_{X^\prime_\ell}
$,
where the $X^\prime_\ell$'s are i.i.d. from $G_0$. Again, the distribution of $Z^\prime_{n+1}$ depends on the sample only through the sample size $n$.

\subsection{The mixed Poisson process prior} \label{supp:mixed_Poisson}

Let now $\Psi \sim \mathrm{MP}(\nu, f_\gamma)$, where $\nu$ is a (infinite) locally finite measure on $\X \times (0, 1]$. That is, $\Psi $ is a Poisson process with random mean measure $\gamma\nu$ and $\gamma \sim f_{\gamma}$ is a positive random variable.
As usual, we assume that $Z_i(\X) < \infty$ a.s.. Moreover, we suppose that $\gamma$ has finite moments of any order $k \geq 1$. Then, the $k$-th factorial moment measures are $\sigma$-finite for any $k$ (cf. \Cref{prop:poisson_details}) and $\nu(\dd x \, \dd s) = \rho(\dd s\mid x) G_0(\dd x)$, where $\rho$ is a kernel from $\X$ to $(0,1]$ and $G_0$ is a $\sigma$-finite measure on $\X$. 
The following corollary specifies the Bayesian analysis of this extended feature allocation model.

\begin{corollary}[Bayesian analysis under the mixed Poisson process] \label{thm:mixed_poisson_bayesian}
Consider a sample $\bm Z$ from the statistical model \eqref{eq:representation_theorem}, where $\mu$ is the functional of  a mixed Poisson point process $\Psi$, i.e.,  $\Psi \sim \mathrm{MP}(\nu, f_\gamma)$, defined via \eqref{eq:mu_definition}.
    \begin{enumerate}
        \item[(i)] The marginal distribution of $\bm Z$ equals
        \[
            \E\left(e^{-\gamma \varphi_n} \gamma^k\right)  \prod_{\ell=1}^k \int_{(0,1]} (1 - s)^{n - m_\ell} s^{m_\ell} \rho(\dd s\mid x^*_\ell)  G_0(\dd x^*_\ell),
        \]
        where $\varphi_n$ is defined as in \eqref{eq:phi_def}.
        \item[(ii)] The posterior distribution of $\mu$ satisfies the distributional equality in  \eqref{eq:post_mu_sum}, where  $\mu^{\prime} = \sum_{j \geq 1} S^\prime_j \delta_{X^\prime_j}$ and $\{(X_j^\prime, S_j^\prime)\}_{j \geq 1}$ are the points of a mixed Poisson point process $\Psi^\prime \sim  \mathrm{MP} (\nu^\prime, f_{\tilde\gamma})$, with 
\[
\nu^\prime (\dd x \, \dd s) =  \rho_n(\dd s\mid x) G_0(\dd x) \quad \text{and} \quad f_{\tilde\gamma} (\dd \gamma) \propto e^{-\gamma \varphi_n} \gamma^k f_\gamma(\dd \gamma).
\]
       In addition, the  $S^*_\ell$'s in \eqref{eq:post_mu_sum} are independent random variables, and independent of $\mu^\prime$, with marginal density $f_{S^*_\ell}(\dd s) \propto s^{m_\ell} (1 - s)^{n - m_\ell} \rho(\dd s\mid x^*_\ell)$, as $\ell=1, \ldots, k$.
        \item[(iii)] The predictive distribution of $Z_{n+1}$, given the sample $\bm Z$, satisfies the distributional equality in \eqref{eq:pred_sum}, where $Z^\prime_{n+1}$ is a mixed Poisson process on $\X$ with distribution
        \begin{equation} \label{eq:Z_prime_mixed_poisson}
        Z^\prime_{n+1} \sim \mathrm{MP} \Big(\int_{(0,1]} s \rho_n(\dd s\mid x) G_0(\dd x), f_{\tilde{\gamma}} \Big), 
        \end{equation}
        where the first parameter is a finite measure on $\X$. Moreover, the $A_{n+1, \ell}^*$'s in \eqref{eq:pred_sum} are independent Bernoulli random variables with parameters $S^*_\ell$'s, as $\ell=1, \ldots , k$. 
    \end{enumerate}
\end{corollary}
We emphasize that the process $Z^\prime_{n+1}$ in \eqref{eq:Z_prime_mixed_poisson} depends on the initial sample $\bm Z$ only through the sample size $n$ and the number of distinct features $k$, which appear in the mixing law $f_{\tilde{\gamma}}$, and on no additional sampling information, as expected from Theorem \ref{teo:pred_char}.
Finally, if the kernel $\rho$ does not depend on the location $x$, i.e.,  $\rho(\dd s\mid x) = \rho(\dd s)$, then $G_0$ can be taken as a probability measure (see \Cref{prop:poisson_details}, point (ii)) and $\mu$, conditionally on $\gamma$,  turns out to be a homogeneous \textsc{crm}. Thus, thanks to \Cref{lemma:MP_as_MB}, $Z_{n+1}^\prime$ is a mixed binomial process distributed as $Z_{n+1}^\prime\sim\mathrm{MB}(G_0, q_{K^\prime})$, where $q_{K^\prime}\mid \tilde\gamma$ is a Poisson density with parameter
$\tilde{\gamma}\int_{(0,1]} s \rho_n(\dd s)$, where $\rho_n(\dd s) = (1-s)^n \rho(\dd s)$, conditionally on a positive random variable $\tilde{\gamma}$ with probability distribution $f_{\tilde{\gamma}}$.
The stable beta scaled process in \cite{Cam(23)} and the mixtures of Indian buffet processes analyzed in \cite{ghilotti2024bayesian} fall under the umbrella of models considered in this section.

\subsubsection{A fresh look at scaled processes}

We now clarify the connection between a stable beta scaled process \citep{Cam(23)} and a mixed Poisson process, which has been only mentioned at the end of the last section. A scaled process is a random measure $\mu= \sum_{j \geq 1} S_j \delta_{X_j}$, where the $X_j$'s are i.i.d. random variables from a base measure $G_0$ on $\X$, while $(S_j)_{j \geq 1}$ arises as a suitable transformation of a Poisson process on the positive real line with intensity measure $\rho(\dd s)$. More precisely, to define the sequence of weights, one considers the jumps $(\Delta_j)_{j \geq 1}$ of a Poisson process in decreasing order and set $S_j^{\Delta_1}:= \Delta_{j+1}/\Delta_1$. Then, the distribution of $(S_j)_{j \geq 1}$ is obtained by mixing the conditional distribution of $(S_j^{\Delta_1})_{j \geq 1}$, given $\Delta_1$, with respect to a new distribution for the largest jump $\Delta_1$. See \cite{Cam(23)} for additional details.

Under a scaled process prior for the model \eqref{eq:representation_theorem}, the
distribution of $Z^\prime_{n+1}$ in \eqref{eq:pred_sum} may depend on $n$, $k$ and $\bm m$, and often involves intractable expression \citep[Proposition 2]{Cam(23)}. The notable tractable case is represented by scaled processes obtained from stable subordinators, i.e., $\rho(\dd s)= c s^{-1 - \alpha} \dd s$ for some constants $c > 0$ and $\alpha \in (0, 1)$, referred to as stable beta scaled processes.
In this case, $Z^\prime_{n+1}$ depends on the sample only through  $n$ and $k$.
However, \Cref{teo:pred_char} claims that such a dependence is retained if and only if $\Psi$ is a mixed Poisson or a mixed binomial process and, at a first glance,  the definition of scaled processes does not align with that of mixed Poisson processes.
This apparent inconsistency can be solved by resorting to an alternative construction of scaled processes that can be evinced either by   \citep[Theorem 1.1]{Jam(15)} or \citep[Lemma~1]{Cam(23)}.
\begin{proposition} \label{def:scaled}
    Let $\rho: \R_+ \rightarrow \R_+$ be the L\'evy intensity of a subordinator, $C$ a positive random variable, and $G_0$ a diffuse measure on $\X$. Let $\Psi = \sum_{j\geq 1}  \delta_{(X_j, S_j)}$ be such that 
    \[
        \Psi \mid  C \sim \mathrm{PP}\left( C \rho(  C s) \indicator_{(0,1]}(s) \dd s \, G_0(\dd x)\right).
    \]
    Then, $\mu= \sum_{j \geq 1} S_j \delta_{X_j}$ is a scaled process. 
\end{proposition}
The construction outlined in Proposition \ref{def:scaled}  is available for any scaled process, as introduced at the beginning of the section.
It becomes clear that $\Psi$ is a mixed Poisson process if and only if for any $s$ and any $C$, the following factorization holds  $C \rho(C s) = h(C) \tilde \rho(s)$ where $\tilde \rho(s)$ is a L\'evy density and $h: \R_+ \rightarrow \R_+$ is a function. When $\rho(s) = s^{-1 - \alpha}$, i.e., the case of stable subordinators, we get $C \rho(C s) = C^{-\alpha} s^{-1 -\alpha}$ so that in this case the scaled process $\mu$ is induced by a mixed Poisson process $\Psi$ directed by the law of $C^{-\alpha}$.
Finally, the predictive characterization \citep[Theorem 1]{Cam(23)} of stable beta scaled processes within the class of scaled processes is straightforward from Theorem \ref{teo:pred_char}.
Indeed, the L\'evy intensities of stable subordinators are the only $\rho$ such that $C \rho(C s) = h(C) \tilde \rho(s)$. Therefore, scaled processes obtained from stable subordinators are the only scaled processes induced by a mixed Poisson process  $\Psi$. Finally, \Cref{teo:pred_char} highlights how the dependence of the distribution of $Z_{n+1}^\prime$ solely on $n$ and $k$ is not a distinctive feature of stable beta scaled processes but pertains to the larger class of priors obtained from mixed Poisson and mixed binomial processes.

\subsection{The mixed binomial process prior} \label{sec:mixed_binomial}

We now consider the other class of priors identified in Theorem \ref{teo:pred_char}, i.e., mixed binomial processes.
Let $\Psi \sim \mathrm{MB}(\nu, q_M)$, where $q_M$ is a probability mass function on the nonnegative integers with finite moments and $\nu$ is a probability measure on $\X \times (0,1]$.
Therefore, the disintegration $\nu(\dd x \, \dd s) = \rho(\dd s \mid x) G_0(\dd x)$ always holds, where $G_0$ is a probability measure on $\X$, and $\rho$ is a probability kernel from $\X$ to $(0,1]$. This follows from a suitable application of \cite[Theorem 3.4]{Kallenberg21}. 
The Bayesian analysis of model \eqref{eq:representation_theorem} under the mixed binomial process prior for $\Psi$ is presented in the next result.  For ease of notation, we will use $\kappa_n$ to represent the integral $\kappa_n = \int_{\X \times (0, 1]} \rho_n(\dd s \mid x) G_0(\dd x)$.

\begin{corollary}[Bayesian analysis under the mixed binomial process] \label{thm:mixed_binom_bayesian}
Consider a sample $\bm Z$ from the statistical model \eqref{eq:representation_theorem}, where $\mu$ is the functional of a mixed binomial point process $\Psi$, i.e.,  $\Psi \sim \mathrm{MB}(\nu, q_M)$, defined via \eqref{eq:mu_definition}. 
    \begin{enumerate}
        \item[(i)] The marginal distribution of $\bm Z$ equals
        \[
             \E\left\{ (\kappa_n) ^{\tilde M}\right\} \E(M^{(k)}) \prod_{\ell=1}^k \int_{(0,1]} s^{m_\ell} (1 - s)^{n - m_\ell} \rho(\dd s \mid x^*_\ell) G_0(\dd x^*_\ell),
        \]
        where the first expected value is taken with respect to a nonnegative integer-valued random variable $\tilde M$ having probability mass function 
        \[
        q_{\tilde{M}}(m) \propto q_M(m+k) \frac{(m + k)!}{m!}.
        \]
        
        \item[(ii)] The posterior distribution of $\mu$ satisfies the distributional equality in  \eqref{eq:post_mu_sum}, where  $\mu^{\prime} = \sum_{j \geq 1} S^\prime_j \delta_{X^\prime_j}$ and $\{(X_j^\prime, S_j^\prime)\}_{j \geq 1}$ are the points of a mixed binomial process $\Psi^\prime \sim \mathrm{MB}  (\nu^\prime, q_{M^\prime})$  with 
        \[
        \nu^\prime (\dd s\, \dd x)=\rho_n(\dd s \mid x) G_0(\dd x) \quad \text{ and }  \quad q_{M^\prime}(m) \propto (\kappa_n)^m q_{\tilde{M}}(m).  
        \]
        Moreover, the $S^*_\ell$'s are independent random variables, and independent of $\mu^\prime$, with marginal density $f_{S^*_\ell}(\dd s) \propto s^{m_\ell} (1 - s)^{n - m_\ell} \rho(\dd s\mid x^*_\ell)$, as $\ell=1, \ldots , k$.

        \item[(iii)] The predictive distribution of a future observation $Z_{n+1}$, given the sample $\bm Z$, satisfies the distributional equality in \eqref{eq:pred_sum}, where $Z^\prime_{n+1}$ is 
        a mixed binomial point process on $\X$ with parameters $(\tilde G, q_{K^\prime})$, with 
        $\tilde G(\dd x) = \int_{(0,1]} s \rho_n(\dd s \mid x) G_0(\dd x)$ and
        \[
            q_{K^\prime}(m) \propto \sum_{z \geq m} \binom{z}{m} \kappa_n^z q_M(z+k) 
        c_p^{m} (1 - c_p)^{z-m} \frac{(z + k)!}{z!} ,
        \]
        having set 
        \[
            c_p = \frac{\int_{\X \times (0,1]} s \rho_n(\dd s \mid x)G_0(\dd x)}{ \int_{\X \times (0,1]} \rho_n(\dd s \mid x) G_0(\dd x)}. 
        \]
        Moreover, the $A_{n+1, \ell}^*$'s in \eqref{eq:pred_sum} are independent Bernoulli random variables with parameters $S^*_\ell$'s, as $\ell=1, \ldots, k$.
    \end{enumerate}
\end{corollary}

As expected, the distribution of $Z^\prime_{n+1}$ in point (iii) of the previous corollary depends on the sample only through the sample size $n$ and the number of observed features $k$. 
A special case of interest is obtained by assuming that  $\rho(\dd s\mid x) = \rho(\dd s)$, so that $\rho_n(\dd s \mid x) = \rho_n(\dd s) = (1 - s)^n \rho(\dd s)$.
Under this assumption, $Z^\prime_{n+1}$ is a mixed binomial process $\mathrm{MB}(G_0, q_{K^\prime})$ and $q_{K^\prime}$ does not depend on $G_0$. We further specialize this case to two choices of $M \sim q_M$.

If $M \sim \mathrm{Poi}(\lambda)$, then $\Psi$ is a Poisson process with finite intensity measure. In this case, $\Psi^\prime$ is a Poisson process with finite intensity $\lambda \rho_n(\dd s) G_0(\dd x)$, having $M^\prime \sim \mathrm{Poi}(\lambda \rho_n(0, 1])$  points. 
In addition, $Z^\prime_{n+1}$ is a Poisson process with finite intensity, having $K^\prime$ points, with $K^\prime$ being a Poisson random variable with mean $\lambda \int_{(0,1]} s \rho_n(\dd s)$.

Let now $M$ be a negative binomial random variables with parameters $(r, p)$, where $p \in [0,1]$ denotes the success probability and $r > 0$ is the number of successes. 
In this case, $\Psi^\prime$ is a mixed binomial process with $M^\prime$ points, where $M^\prime$ is a negative binomial random variable with updated parameters $r^\prime = r+k $ and $p^\prime= 1 - \kappa_n (1-p)$. Moreover, the process involving new features  $Z^\prime_{n+1}$ is a mixed binomial process with $K^\prime$ points, where ${K^\prime}$ is a negative binomial with parameters $ r^{th}=r+k$ and 
\[
    p^{th} = \frac{p^\prime}{p^\prime + (1 - p) \int_{(0,1]} s\rho_n(\dd s) }.
\]
We finally point out that the Gibbs-type feature models with finitely many features introduced in \cite{ghilotti2024bayesian} correspond to specific choices of the two models specified above.

\subsection{Predictions depending on the whole frequency spectrum} \label{supp:prediction_all}

In the examples analyzed so far, we have found predictive distributions for $Z^\prime_{n+1}$ depending either solely on $n$, or $n$ and $k$, or $n$ and $\bm x^*$ (including $k$).
As shown in \cite{Jam(15), Cam(23)} and as discussed at the end of Section \ref{sec:predictive_characterization}, in general, scaled processes yield predictive distributions for the newly discovered features which depend on the sampling information through $n, k$ and the frequency spectrum $m_1,\ldots,m_k$. 
However, outside the case of stable beta scaled processes, the resulting expressions for the posterior and marginal distributions (and the associated computations) are somewhat involved. For instance, if one considers scaled processes built from gamma and generalized gamma subordinators, the distribution of $\mu^\prime$ involves intractable integrals that must be evaluated numerically.

Here, we propose a different and straightforward strategy to induce predictive distribution depending on the whole frequency spectrum while maintaining computational convenience. 
The idea is simple and generally applicable: starting from the expression of the L\'evy intensity $\rho$ of a subordinator, assign a prior distribution to a parameter that does not enter only multiplicatively in the expression.
For simplicity and specificity, we will consider here the case of the three-parameter beta process \citep{Teh09, Bro(14)}, i.e., model  \eqref{eq:representation_theorem} with $\mu$ a homogeneous \textsc{crm} having L\'evy intensity measure
\begin{equation}\label{eq:sb_intensity}
     \rho(\dd s) \, G_0 (\dd x) = s^{-1 - \alpha} (1 - s)^{\beta + \alpha - 1}\indicator_{(0, 1]}(s) \dd s \, G_0(\dd x),
\end{equation}
where $\alpha \in (0, 1)$ and $\beta > - \alpha$. 
The next proposition shows how considering $\alpha$ random yields the desired predictive distribution.
\begin{proposition}\label{prop:frequency_dependence}
    Consider a sample $\bm Z$ from the statistical model \eqref{eq:representation_theorem}, where $\mu$ is the functional of a point process $\Psi$ defined via \eqref{eq:mu_definition} and $\Psi$ is such that 
    \begin{align*}
        \Psi \mid \alpha &\sim \mathrm{PP}(\gamma s^{-1 - \alpha} (1 - s)^{\beta + \alpha - 1}\indicator_{(0, 1]}(s) \dd s \, G_0(\dd x)), \\
        \alpha &\sim \pi_\alpha,
    \end{align*}
    where $\gamma > 0, \beta > 0$, $G_0$ is a diffuse probability measure on $\X$ and $\pi_\alpha$ is supported on $(0,1)$.
Then, if $\pi_\alpha$ is not a degenerate (Dirac) measure, the predictive distribution of $Z^\prime_{n+1}$, given the sample $\bm Z$,
    depends on $n$, $k$ and $m_1,\ldots, m_k$.
\end{proposition}

To the best of our knowledge, there is no conjugate prior for $\alpha$. However, inference under this model is straightforward by means of Markov chain Monte Carlo algorithms. 
Indeed, conditionally to $\alpha$, marginal, posterior, and predictive distributions are readily available, and the posterior distribution of $\alpha$ has a simple density function that is amenable to posterior simulation algorithms.

\section{Proofs of Section \ref{app:examples_bayesian_analysis}}\label{app:examples_computation}

\subsection{The Poisson process prior: proof of \Cref{thm:james}}

\noindent \emph{Proof of point (i) of \Cref{thm:james}}.
The marginal distribution of $\bm Z$ follows from specializing \Cref{thm:marg}. 
In particular, we first observe that $\tilde m_{\xi}^{(k)}(\dd \bm x^*) = \prod_{\ell=1}^k G_0(\dd x^*_\ell)$ and $\rho^{(k)}(\dd \bm s \mid \bm x^*) = \prod_{\ell=1}^k \rho(\dd s_\ell\mid x^*_\ell)$. 
Moreover, since $\Psi$ is a Poisson process, by \Cref{lem:poi}, it holds $\Psi^!_{\bm x^*, \bm s} \dequal \Psi$. Consequently, the expected value in the marginal expression of \Cref{thm:marg} equals
\begin{equation}\label{eq:ev_marginal_poisson}
\E\left\{ e^{\int_{\X \times (0, 1]} n \log(1 - t) \Psi^!_{\bm x^*, \bm s}(\dd z \, \dd t)} \right\} = \exp\left[ - \int_{\X \times (0, 1]} \{1 - (1 - s)^n\} \rho(\dd s\mid x) G_0(\dd x) \right].    
\end{equation}
It is easy to verify that $\int_{\X \times (0, 1]} \{1 - (1 - s)^n\} \rho(\dd s\mid x) G_0(\dd x)$ equals $\varphi_n$, where $\varphi_n$ is defined in \eqref{eq:phi_def}. The resulting marginal distribution recovers the marginal expression found in \cite[Proposition 3.1]{Jam(17)}.

\noindent \emph{Proof of point (ii) of \Cref{thm:james}}.
For the posterior distribution of $\mu$, expressed in \Cref{thm:post}, we need to determine the law of the vector $\bm S^* = (S^*_1,\ldots,S^*_k)$ and the law of $\mu^\prime$, conditionally to $\bm S^*$. From point (i) of \Cref{thm:post} and \eqref{eq:ev_marginal_poisson}, the $S^*_\ell$'s are independent with marginal laws $f_{S^*_\ell}(\dd s) \propto s^{m_\ell}(1 - s)^{n-m_\ell} \rho(\dd s\mid x^*_\ell)$. Moreover, since $\Psi^!_{\bm x^* , \bm s^*}$ is a Poisson point process with intensity $\rho(\dd s\mid x) G_0(\dd x)$, simple algebra applied to \eqref{eq:post_laplace_main} leads to recognizing
\[
        \E\left\{ e^{- \int f(z) \mu^\prime(\dd z)} \mid \bm S^* = \bm s^* \right\} = \exp\left\{ - \int_{\X \times (0, 1]} (1 - e^{-s f(x)})(1 - s)^n  \rho(\dd s\mid x) G_0(\dd x) \right\},
\]
that is the Laplace functional of a \textsc{crm} with L\'evy measure $(1 - s)^n \rho(\dd s\mid x) G_0(\dd x)$, thus $\Psi^\prime$ is a Poisson process with the same mean measure. Further note that $\mu^\prime$ is independent of $\bm S^*$.

\noindent \emph{Proof of point (iii) of \Cref{thm:james}}.
To obtain the predictive distribution of $Z^\prime_{n+1}$, we observed in \Cref{cor:pred} how $Z^\prime_{n+1}$ is obtained by first thinning the process $\Psi^\prime$ with retention probability $p(x, s) = s$, and then discarding the second component. 
Since $\Psi^\prime$ is a Poisson process with intensity $(1 - s)^n \rho(\dd s\mid x) G_0(\dd x)$, the thinned process is still Poisson with intensity $s(1 - s)^n \rho(\dd s\mid x) G_0(\dd x)$ and the resulting $Z^\prime_{n+1}$ is a Poisson process on $\X$ with intensity  $\int_{(0,1]} s(1 - s)^n \rho(\dd s\mid x) G_0(\dd x)$.

\subsection{The mixed Poisson process prior: proof of \Cref{thm:mixed_poisson_bayesian}}

\noindent \emph{Proof of point (i) of \Cref{thm:mixed_poisson_bayesian}}.
The marginal distribution of $\bm Z$ follows from exploiting that $\Psi\mid \gamma \sim \mathrm{PP}(\gamma\nu)$, for which the marginal law is given in point (i) of \Cref{thm:james}, and then integrating out the variable $\gamma$.

\noindent \emph{Proof of point (ii) of \Cref{thm:mixed_poisson_bayesian}}.
For the posterior distribution of $\mu$, expressed in \Cref{thm:post}, we need to determine the law of the vector $\bm S^* = (S^*_1,\ldots,S^*_k)$ and the law of $\mu^\prime$, conditionally to $\bm S^*$. To this end, it is convenient to exploit the disintegration of the law of $\Psi$ into $\Psi\mid \gamma$ and $\gamma$. First, as for the posterior distribution of $\Psi\mid \gamma$, since $\Psi\mid \gamma \sim \mathrm{PP}(\gamma \nu)$, we resort to point (ii) of \Cref{thm:james}. Specifically, such a posterior is equal in distribution to $\Psi^\prime + \sum_{\ell = 1}^k {\delta_{(x^*_\ell, S^*_\ell)}}$, where $\Psi^\prime\mid \gamma \sim \mathrm{PP}(\gamma (1-s)^n \nu(\dd x\, \dd s))$ and $S^*_\ell \mid \gamma$ are independent random variables, further independent of $\Psi^\prime\mid \gamma$, with laws $f_{S^*_\ell}(\dd s) \propto s^{m_\ell}(1 - s)^{n-m_\ell} \rho(\dd s\mid x^*_\ell)$. Since the laws of $S^*_\ell \mid \gamma$ do not depend on $\gamma$, then the $S^*_\ell$'s and $\Psi^\prime\mid \gamma$ are independent. Second, the posterior distribution of $\gamma$ is obtained from the likelihood in point (i) of \Cref{thm:james} and the prior $f_\gamma$, thus $\gamma\mid \bm Z \sim f_{\tilde{\gamma}}$, with $f_{\tilde{\gamma}}(\dd \gamma) \propto e^{-\gamma \varphi_n} \gamma^k f_\gamma(\dd \gamma)$. The thesis in point (ii) follows.

\noindent \emph{Proof of point (iii) of \Cref{thm:mixed_poisson_bayesian}}.
To describe the predictive distribution for $Z^\prime_{n+1}$ we proceed as follows: first consider the thinning of the process $\Psi^\prime$ with retention probability $p(x, s) = s$, and then discard the second component. 
Since $\Psi^\prime$ is a mixed Poisson process $\mathrm{MP}((1 - s)^n \nu(\dd x \, \dd s), f_{\tilde\gamma})$, the thinned process is still a mixed Poisson $\mathrm{MP}(s(1 - s)^n \nu(\dd x \, \dd s), f_{\tilde{\gamma}})$ and the resulting $Z^\prime_{n+1}$ is a mixed Poisson process (on $\X$) distributed as $\mathrm{MP}( \int_{(0,1]} s(1 - s)^n \nu(\dd x \, \dd s), f_{\tilde{\gamma}})$.

\subsection{The mixed binomial process prior: proof of \Cref{thm:mixed_binom_bayesian}}

\noindent \emph{Proof of point (i) of \Cref{thm:mixed_binom_bayesian}}.
The marginal distribution of $\bm Z$ is recovered from \Cref{thm:marg} as follows. 
First, it is straightforward to see that $\tilde m^{(k)}_{\xi}(\bm x) = \E(M^{(k)})  G_0^k(\bm x)$ and $\rho^{(k)}(\dd \bm s \mid \bm x) = \prod_{\ell=1}^k \rho(\dd s_\ell \mid x_\ell)$ satisfy \eqref{eq:mpsi_disinteg}. Second, from \Cref{prop:palm_mb}, we have that if $\Psi$ is a mixed binomial process $\mathrm{MB}(\nu, q_M)$, then $\Psi^!_{\bm x^* ,\bm s}$ is a mixed binomial process $\mathrm{MB}(\nu, q_{\tilde{M}})$ with $q_{\tilde{M}}(m) = q_M(m+k) (m + k)!/(\E(M^{(k)}) m!)$. Consequently, the expected value in the marginal expression of \Cref{thm:marg} equals
\begin{equation}\label{eq:ev_marginal_mixed_bin}
\begin{aligned}
    &\E\left\{ e^{ \int_{\X \times (0, 1]} n \log(1 - t) \Psi^!_{\bm x^*, \bm s}(\dd z \, \dd t)} \right\} = \E\left[ \exp\left\{  n \textstyle \sum_{j=1}^{\tilde{M}} \log(1 - \tilde Sj)\right\} \right] \\ 
        & = \E \left[\left\{\int_{\X\times (0,1]} (1 - s)^{n} \nu(\dd x \, \dd s)\right\}^{\tilde{M}}\right] = \mathcal{G}_{\tilde M}(\kappa_n),
\end{aligned}
\end{equation}
where $\tilde{M}$ has probability mass function $q_{\tilde{M}}$ and $(\tilde{X}_j, \tilde{S}_j) \iid \nu$, for $j=1,\ldots,\tilde{M}$; moreover, $\mathcal{G}_{\tilde M}(z) = \E(z^{\tilde M})$ is the probability-generating function of $\tilde M$ and $\kappa_n = \int_{\X\times (0,1]} (1 - s)^{n} \nu(\dd x \, \dd s)$.
It follows that the marginal distribution in \Cref{thm:marg} boils down to the expression in point (i) of the statement.

\noindent \emph{Proof of point (ii) of \Cref{thm:mixed_binom_bayesian}}.
For the posterior distribution of $\mu$, expressed in \Cref{thm:post}, we need to determine the law of the vector $\bm S^* = (S^*_1,\ldots,S^*_k)$ and the law of $\mu^\prime$, conditionally to $\bm S^*$. From point (i) of \Cref{thm:post} and \eqref{eq:ev_marginal_mixed_bin}, the $S^*_\ell$'s are independent with marginal laws $f_{S^*_\ell}(\dd s) \propto s^{m_\ell}(1 - s)^{n-m_\ell} \rho(\dd s\mid x^*_\ell)$, as for the Poisson case. Moreover, from point (ii) of \Cref{thm:post} we have that for any measurable function $g: \X\times (0,1] \to \R_+$,
\begin{equation*}
    \begin{aligned}
        \mathcal{L}_{\Psi^\prime}(g) &=  \mathcal{L}_{\Psi^!_{\bm x^*, \bm s^*}}(g(x,s) - n\log(1-s))/\mathcal{L}_{\Psi^!_{\bm x^*, \bm s^*}}(- n\log(1-s)) \\
        &= \mathcal{G}_{\tilde{M}}\left( \int_{\X\times (0,1]} \exp\left\{ - g(x,s) + n \log(1-s)\right\} \nu(\dd x \, \dd s) \right)/  \mathcal{G}_{\tilde{M}}(\kappa_n) \\
        &= \mathcal{G}_{\tilde{M}}\left( \int_{\X\times (0,1]} e^{- g(x,s)} (1-s)^n \nu(\dd x \, \dd s) \right) /\mathcal{G}_{\tilde{M}}(\kappa_n) \\
        &= \mathcal{G}_{M^\prime}\left( \int_{\X\times (0,1]} e^{- g(x,s)} (1-s)^n \nu(\dd x \, \dd s)/\kappa_n \right),
    \end{aligned}
\end{equation*}
where the second equality follows from the fact that $\Psi^!_{\bm x^* ,\bm s^*}$ is a mixed binomial process $\mathrm{MB}(\nu, q_{\tilde{M}})$; moreover, $M^\prime$ is a nonnegative integer-valued random variable with probability mass function $q_{M^\prime}$ with $q_{M^\prime}(m) \propto q_{\tilde{M}}(m) \kappa_n^m \propto  \kappa_n^m q_M(m+k) (m + k)!/m! $. Therefore, $\Psi^\prime$ is a mixed binomial process $\mathrm{MB}((1 - s)^n \nu(\dd x \, \dd s), q_{M^\prime})$ and it is independent of $\bm S^*$.

\noindent \emph{Proof of point (iii) of \Cref{thm:mixed_binom_bayesian}}.
To describe the predictive distribution of $Z_{n+1}$ remind that $Z^\prime_{n+1}$ is obtained by first thinning the process $\Psi^\prime$ with retention probability $p(x, s) = s$, and then discarding the second component, as described in \Cref{cor:pred}. By \Cref{prop:mbp_thin}, since $\Psi^\prime$ is a mixed binomial process, the thinned process is still mixed binomial, specifically $\mathrm{MB}(s(1-s)^n \nu(\dd x \, \dd s), q_{K^\prime})$, with 
\begin{equation*}
    \begin{aligned}
       q_{K^\prime}(m) &= \sum_{z \geq m} \binom{z}{m} q_{M^\prime}(z) c_p^{m} (1 - c_p)^{z-m}\\
        &\propto \sum_{z \geq m} \binom{z}{m} \kappa_n^z q_M(z+k) 
        c_p^{m} (1 - c_p)^{z-m} (z + k)!/z! ,
    \end{aligned}
\end{equation*}
where $c_p = \int_{\X\times (0,1]} s (1 -s)^n \nu(\dd x \, \dd s) / \int_{\X\times (0,1]} (1 -s)^n \nu(\dd x \, \dd s)$. Removing the second component of the retained points, we obtain $Z^\prime_{n+1}$, which is then a mixed binomial $\mathrm{MB}(\int_{(0,1]} s (1 -s)^n \nu(\dd x \, \dd s), q_{K^\prime})$.

\subsection{Proof of \Cref{prop:frequency_dependence}}
    Conditionally to $\alpha$, the posterior distribution of $\mu$ and the marginal law of the sample $\bm Z$ are as in \Cref{thm:james}. Moreover, the posterior law of $\alpha$ is given by
    \begin{equation}\label{eq:post_alpha}
         \pi_{\alpha \mid \bm Z}(\dd a) \propto \exp\{- \varphi_n(a) \} \prod_{\ell=1}^k B(m_\ell - a, n - m_\ell + \beta + a)  \times \pi_\alpha(\dd a)
    \end{equation}
    where $\varphi_n(a) = \gamma \int_{(0,1]} \{1 - (1 - s)^n\}  s^{-1 - a} (1 - s)^{\beta + a - 1}\dd s = \sum_{h=0}^{n-1} B(-a +1, h + \beta +a)$.

    From the posterior distribution of $\mu$ given $\alpha$ and the posterior density of $\alpha$ in \eqref{eq:post_alpha}, it is evident that the predictive distribution for the newly discovered features depends on the observed sample through $n$, $k$ and $m_1,\ldots, m_k$, while it does not depend on the labels $x^*_1,\ldots,x^*_k$.

\end{document}